%% file: lagtop.tex
\documentclass[12pt]{amsart}

\usepackage[colorlinks=true,urlcolor=blue,
bookmarks=true,bookmarksopen=true,citecolor=blue
]{hyperref}

\usepackage{pdfsync}
\usepackage{graphicx}

\usepackage{epsfig}
\usepackage{amscd}
\usepackage[mathscr]{eucal}
\usepackage{amssymb}
\usepackage{amsxtra}
\usepackage{amsmath}
\usepackage[all]{xy}
\usepackage{enumerate}
\usepackage{mathrsfs}

\usepackage{color}
\theoremstyle{plain}

\newtheorem{thm}{Theorem}[subsection]

\newtheorem{cor}[thm]{Corollary}

\newtheorem{lem}[thm]{Lemma}
\newtheorem{prop}[thm]{Proposition}

\theoremstyle{definition}
\newtheorem{dfn}[thm]{Definition}

\theoremstyle{remark}
\newtheorem{rem}[thm]{Remark}

\newtheorem{rems}[thm]{Remarks}

\newtheorem{ex}[thm]{Example}

\theoremstyle{plain}

\newcommand{\Qed}{\hfill \qedsymbol \medskip}

\newcommand{\hooklongrightarrow}{\lhook\joinrel\longrightarrow}

\newcommand{\R}{\mathbb{R}}
\newcommand{\Z}{\mathbb{Z}}

\newcommand{\C}{\mathbb{C}}
\newcommand{\la}{\lambda}
\newcommand{\La}{\Lambda}

\newcommand{\Crit}{\textnormal{Crit\/}}

\newcommand{\leftexp}[2]{{\vphantom{#2}}^{#1}{#2}}
\newcommand{\lrsub}[3]{{\vphantom{#2}}_{#1}{#2}_{#3}}

\newcommand{\pbaddress}{biran@math.ethz.ch}
\newcommand{\ocaddress}{cornea@dms.umontreal.ca}

\begin{document}

\title{Lagrangian topology and enumerative geometry} \date{\today}

\thanks{The second author was supported by an NSERC Discovery grant
  and a FQRNT Group Research grant}


\author{Paul Biran and Octav Cornea}

\address{Paul Biran, Department of Mathematics, ETH-Z\"{u}rich,
  R\"{a}mistrasse 101, 8092 Z\"{u}rich, Switzerland}
\email{\pbaddress} \address{Octav Cornea, Department of Mathematics
  and Statistics University of Montreal C.P. 6128 Succ.  Centre-Ville
  Montreal, QC H3C 3J7, Canada} \email{\ocaddress}

\bibliographystyle{alphanum}

%
%

\maketitle

%
%

\tableofcontents

\input{intro.tex}
\input{setting.tex}

\input{wide-var.tex}

\input{qforms.tex}
\input{defo.tex}

\input{enum.tex}
\input{enum-mod-2.tex}
\input{toric.tex}

\input{examples.tex}
\appendix
\input{orientations.tex}

\bibliography{bibliography}

%
%
%

\end{document}

%% file: intro.tex
\section{Introduction}\label{sec:intro}

The main motivation for this paper is the search for enumerative
invariants for Lagrangian submanifolds.  One of the simplest
fundamental questions in this topic can be formulated as follows.  Fix
a closed, connected Lagrangian submanifold $L^{n}$ inside some
symplectic manifold $ (M^{2n},\omega)$. Fix an almost complex
structure $J$ on $M$ that is compatible with $\omega$ and let
$P,Q,R\in L$ be three distinct points.

\subsubsection*{Problem} Estimate the number $n_{PQR}(L,J)$ of disks
$u:(D^{2}, \partial D^{2})\to (M,L)$ that are $J$-holomorphic (in the
sense that $\bar{\partial}_{J}(u)=0$ \cite{McD-Sa:Jhol-2}) and that go
through $P,Q,R$ in this order.

It is easily seen that for this question to make sense one should
restrict to generic almost complex structures $J$ and, to ensure that
the number in question is finite, we have to consider only those disks
$u$ belonging to homotopy classes $\la\in \pi_{2}(M,L)$ so that the
Maslov index of $\la$, $\mu(\la)$, equals $2n$. The count providing
the number $n_{PQR}(L,J)\in \mathbb{Z}$ takes into account appropriate
orientations. Ideally, one would like to obtain more refined estimates
by evaluating the numbers $n_{PQR}(L,J;\la)$ of disks $u$ as above
that belong to each specific class $\la$.

\subsection{An enumerative invariant}   We work in this paper under the restriction that $L\subset M$ is a monotone
Lagrangian, oriented and endowed with a fixed spin structure.

\

It is easy to see that the numbers
$n_{PQR}(L,J)$ above are in general not invariant, they depend on the
choice of the points $P,Q,R$ as well as on $J$. Thus, it is natural to investigate
whether this lack of invariance can possibly be compensated by some more
complicated enumerative ``counts''.

The origin of the present paper  lies precisely in such a formula
(closely related to expressions first detected in our paper \cite{Bi-Co:qrel-long}).

Assume that $L$ is the $2$-torus $\mathbb{T}^{2}$. Fix a triangle
$PQR$ on the torus. By this we mean three distinct points $P,Q, R\in
L$ together with a smooth oriented path $\overrightarrow{PQ}$ starting
from $P$ and ending at $Q$ as well as similar paths connecting $Q$ to
$R$ and $R$ to $P$.  Fix also a generic almost complex structure $J$.
Let $n_{P}$ be the number of $J$-holomorphic disks of Maslov index $2$
that go through $P$ and cross transversely the edge
$\overrightarrow{QR}$ (this number takes into account orientations -
it is defined with more precision in \ref{subsec:2-torus}). Define
similarly the numbers $n_{Q}$ and $n_{R}$.

We will see that if the Floer homology
   $HF(L,L)\not=0$, then the expression
   \begin{equation}\label{eq:numerical-torus}\Delta= 4 n_{PQR} + n_P^2 + n_Q^2 +
   n_R^2 - 2n_P n_Q - 2n_Q n_R - 2n_R n_P
   \end{equation} is independent of the
   triangle $P,Q,R$ as well as of $J$.

\subsection{Formula (\ref{eq:numerical-torus}): its meaning and generalizations}
In  the paper we investigate the invariant $\Delta$ and the meaning of 
formula (\ref{eq:numerical-torus}), besides, of course, proving this formula. We also  provide a more
general and conceptual perspective on other  enumerative expressions in arbitrary dimensions.
To summarize, we will see that:
\begin{itemize}
\item[-] $\Delta$ coincides with the discriminant of a certain quadratic form that can be read off the
quantum homology product of $L$.
\item[-] $\Delta$ is actually the unique (symmetric) polynomial, enumerative  invariant  that can be extracted from the
    quantum product. Interestingly, this uniqueness is a consequence of the classification of polynomial
     invariants associated to quadratic
    forms as in \cite{Hilbert}.
\item[-] $\Delta$ and/or other invariants like it, as well as formulae like (\ref{eq:numerical-torus}) also exist for more general Lagrangians and in arbitrarily high dimensions.
\item[-] There are refinements of these formulae that take into
 account the specific homotopy classes $\la\in\pi_{2}(M,L)$.  They
 allow for these invariants to be written as expressions with
 coefficients in the ring of regular functions $\mathcal{R}$ of
 certain algebraic subvarieties of the variety of representations
 $\pi_{2}(M,L)\to \C^{\ast}$.
\item[-]  In the case of toric fibres the ring of representation point of view is particularly useful as it relates $\Delta$
to the quantum Euler class of the ambient manifold.
\end{itemize}

In a number of examples, we also compute the relevant invariants explicitly over $\mathcal{R}$.
Some remarkable numerical identities follow.

\subsection{Structure of the paper} We now describe more thoroughly our approach
and the structure of the paper.

In \S\ref{s:setting} we summarize the main properties of
Lagrangian quantum homology of $L$, $QH(L)$, as described in
\cite{Bi-Co:rigidity} together with a number of its algebraic
properties.  In particular, we recall that $QH(L)$ is a ring - we will
denote the respective multiplication by $\ast$. We also fix a few basic
orientation conventions. To avoid disrupting  the natural flow of the paper
a complete and more technical discussion of orientations is postponed to
the Appendix \ref{a:orientations}.

\

In \S\ref{s:wide} we consider the representation variety
$$\mathcal{R}ep(L)=\{\rho : \pi_{2}(M,L)\to \C^{\ast}\}~.~$$
We show that for a certain algebraic subset $\mathcal{W}\subset
\mathcal{R}ep (L)$ the regular functions on $\mathcal{W}$,
$\mathcal{O}(\mathcal{W})$, can be used as coefficient ring for
quantum homology with the effect that the resulting object
$Q^{+}H(L;\mathcal{W})$ is isomorphic as a vector space to the
singular homology of $L$ taken with the appropriate coefficients. The
$+$ in $Q^{+}H$ reflects the fact that quantum homology as constructed
in \cite{Bi-Co:rigidity} has some strong positivity features due to
the fact that the various quantum structures are defined by using
unperturbed $J$-holomorphic objects. A key consequence of positivity
is that the algebra $Q^{+}H(L;\mathcal{W})$ is a deformation of
singular homology - viewed as algebra with the intersection product.

\

In \S\ref{s:qforms} we make use of this setting to define a quadratic
form associated to the quantum product with coefficients in
$\mathcal{O}(\mathcal{W})$ and its associated discriminant $\Delta$.
In the case of the $2$-torus this will later be seen to be precisely
the term on left hand side of (\ref{eq:numerical-torus}).

\

Section~\ref{s:defo} is based on the remark that the isomorphism
between $Q^{+}H(L;\mathcal{W})$ and singular homology that was
mentioned above {\em is not} canonical.  In particular, if a specific
isomorphism between quantum homology and singular homology is used to
expand the quantum product with respect to a singular basis,
$\mathbf{a}=(a_{1},\ldots a_{i},\ldots)$, as
\begin{equation}\label{eq:struct_cst0}
   a_{i}\ast  a_{j}=\sum_{s} k^{i,j}_{s}a_{s}t^{\epsilon'(i,j,s)}~,~
\end{equation}
then the resulting structural constants $k^{i,j}_{s}$ are not
invariants -- they depend on $J$ as well as on the other data used to
define the various structures involved (here $\epsilon'(i,j,s)$ are
appropriate integers -- see \S\ref{subsubsec:struct-const} ; $t$ is a
formal deformation variable used in the definition of the quantum
homology $Q^{+}H(L)$).  Notice that this lack of invariance of the
$k^{i,j}_{s}$'s is in marked contrast with the closed case where the
same type of expansion of the classical quantum product produces
structural constants that are identified with triple Gromov-Witten
invariants.

On the other hand, the deformation equivalence class of
$Q^{+}H(L;\mathcal{W}_{2})$ (as deformation of the singular
intersection algebra) {\em is invariant}. Thus, in searching for
invariant enumerative expressions, it is natural to look for
polynomial invariants in the $k^{i,j}_{s}$'s that only depend on this
equivalence class. This type of invariants is introduced in
\S\ref{s:defo} and most of the section is spent discussing them from a
variety of points of view.  It is also noticed that $\Delta$, as
defined in \S\ref{s:qforms}, is a particular such invariant.
Conceptually, one way to view this is by the prism of Hochschild
cohomology.  Indeed, this cohomology classifies algebra deformations
and we notice that there is a natural map that associates to specific
Hochschild cohomology classes (of the correct degrees) equivalence
classes of quadratic forms. As we will see, $\Delta$ is simply the
associated discriminant for these forms.

 \

In~\S\ref{s:enum} we start by revisiting formula (\ref{eq:numerical-torus}) from a
related but slightly different perspective. It turns out that $\Delta$, as defined
in \S\ref{s:qforms} only depends on counts of $J$-holomorphic disks of
Maslov class $2$. Thus  formula (\ref{eq:numerical-torus}) can be viewed as a {\em splitting formula}
expressing counts of Maslov $4$ disks in terms of counts  of configurations involving only
 Maslov $2$ disks. The first part of \S\ref{s:enum} contains a general definition of such splitting formulae
 and a proof that   they exist for  Lagrangians of arbitrary dimensions (under very mild assumptions).
We also notice that, as illustrated by formula (\ref{eq:numerical-torus}) there is a close
relationship between the invariant polynomials described in \S\ref{s:defo} and these splitting formulae.
The second part of \S\ref{s:enum} contains the proof of  (a more general version) of (\ref{eq:numerical-torus}).

\

As mentioned before, the role of the discriminant $\Delta$ is central
in our study especially for Lagrangian tori. In view of this,
in~\S\ref{s:toric} we focus on a variety of further properties for
Lagrangian tori that appear as fibres of the moment map in toric
manifolds. An extensive study of Floer theory of such tori has been
carried out by several authors, e.g.~\cite{Cho:Clifford,
  Cho-Oh:Floer-toric, Cho:products, FO3:toric-1, FO3:toric-2,
  Aur:t-duality, Aur:slag-1, Aur:slag-2}.  We build on these works and
exemplify our theory on the case of toric fibres. In particular, we
describe a relation between our machinery and the Frobenius structure
on the quantum homology of the ambient toric manifold and see that the
quantum Euler class viewed in the appropriate context can be
identified with our discriminant $\Delta$. Finally,~\S\ref{s:exp}
contains a series of explicit computations mostly for toric fibers.

\subsubsection*{Acknowledgments.}
We would like to thank Sasha Givental for explaining to us his
perspective on the Frobenius structure of the quantum homology of
toric manifolds and the reference to his paper~\cite{Gi:elliptic-GW}.
We would also like to thank Denis Auroux for useful discussions
related to the example in~\S\ref{sb:chekanov-torus}. Thanks to Andrew
Granville for useful discussions and for explaining us an elementary
and beautiful approach to verify the arithmetic identities
in~\S\ref{s:toric}. Finally, we would like to thank Jean-Yves
Welschinger for interesting and insightful discussions related to
enumerative invariants.

%% file: setting.tex
\section{Setting} \label{s:setting}

All our symplectic manifolds will be implicitly assumed to   be
connected and tame (see~\cite{ALP}).  The main examples of such
manifolds are closed symplectic manifolds, manifolds which are
symplectically convex at infinity as well as products of such. We
denote by $\mathcal{J}$ the space of $\omega$--compatible almost
complex structures on $M$ for which $(M, g_{\omega,J})$ is
geometrically bounded, where $g_{\omega,J}$ is the associated
Riemannian metric.

Lagrangian submanifolds $L \subset (M,\omega)$ will be assumed to be
connected and closed. Write $H_2^D = H_2^D(M,L) = \textnormal{image\,}
(\pi_2(M,L) \longrightarrow H_2(M,L))$ for the image of the Hurewicz
homomorphism.  Denote by $\mu:H_2^D \longrightarrow \mathbb{Z}$ the
Maslov index and by $N_L = \min \{ \mu(A) \mid \mu(A)>0 \}$ the
minimal Maslov number, so that $\mu(H_2^D) = N_L \mathbb{Z}$. Since
Maslov numbers come in multiples of $N_L$ we put $\bar{\mu} : =
\frac{1}{N_L} \mu$.

Denote by $\omega: H_2^D \longrightarrow \mathbb{R}$ the homomorphism
induced by integration of $\omega$. We will mostly assume that our
Lagrangians are {\em monotone}, that is there exists a constant
$\tau>0$ such that
\begin{equation}\label{eq:monotonicity}
   \omega(A)=\tau \mu(A),\ \forall \ A \in
   H^D_{2}(M,L),
\end{equation}
and moreover that $N_L \geq 2$.

\subsection{Coefficient rings} \label{sb:coefs} Our ground ring will
be denoted by $K$. We will mostly take $K$ to be either $\mathbb{C}$,
$\mathbb{Q}$ or $\mathbb{Z}$ and sometimes $\mathbb{Z}_2$. In case $K
\neq \mathbb{Z}_2$ we implicitly assume that our Lagrangian $L$ is
orientable and spin. Moreover we fix an orientation and a spin
structure on $L$.

The following rings will be used frequently in the sequel: $\Lambda =
K[t^{-1}, t]$, $\Lambda^+ = K[t]$. We grade these rings by setting
$|t| = -N_L$. Next consider the group ring $K[H_2^D]$ whose elements
we write as ``polynomials'' in the variable $T$, i.e.  $P(T) = \sum_{A
  \in H_2^D} a_A T^A$, with $a_A \in K$. We grade this ring by setting
$|T^{A}| = -\mu(A)$.

The most important ring for our considerations will be
$\widetilde{\Lambda}^{+}$ which is defined as
$$\widetilde{\Lambda}^{+} = \{ P(T) \in
K[H_2^D] \mid P(T) = a_0 + \sum_{A, \mu(A)>0} a_A T^A \}.$$ Note that
the degree $0$ component of $\widetilde{\Lambda}^+$ is just $K$ (i.e.
constants) while that of $K[H_2^D]$ is the whole of $K[\ker \mu]$.
We denote by $\widetilde{\La}^{>0}$ the elements of degree strictly
bigger than zero in $\widetilde{\La}^{+}$.

In what follows we will work with $\widetilde{\Lambda}^+$--algebras.
By this we mean commutative, graded rings $\mathcal{R}$ which are also
graded algebras over $\widetilde{\Lambda}^+$. This structure is
typically specified by a graded morphism of rings
$\widetilde{\Lambda}^{+} \longrightarrow \mathcal{R}$.

\subsection{Lagrangian quantum homology and quantum structures}
\label{sb:qh}

The pearl complex, Lagrangian quantum homology and its associated
quantum structures have been described in detail
in~\cite{Bi-Co:Yasha-fest, Bi-Co:rigidity, Bi-Co:qrel-long}. We refer
the reader to these papers for the detailed constructions. Here we
just set up the notation and recall the main properties of this
homology. In addition, we explain how to carry out the construction
over general ground rings $K$, other that $\mathbb{Z}_2$. This
requires orienting the moduli spaces of pearly trajectories and is
explained in detail in the appendix~\S\ref{a:orientations}.

Let $\mathcal{R}$ be an $\widetilde{\Lambda}^+$--algebra.  Fix a
triple $\mathscr{D} = (f, (\cdot,\cdot), J)$ where $f:L
\longrightarrow \mathbb{R}$ is a Morse function $(\cdot, \cdot)$ is a
Riemannian metric on $L$ and $J$ an $\omega$--compatible almost
complex structure on $M$.  Denote by
$$\mathcal{C}(\mathscr{D}) = K \langle \textnormal{Crit} f \rangle
\otimes \mathcal{R}, \quad d:\mathcal{C}_*(\mathscr{D}) \longrightarrow
\mathcal{C}_{*-1}(\mathscr{D})$$ the pearl complex with coefficients
in $\mathcal{R}$.  This complex is defined for generic $\mathscr{D}$,
its  homology does not depend on
$\mathscr{D}$ and is denoted by $QH(L; \mathcal{R})$.

\subsubsection{Product}\label{sbsb:prod-setting}
Recall that $QH(L;\mathcal{R})$ has the
structure of an associative (but not necessarily commutative) ring
with unity:
\begin{equation} \label{eq:qprod} QH_i(L;\mathcal{R})
   \otimes_{\mathcal{R}} QH_j(L; \mathcal{R}) \longrightarrow
   QH_{i+j-n}(L;\mathcal{R}), \quad \alpha \otimes \beta \longmapsto
   \alpha*\beta,
\end{equation}
where $n = dim L$. The unity lies in $QH_n(L;\mathcal{R})$ and is
denoted by $[L]$ (in analogy to the fundamental class in singular
homology).

\subsubsection{Module structure} \label{subsubsec:module} Denote by
$QH(M;
\mathcal{R})$ the quantum homology of the (ambient) symplectic
manifold $(M,\omega)$ endowed with the quantum product $*$. The
extension of the coefficients to $\mathcal{R}$ is induced by the
composition of the natural maps $\pi_2(M) \longrightarrow \pi_2(M,L)
\longrightarrow H_2^D(M,L)$, see~\cite{Bi-Co:rigidity} for details.
Then $QH(L;\mathcal{R})$ becomes an algebra over $QH(M;\mathcal{R})$
in the sense that there exists a canonical map
\begin{equation} \label{eq:qmod}
   QH_i(M;\mathcal{R})
   \otimes_{\mathcal{R}} QH_j(L;\mathcal{R}) \longrightarrow
   QH_{i+j-2n}(L;\mathcal{R}), \quad a \otimes \alpha \longmapsto
   a*\alpha,
\end{equation}
which turns $QH(L;\mathcal{R})$ into an algebra over the
ring $QH(M;\mathcal{R})$.

\subsubsection{Inclusion} \label{subsubsec:inclu} We also have a
quantum version of the map induced in homology by the inclusion $L
\longrightarrow M$. This is a map
\begin{equation} \label{eq:qinc}
   i_L :
   QH_*(L;\mathcal{R}) \longrightarrow QH_*(M;\mathcal{R})
\end{equation}
which extends the classical inclusion on the chain level. The map
$i_L$ is a $QH(M;\mathcal{R})$--module morphism.

\subsubsection{Minimal models} \label{subsubsec:minimal-mod} It is
important throughout the paper that all the structures above are
defined over $\widetilde{\La}^{+}$ and that, at the chain level, they
are deformations of the respective Morse-theoretic structures. The
Morse theoretic structures (on the chain level) are obtained from the
ones defined above by specializing to $t=0$. For an algebraic
structure defined over $V\otimes \La^{+}$ where $V$ is some $K$-vector
space we will refer to the algebraic object obtained by specializing
to $t=0$ as the ``Morse level'' or ``classical'' associated structure.

A very useful consequence of positivity is the existence of minimal
models whose definition and properties we now recall.

If $f$ is a perfect Morse function, in the sense that the differential
of its Morse complex is trivial, then the pearl complex is quite
efficient for computations. However, not all manifolds admit perfect
Morse functions. The existence of the minimal models allows to reduce
algebraically the pearl complex to such a minimal form whenever the
base ring $K$ is a field.  We recall the relevant result from
\cite{Bi-Co:rigidity}.

\begin{prop} \cite{Bi-Co:rigidity} \label{prop:min_pearls} Let $K$ be
   a field. For any monotone Lagrangian $L$ there exists a complex
   $\mathcal{C}_{min}(L)=(H_{\ast}(L;K)\otimes \widetilde{
     \La}^{+},\delta)$, with
  $$\delta : H_{\ast}(L;K)\otimes
  \widetilde{ \La}^{+}\to H_{\ast}(L;K)\otimes \widetilde{ \La}^{>0}$$
  so that, for any triple $\mathscr{D}=(f, (\cdot,\cdot), J)$ such
  that $\mathcal{C}(\mathscr{D})$ is defined, there are chain
  morphisms $\phi:\mathcal{C}(\mathscr{D})\to \mathcal{C}_{min}(L)$
  and $\psi : \mathcal{C}_{min}(L)\to \mathcal{C}(\mathscr{D})$ that
  both induce isomorphisms in quantum homology as well as in Morse
  homology and verify $\phi\circ\psi=id$.  The complex
  $\mathcal{C}_{min}(L)$ with these properties is unique up to (a
  generally non-canonical) isomorphism and is called the minimal pearl
  complex of $L$.  The maps $\psi$, $\phi$ are called structural maps
  associated to $\mathscr{D}$.
\end{prop}

All the algebraic structures described before (product, module
structure etc.) can be transported and computed on the minimal
complex. For instance, the product is the composition:
\begin{equation}\label{eq:prod_min}
   \mathcal{C}_{min}(L)\otimes
   \mathcal{C}_{min}(L)\stackrel{\psi_{1}\otimes\psi_{2}}
   {\longrightarrow}\mathcal{C}(\mathscr{D}_{1})\otimes
   \mathcal{C}(\mathscr{D}_{2})\stackrel{\ast}{\longrightarrow}
   \mathcal{C}(\mathscr{D}_{3})\stackrel{\phi_{3}}{\longrightarrow}
   \mathcal{C}_{min}(L)
\end{equation}

where $\phi_{i},\psi_{i}$ are structural maps associated to the data
set $\mathscr{D}_{i}$.

\begin{rem}\label{rem:min_mod} If the Lagrangian $L$ admits perfect
   Morse functions, then any pearl complex associated to such a
   function is a minimal pearl complex over any ring $K$ (not only
   when $K$ is a field).  Moreover, any two such minimal models are
   related by canonical comparison maps. This means for instance that
   for tori we may choose to work over $\Z$.
\end{rem}

\subsection{Additional conventions} \label{sb:additional-conv}
\subsubsection{Orientations of the pearly moduli spaces}
\label{sbsb:orient-pearly-1} In order to define the pearl complex over
a general ground ring we need to orient the moduli space of pearl
trajectories. These are a combination of moduli spaces of gradient
trajectories arising from Morse theory together with moduli spaces of
$J$-holomorphic disks. The precise orientation conventions
are described in detail in the
appendix~\S\ref{a:orientations}, we only mention here some of the
very basic choices used later in the paper.

Throughout the paper, by a Lagrangian $L\subset (M,\omega)$ we mean an
oriented Lagrangian submanifold together with a fixed spin structure.

Denote by $D \subset \mathbb{C}$ the closed unit disk. We orient its
boundary $\partial D$ by the counterclockwise orientation. Denote by
$G = Aut(D)$ the group of biholomorphisms of the disk, and by $H
\subset G$ the subgroup of elements that preserve the two points $-1,
+1 \in \partial D$. We orient both $G$ and $H$ as described
in~\S\ref{sbsb:aut-D}.

Fix a generic almost complex structure $J \in \mathcal{J}$. Let $B \in
H_2^D$. Denote by $\widetilde{\mathcal{M}}(B,J)$ the space of
(parametrized) $J$-holomorphic disks $u:(D, \partial D)
\longrightarrow (M,L)$ with $u_*([D]) = B$. It is well-known by the
work~\cite{FO3} that a spin structure on L induces orientations on the
moduli spaces $\widetilde{\mathcal{M}}(B,J)$. The groups $G$ and $H$
act on $\widetilde{M}(B,J)$ by $\sigma \cdot u = u \circ
\sigma^{-1}$, and similarly on $\widetilde{M}(B,J) \times \partial D$
by $\sigma \circ (u,z) = (u \circ \sigma^{-1}, \sigma(z)$. The
following spaces will play an important role in the sequel:
$$\mathcal{M}_2(B,J) = \widetilde{M}(B,J)/H,
\qquad (\widetilde{M}(B,J) \times \partial D)/G.$$ Both spaces come
with an orientation induced from those of $\widetilde{M}(B,J)$ and of
$G$ and $H$ as described in~\S\ref{sbsb:group-actions}
and~\S\ref{sbsb:aut-D}. There are natural
evaluation maps that will be used frequently in the sequel:
\begin{equation} \label{eq:ev-I}
   \begin{aligned}
      & e_{\scriptscriptstyle -1}: \mathcal{M}_2(B,J) \longrightarrow
      L, \quad [u]
      \longmapsto u(-1), \\
      & e_{\scriptscriptstyle +1}: \mathcal{M}_2(B,J) \longrightarrow
      L, \quad [u]
      \longmapsto u(+1), \\
      & ev : (\widetilde{\mathcal{M}}(B,J) \times \partial D)/G
      \longrightarrow L, \quad (u,z) \longmapsto u(z).
   \end{aligned}
\end{equation}
See~\S\ref{sbsb:moduli-of-disks} for more details concerning the
orientations of these spaces.

\subsubsection{The intersection product}
\label{sbsb:intersection-prod}
In the sequel we will use a version of the classical intersection
product on singular homology which we denote by:
$$H_i(L) \otimes H_j(L) \longrightarrow H_{i+j-n}(L),
\quad a \otimes b \longmapsto a \cdot b.$$ We remark here that our
convention for this operation is somewhat non-standard concerning
signs and orientations. Our intersection product is characterized by
the following property: if $a=[A]$, $b=[B]$, where $A, B \subset L$
are two transverse oriented submanifolds, then $a \cdot b = [B \cap
A]$, ({\em not $A \cap B$ !}), where $\cap$ stands for oriented
intersection (see~\S\ref{sbsb:intersections}). When $a$ and $b$ have
complementary dimensions we will also use their {\em intersection
  number} which we denote $\#(a \cap b) = \# (A \cap B)$. (Thus in
this case $a \cdot b = \#(B \cap A)[\textnormal{pt}] =
(-1)^{(n-i)(n-j)} \# (A \cap B)[\textnormal{pt}]$, where $n=\dim L$,
$i=\dim A$, $j=\dim B$.) Also by abuse of notation, when $i+j=n$ we
will sometimes view $(- \cdot -)$ as a $\mathbb{Z}$-valued pairing and
write $a \cdot b \in \mathbb{Z}$, instead of $a \cdot b \in
H_0(L)=\mathbb{Z} [\textnormal{pt}]$.

In favorable situations the product mentioned
in~\S\ref{sbsb:prod-setting} can be considered as a deformation of the
above version of the classical intersection product on the singular
homology. (See~\S\ref{sbsb:or-prod}.) The signs defining this product
were so chosen in order to make duality more natural
(see~\S\ref{sbsb:or-duality}).

Analogous remarks apply also to the module structure
from~\S\ref{subsubsec:module} and~\S\ref{sbsb:or-mod} (both the
classical and the quantum operations).

\subsection{Twisted coefficients} \label{sb:tc}

The most relevant ground ring here will be $K = \mathbb{C}$, though
one could work with $K = \mathbb{Q}$ or $K=\mathbb{Z}$ too.

Let $\rho: H_2^D \longrightarrow \mathbb{C}^*$ be a homomorphism. This
induces a structure of a $\widetilde{\Lambda}^+$--algebra on the ring
$\Lambda^{+} = K[t]$ induced by the morphism $\widetilde{\Lambda}^{+}
\longrightarrow \Lambda^{+}$ defined by
\begin{equation} \label{eq:algebra-rho} \widetilde{\Lambda}^+ \ni
   T^{A} \longmapsto \rho(A) t^{\bar{\mu}(A)} \in \Lambda, \quad
   \forall A \in H_2^D.
\end{equation}
In order to emphasize the dependence on $\rho$ in this algebra
structure we will sometimes write $(\Lambda^{\rho})^{+}$ rather than
$\Lambda^+$. Similarly, we denote $\Lambda^{\rho}$ the
$\widetilde{\La}^{+}$-algebra structure induced on $K[t^{-1},t]$ by
(\ref{eq:algebra-rho}). With these conventions we now have the
Lagrangian quantum homology $QH(L;(\Lambda^{\rho})^{+})$ together with
the quantum operations as described in the previous section as well as
the corresponding structures over $\Lambda^{\rho}$.  The differential
of the pearl complex with coefficients in $(\Lambda^{\rho})^{+}$ is
denoted by $d^{\rho}$ and we adopt a similar notation for all further
structures over these twisted coefficients.

A particular important case comes from representations of
$H_1(L;\mathbb{Z})$. More specifically, let $\rho' : H_1(L;\mathbb{Z})
\longrightarrow \mathbb{C}^*$ be a homomorphism. Then we can take in
the preceding construction $\rho = \rho' \circ \partial : H_2^D
\longrightarrow \mathbb{C}^*$, where $\partial : H_2^D \longrightarrow
H_1(L;\mathbb{Z})$ is the connectant map.

\

From now on we will use the following notation.  We abbreviate $H_1 =
H_1(L;\mathbb{Z})$. For an abelian group $H$ we write
$\textnormal{Hom}_0(H,\mathbb{C}^*)$ for the group of homomorphisms
$\rho: H \longrightarrow \mathbb{C}^*$ that are trivial (i.e. equal
$1$) on all torsion elements in $H$. Clearly there exists a
(non-canonical) isomorphism $\textnormal{Hom}_0(H, \mathbb{C}^*) \cong
(\mathbb{C}^*)^{\times r}$, where $r = \textnormal{rank}
(H/\textnormal{Torsion}(H))$.

\begin{rem}\label{rem:non-com}
   While we will not use non-commutative representations in this
   paper, much of the discussion described below can be adapted to
   coefficients twisted by representations of $\pi_{2}(M,L)$ with
   values in some non-necessarily commutative Lie group.

\end{rem}

\subsubsection{Relation to Floer homology} \label{sbsb:rel-hf}
Twisting the coefficients using representations $\rho$ has a
counterpart in Floer homology. Recall from~\cite{Bi-Co:rigidity} that
for the ring $\mathcal{R} = \Lambda$, or more generally for
$\mathcal{R}$'s that are $K[H_2^D]$--algebras, there is a canonical
isomorphism $QH(L;\mathcal{R}) \cong HF(L,L;\mathcal{R})$. Note that
if we take here $\mathcal{R} = \Lambda^{\rho}$ then
$HF(L,L;\Lambda^{\rho})$ can be naturally identified with
$HF((L,E_{\rho}), (L,E_{\rho}))$ which is a version of Floer homology
in which the coefficients are twisted in a flat complex line bundle
$E_{\rho} \longrightarrow L$. The relation of $E_{\rho}$ to $\rho$ is
that $\rho : H_1 \longrightarrow \mathbb{C}^*$ determines the
holonomies of the corresponding flat connection along loops in $L$.
Incorporating flat bundles into Floer homology was first introduced by
Kontsevich~\cite{Ko:ICM-HMS} (see also Fukaya~\cite{Fu:hfms-1}) in the
context of homological mirror symmetry. Due to considerations coming
from physics only unitary bundles were considered (i.e. $\rho: H_1
\longrightarrow S^1$). More recently it was discovered by
Cho~\cite{Cho:non-unitary} that working with non-unitary bundles makes
sense and is actually very useful. This point of view was further
developed and generalized by Fukaya, Oh, Ohta and
Ono~\cite{FO3:toric-1}.

\subsection{Elementary enumerative invariants}\label{sbsb:enum-elem}
As discussed in \S\ref{sec:intro} the purpose of the paper is to
discuss enumerative invariants that can be extracted from the algebraic
structures before. While later in the paper we will mainly concentrate
on the quantum product we now make explicit some simpler such
invariants.

\subsubsection{Disks of Maslov $2$} \label{subsubsec:enum-2disks} The
number of disks of Maslov class $2$, through a generic point of $L$
and in a given homotopy class $B\in \pi_{2}(M,L)$.  Thus, with the
notation of \S \ref{sb:additional-conv}, we are talking about the
degree of the evaluation map
$ev:(\widetilde{\mathcal{M}}(B,J)\times\partial D)/G \to L$, $(u,z)\to
u(z)$ under the assumption $\mu(B)=2$. This degree is well-defined and
independent of $J$ precisely because we work under the assumption
$N_{L}\geq 2$ which avoids any bubbling for disks of Maslov index $2$.
This point of view is formalized in the ``superpotentials'' of
\S\ref{sb:superpotential}.
\subsubsection{Invariants related to the quantum  inclusion}
\label{subsubsec:enum-inclu}
Assume that $L$ is so that $QH(L)\not=0$ and that $2n=4$. Fix two
points $P\in M\backslash L$ and $Q\in L$. The number of
$J$-holomorphic disks $u:(D,\partial D)\to (M,L)$ with $\mu(u)=4$ and
$u(-1)=Q$, $u(0)=P$ (for a generic $J$) is independent of $J$ and $P$
and $Q$.  By contrast with the remark in \S\ref{subsubsec:enum-2disks}
an argument is necessary here.  One way to see this is by using the
quantum inclusion of \S\ref{subsubsec:inclu}.  We will assume for this
argument the definition of the quantum inclusion in terms of the pearl
complex (the relevant moduli spaces are recalled in
\S\ref{sbsb:or-mod}). We first remark that for $\dim(L)=2$ we have
that $QH(L)\not=0$ implies that $QH(L)\cong H(L)\otimes \La$.  We now
consider a Morse function $f:L\to \R$ with a single minimum $x_{0}$
together with a Riemannian metric $(\cdot, \cdot)$ and a generic
almost complex structure $J$.  To define the quantum inclusion we also
need a metric on $M$ and a Morse function $h:M\to \R$. We assume that
$h$ has a single maximum $y_{4}$.  With these assumptions, the quantum
homology class $[x_{0}]$ is defined. However, in general, this class
is not independent of the choice of $\mathscr{D}=(f, (\cdot,
\cdot),J)$.  For a second choice of data $\mathscr{D}'=(f',(\cdot,
\cdot),J')$ where $f'$ is a Morse function with a single minimum
$x'_{0}$, the relation between the two classes is
$[x_{0}]=[x'_{0}]+q[L]t$ for some $q \in K$. Here $[L]$ represents the
fundamental class of $L$ (this class is well defined and independent
of the choices of data $\mathscr{D}$). It is easy to see that
$i_{L}([L])$ coincides with the classical singular inclusion. But this
implies that $i_{L}([x_{0}])=i_{L}([x'_{o}])$.  On the other hand, we
write $i_{L}([x_{0}])=[pt]+k_{1}at+k_{2}[M]t^{2}$ where $k_{i}\in K$,
$a\in H(M)$ and $[M]$ represents the fundamental class of $M$. By
using the chain level definition of $i_{L}$ for the data
$\mathscr{D}$, we see that $k_{2}$ equals the number of
$J$-holomorphic disks through $x_{0}$ and and $y_{4}$. As the
coefficient $k_{2}$ is independent of $\mathscr{D}$ (as well as of
$h$) we deduce that this is invariant.
 
The argument above does not work anymore in higher dimensions: both
the equality $i_{L}([x_{0}])=i_{L}([x'_{0}])$ and the interpretation
of $k_{2}$ are not necessarily valid anymore.  However, in Proposition
\ref{p:iL-x_0} we will see that - by a different and much more subtle
argument - the quantum inclusion of the point $[x_{0}]$ is independent
of $\mathscr{D}$ for monotone toric fibers. Moreover, it follows from
~\S\ref{sb:wide-qh} that for monotone toric fibres in dimension
$2n=4$, the invariant $k_2$ can be computed via the well known
Batyrev-Givental isomorphism (see formula~\eqref{eq:iL-x_0-2} after
Proposition~\ref{p:iL-x_0}).

%% file: wide-var.tex
\section{Wide varieties} \label{s:wide}

Let $L \subset M$ be a Lagrangian submanifold and $\mathcal{R}$ a
$\widetilde{\Lambda}^+$--algebra. Following~\cite{Bi-Co:rigidity,
  Bi-Co:Yasha-fest} we say that $L$ is {\em $\mathcal{R}$--wide} if
there exists an isomorphism $QH(L;\mathcal{R}) \cong H(L;\mathcal{R})$
between the quantum homology and the singular homologies of $L$, taken
with coefficients in $\mathcal{R}$. (Note that we do not require
existence of a canonical isomorphism here.) At the other extreme we
have {\em $\mathcal{R}$--narrow} Lagrangians. By this we mean
Lagrangians $L$ with $QH(L;\mathcal{R})=0$.

We now consider the moduli of all representations $\rho$ which make a
given Lagrangian wide. More precisely define:
\begin{equation} \label{eq:W_2} \mathcal{W}_2 = \{ \rho \in
   \textnormal{Hom}_0(H_2^D, \mathbb{C}^*) \mid L \, \textnormal{is }
   \, \Lambda^{\rho}\textnormal{--wide} \}.
\end{equation}
Similarly, put:
\begin{equation} \label{eq:W_1} \mathcal{W}_1 = \{ \rho' \in
   \textnormal{Hom}_0(H_1, \mathbb{C}^*) \mid \rho' \circ
   \partial \in \mathcal{W}_2  \}.
\end{equation}
We call $\mathcal{W}_2$ and $\mathcal{W}_1$ the {\em wide varieties}
associated to $L$. The connectant $\partial : H_2^D \longrightarrow
H_1$ induces a map $\partial_{\mathcal{W}}: \mathcal{W}_1
\longrightarrow \mathcal{W}_2$.

Note that both $\textnormal{Hom}_0(H_2^D, \mathbb{C}^*) \cong
(\mathbb{C}^*)^{\times r}$ and $\textnormal{Hom}_0(H_1, \mathbb{C}^*)
\cong (\mathbb{C}^*)^{\times l}$ are complex algebraic varieties (in
fact algebraic groups isomorphic to complex tori), where $r =
\textnormal{rank} (H_2^D)_{\textnormal{free}}$, $l = \textnormal{rank}
(H_1)_{\textnormal{free}}$.

\subsection{The wide varieties are algebraic}
We will now see that:
\begin{prop} \label{p:wide-var} For any monotone Lagrangian and with
   the notations above the sets $\mathcal{W}_2$ and $\mathcal{W}_1$
   are algebraic subvarieties of $\textnormal{Hom}_0(H_2^D,
   \mathbb{C}^*)$ and $\textnormal{Hom}_0(H_1, \mathbb{C}^*)$
   respectively.  Moreover, the map $\partial_{\mathcal{W}}:
   \mathcal{W}_1 \longrightarrow \mathcal{W}_2$ is a morphism of
   algebraic varieties.
\end{prop}
\begin{proof} We first treat the case when $L$ admits a perfect Morse
   function.  Let $f: L \longrightarrow \mathbb{R}$ be a perfect Morse
   function.  Add to it a Riemannian metric $(\cdot,\cdot)$ and an
   almost complex structure $J \in \mathcal{J}$ so that the triple
   $\mathscr{D} = (f, (\cdot,\cdot), J)$ is regular. Since $f$ is
   perfect we have an isomorphism of graded vector spaces
   $\mathcal{C}(\mathscr{D};\widetilde{\La}^{+})\cong (H(L;\C)\otimes
   \widetilde{\La}^{+})_*$ and $\mathcal{C}_*(\mathscr{D}; \Lambda)
   \cong (H(L;\mathbb{C}) \otimes \Lambda)_*$. For dimension reasons
   it follows that $L$ is $\Lambda^{\rho}$--wide if and only if the
   twisted pearly differential vanishes: $d^{\rho} (x) = 0$ for every
   $x \in \textnormal{Crit}(f)$. Notice that the differential $d$ of
   the complex $\mathcal{C}(\mathscr{D};\widetilde{\La}^{+})$ applied
   on $x\in\Crit(f)$ is of the form:
   $$ dx = \sum_{A,y}m_{xy}(A)yT^{A}\ ,\  m_{xy}(A)\in \Z$$
   so that $m_{xy}(A)$ vanishes whenever $|x|-|y|+\mu(A)\not=1$.
   The twisted differential $d^{\rho}$ is then written:
   $$ d^{\rho}x=\sum_{A,y}m_{xy}(A)\rho(A)yt^{\bar{\mu}(A)}$$
   in other words $d^{\rho}=d\otimes_{\rho} K[t]$.

   Pick a basis $A_1, \cdots, A_r$ for $(H_2^D)_{\textnormal{free}}$.
   This yields an identification of $\textnormal{Hom}_0(H_2^D,
   \mathbb{C}^*) \cong (\mathbb{C}^*)^{\times r}$. Use this
   identification to write $\rho$ as a tuple $(z_1, \ldots, z_r) \in
   (\mathbb{C}^*)^{\times r}$ so that if $A\in H_{2}^D$ is given by
   $A=\sum a_{i}A_{i}$, then $\rho(A)=\prod z_{i}^{a_{i}}$. Thus the
   condition $d^{\rho}(x) = 0$ translates into a polynomial equation
   in $z_1, \ldots, z_r$. As there are finitely many critical points
   $x$ we get a system with finite number of algebraic equations for
   $\mathcal{W}_2$. The proof for $\mathcal{W}_1$ is similar.

   We now turn to the general case - when perfect Morse functions
   might not exist.  We will make use of Proposition
   \ref{prop:min_pearls} by replacing in the argument above the
   complex $\mathcal{C}(\mathscr{D};\widetilde{\La}^{+})$ with a
   minimal pearl complex $\mathcal{C}_{min}(L)=(H(L;K)\otimes
   \widetilde{ \La}^{+},\delta)$.  Similarly, we replace the twisted
   pearl complex associated to $\rho$ and $\mathscr{D}$ with the
   complex $\mathcal{C}^{\rho}_{min}(L)=\mathcal{C}_{min}(L)\otimes
   \La^{\rho}$.  The differential of this complex, $d^{\rho}_{min}$,
   verifies $d^{\rho}_{min}=\delta\otimes\La^{\rho}$.  Again for
   degree reasons, $L$ is $\La^{\rho}$ wide iff the differential
   $d^{\rho}_{min}$ in the complex $\mathcal{C}^{\rho}_{min}(L)$
   vanishes.  We can then apply the argument above by using any fixed
   basis of $H_{\ast}(L;\C)$ in the place of the set of critical
   points of $f$.
\end{proof}

Versions of the moduli spaces $\mathcal{W}_2$ have already been
considered by Cho~\cite{Cho:non-unitary}. An analogue of
$\mathcal{W}_1$ (but with Novikov ring valued representations) has
played a central role in the work of Fukaya, Oh, Ohta and
Ono~\cite{FO3:toric-1}. Our approach below is somewhat different than
these works. We will not use the varieties $\mathcal{W}$ in order to
study Lagrangian intersections, but rather in order to construct new
invariants of Lagrangians.

\subsection{Regular functions and wide rings of coefficients}
\label{sb:reg-functions} From now on we will implicitly assume that
$\mathcal{W}_i$ (for either $i=1$ or $2$) is not an empty set.

Denote by $\mathcal{O}(\mathcal{W}_1)$ and
$\mathcal{O}(\mathcal{W}_2)$ the rings of global algebraic functions
on $\mathcal{W}_1$ and $\mathcal{W}_2$ respectively.  We do not grade
these rings. Given $A \in H_2^D$, denote by $f_A \in
\mathcal{O}(\mathcal{W}_2)$ the function defined by $f_A(\rho) :=
\rho(A)$. Consider now the following map:
\begin{equation} \label{eq:wide-alg} q: \widetilde{\Lambda}^{+}
   \longrightarrow \mathcal{O}(\mathcal{W}_2) \otimes \Lambda, \quad
   q(T^A) := f_A t^{\bar{\mu}(A)}.
\end{equation}
It is easy to check that $q$ is graded
homomorphism of rings hence $\mathcal{O}(\mathcal{W}_2) \otimes
\Lambda$ becomes a $\widetilde{\Lambda}^{+}$--algebra. In a similar
way we can define such a structure on $\mathcal{O}(\mathcal{W}_1)
\otimes \Lambda$.

With this setup we can define $QH(L; \mathcal{O}(\mathcal{W}_i)
\otimes \Lambda^{+})$, $i=1, 2$ (and similarly for $\Lambda$).
It easily follows from the definitions
that
\begin{equation}\label{eq:iso-with-singular}
QH(L; \mathcal{O}(\mathcal{W}_i) \otimes \Lambda^{+}) \cong
H(L;\mathcal{O}(\mathcal{W}_i) \otimes \La^{+})
\end{equation}
 and similarly for $\La$. Note that
these isomorphisms are not canonical.

Next, we have all the quantum operations with coefficients in
$\mathcal{R}^{+}_{i} = \mathcal{O}(\mathcal{W}_i) \otimes \Lambda^{+}$
as described in~\eqref{eq:qprod},~\eqref{eq:qmod} and~\eqref{eq:qinc}
and similarly for $\mathcal{R}_{i}=\mathcal{O}(\mathcal{W}_{i})\otimes
\Lambda$.

To shorten notation we will write from now on:
\begin{equation} \label{eq:qh-w} Q^{+}H(L;\mathcal{W}_i) := QH(L;
   \mathcal{O}(\mathcal{W}_i) \otimes \Lambda^{+}), \quad
   QH(L;\mathcal{W}_i) := QH(L; \mathcal{O}(\mathcal{W}_i) \otimes
   \Lambda).
\end{equation}

\subsection{The superpotential} \label{sb:superpotential} Here we
assume that $L^n \subset M^{2n}$ is a monotone Lagrangian with $N_L =
2$.

Pick a generic almost complex structure $J \in \mathcal{J}$. Using the
same notations as in \S\ref{sb:additional-conv} (see also \S\ref{sb:orient}) let $B \in H_2^D$ with
$\mu(B)=2$ and denote by $\widetilde{\mathcal{M}}(B,J)$ the space of
$J$-holomorphic disks $u:(D, \partial D) \longrightarrow (M,L)$ with
$u_*([D]) = B$, and by $G = Aut(D) \cong PSL(2, \mathbb{R})$ the group
of biholomorphisms of the disk.  Consider now the space of disks with
one marked point on the boundary, i.e. $(\widetilde{\mathcal{M}}(B,J)
\times \partial D) / G$, where $G$ acts as follows $\sigma \cdot (u,z)
= (u \circ \sigma^{-1}, \sigma(z))$, for $\sigma \in G$. By standard
arguments (see e.g.~\cite{Bi-Co:rigidity}) it follows that
$(\widetilde{\mathcal{M}}(B,J) \times \partial D) / G$ is a smooth
compact manifold without boundary and of (real) dimension $n$.
Moreover by our assumptions on $L$ (i.e. $L$ is oriented, spin and
with a prescribed choice of spin structure) the latter moduli space is
also oriented. Consider the evaluation map
$$ev: (\widetilde{\mathcal{M}}(B,J)
\times \partial D) / G \longrightarrow L, \quad ev(u,z) = u(z).$$ We
denote by $\nu(B)\in \mathbb{Z}$ the degree of this map. Standard
arguments then show that $\nu(B)$ does not depend on $J$ but only on
$B$. Moreover, there can be at most a finite number of classes $B \in
H_2^D$ with $\nu(B) \neq 0$. Put
$$\mathcal{E}_2 = \{ B \in H_2^D \mid \nu(B) \neq 0\}.$$

Define now the following function
\begin{equation} \label{eq:LG-1}
   \mathscr{P}:
   \textnormal{Hom}_0(H_1, \mathbb{C}^*) \longrightarrow \mathbb{C},
   \quad \mathscr{P}(\rho) = \sum_{B \in \mathcal{E}_2} \nu(B)
   \rho(\partial B).
\end{equation}
This function (and other analogous versions of it) is called the {\em
  Landau-Ginzburg superpotential}. It plays an important role in the
theory of mirror symmetry for toric varieties.  Its relation to
Lagrangian Floer theory was first noticed by Hori and
Vafa~\cite{Ho:linear, Ho-Va:ms} and further explored by Cho and
Oh~\cite{Cho-Oh:Floer-toric} and by Fukaya, Oh, Ohta and
Ono~\cite{FO3, FO3:toric-1}.

\subsubsection{Explicit formulae for $\mathcal{W}_{1}$}
\label{sbsb:sp-formulae}
We will now write $\mathscr{P}$ in coordinates. Fix a basis
$$\mathbf{e} = \{ e_1, \ldots, e_l \}$$ for
$H_1(L;\mathbb{Z})_{\textnormal{free}}$. For an element $a \in
H_1(L;\mathbb{Z})_{\textnormal{free}}$ we denote by $(a) = ((a)_1,
\ldots, (a)_l) \in \mathbb{Z}^{\times l}$ the vector of coordinates of
$a$ with respect to the basis $\mathbf{e}$ so that $a = (a)_1 e_1 +
\cdots + (a)_l e_l$.  Using the basis $\mathbf{e}$ we can identify
$\textnormal{Hom}_0(H_1, \mathbb{C}^*) \cong (\mathbb{C}^*)^{\times
  l}$. With these choices fixed, we write an element $\rho \in
\textnormal{Hom}_0(H_1, \mathbb{C}^*)$ as $(z_1, \ldots, z_l) \in
(\mathbb{C}^*)^{\times l}$, where $z_j = \rho(e_j)$.  In these
coordinates~\eqref{eq:LG-1} becomes:
\begin{equation} \label{eq:LG-2} \mathscr{P}(z_1, \ldots, z_l) =
   \sum_{B \in \mathcal{E}_2} \nu(B) z_1^{(\partial B)_1} \cdots
   z_l^{(\partial B)_l}.
\end{equation}

The relevance of $\mathscr{P}$ in our context is that we can describe
the wide variety $\mathcal{W}_{1}$ by means of the derivatives of
$\mathscr{P}$.  To see this fix a basis $\mathbf{C} = \{C_1, \ldots,
C_l \}$ for $H_{n-1}(L;\mathbb{Z})$ which is dual
to $\mathbf{e}$ in the sense that $C_i \cdot e_j = \delta_{i,j}$,
where $\cdot$ is the intersection pairing
(see~\S\ref{sbsb:intersection-prod}). Now let
$\mathcal{C}_{min}(L)=H_{\ast}(L;\C)\otimes \La^{+}$ be a minimal
pearl complex as provided by Proposition \ref{prop:min_pearls}. Let
$\mathcal{C}_{min}^{\rho}(L)=\mathcal{C}_{min}(L)\otimes\La^{\rho}$
and denote by $d^{\rho}_{min}$ the differential of this last complex.
Of course, in case $L$ admits a perfect Morse $f$ function we can
simply take instead of $\mathcal{C}_{min}(L)$ the pearl complex of $f$
and $d^{\rho}_{min}$ coincides in this case with the differential
$d^{\rho}$ of the pearl complex of $f$ twisted by the representation
$\rho$. We can write the (twisted) pearl differential
$$d^{\rho}_{min}: (H(L;\mathbb{C}) \otimes \Lambda)_* \longrightarrow
(H(L;\mathbb{C}) \otimes \Lambda)_{*-1}.$$

\begin{prop} \label{p:dC-i}
   For $\rho = (z_1, \ldots, z_l)$ we have:
   \begin{enumerate}
     \item $d^{\rho}_{min}(C_j) = z_j \frac{\partial
        \mathscr{P}}{\partial z_j}\, [L]t.$ \label{i:dC-i-1}
     \item If $QH(L; \Lambda^{\rho}) \neq 0$ then $\rho$ is a critical
      point of $\mathscr{P}$. \label{i:dC-i-qh}
   \end{enumerate}
   In particular, $\mathcal{W}_1 \subset
   \textnormal{Crit}(\mathscr{P})$. Moreover, if the cohomology ring
   $H^*(L;\mathbb{R})$ (with the classical cup product) is generated
   by $H^1(L;\mathbb{R})$ then $\mathcal{W}_1 =
   \textnormal{Crit}(\mathscr{P})$.
\end{prop}
\begin{proof} In case $L$ admits a perfect Morse function, the proof
   of~(\ref{i:dC-i-1}) follows immediately from the definition of the
   pearl complex together with our orientation conventions.
   Concerning the orientations the main point to verify here is the
   following. Given $B\in H_{2}^{D}$, denote $$Q =
   \mathcal{M}_2(B,J)\, \lrsub{e_{\scriptscriptstyle +1}}{\times}{i}
   \,\{ m\}.$$ Here we use the fiber product and its orientation as
   defined in~\S\ref{sbsb:fiber-product}, and $i:\{m\} \to L$ stands
   for the inclusion of a point. Consider now the evaluation map
   $e_{\scriptscriptstyle -1}: Q \longrightarrow L$. Then we have in
   homology:
   \begin{equation} \label{eq:ev-cycle}
      (e_{\scriptscriptstyle -1})_*([\overline{Q}]) =
      (-1)^n \nu(B) \partial B,
   \end{equation}
   where $\partial : H_{2}(M,L;\Z)\to H_{1}(L;\Z)$ is the connectant.
   This can be checked by a straightforward computation based on the
   conventions described in~\S\ref{a:orientations}.

   If $L$ does not admit a perfect Morse function we use a minimal
   pearl complex together with its structural maps $\phi$ and $\psi$
   as in Proposition~\ref{prop:min_pearls} :
   $$\mathcal{C}(\mathscr{D})\stackrel{\phi}{\longrightarrow}
   \mathcal{C}_{min}(L) \stackrel{\psi}{\longrightarrow}
   \mathcal{C}(\mathscr{D})$$ where $\mathscr{D}$ is a generic set of
   data required to define the pearl complex. By using the fact that
   both these maps induce an isomorphism in Morse homology the result
   is again immediate.

   To prove~(\ref{i:dC-i-qh}), recall that $[L] \in
   \mathcal{C}_{min}(L)$ is a cycle whose homology class is the unity
   of the ring $QH(L; \Lambda^{\rho})$. Thus $QH(L;\Lambda^{\rho})
   \neq 0$ iff $[L]$ is not a boundary. In view of~(\ref{i:dC-i-1}),
   if $QH(L; \Lambda^{\rho}) \neq 0$ we must have $\frac{\partial
     \mathscr{P}}{\partial z_j}(\rho)=0$ for every $j$.

   The last statement follows immediately from the following fact:
   {\sl If $H^*(L)$ is generated by $H^1(L)$ then $L$ is either
     $\Lambda^{\rho}$--narrow or $\Lambda^{\rho}$--wide.  Moreover,
     the second case occurs iff $d^{\rho}=0$ on $H_{n-1}(L)$.} The
   proof of this can be essentially found in~\cite{Bi-Co:rigidity}
   where it is proved for the ground ring $K = \mathbb{Z}_2$ and
   without any representations $\rho$, but the same proof with obvious
   changes extends to our setting.
\end{proof}

\begin{rem} \label{r:wide-var-arith} Both varieties $\mathcal{W}_1$
   and $\mathcal{W}_2$ are arithmetic in the sense that in some
   coordinate system they are cut by a system of equations with
   integral coefficients.
\end{rem}

\begin{rem}\label{rem:Floer-potential}
   Given that $QH(L;\mathcal{W}_{i})$ is isomorphic to Floer homology
   with coefficients in $\mathcal{O}(\mathcal{W}_{i})\otimes
   \C[t,t^{-1}]$ (as discussed in \S\ref{sbsb:rel-hf}) and in view of
   equation (\ref{eq:iso-with-singular}) it follows that for a
   Lagrangian with $\mathcal{W} \neq \emptyset$ any $L'$ transverse to
   $L$ and Hamiltonian isotopic to it, intersects $L'$ in at least
   $\sum_{i} \dim (H_{i}(L;\C))$ intersection points. Note that when
   $L$ is a torus, checking that $\mathcal{W} \neq \emptyset$ can be
   done by verifying that $\textnormal{Crit}(\mathscr{P}) \neq
   \emptyset$, according to Proposition~\ref{p:dC-i}.
\end{rem}

We now turn to the relation between the quantum product and the
superpotential. Recall that when $L$ is $\mathcal{R}$--wide,
$QH(L;\mathcal{R})$ is not in general {\em canonically} isomorphic to
$H(L;\mathbb{C}) \otimes \mathcal{R}$.  However, there exist canonical
embeddings $H_i(L;K) \hooklongrightarrow QH_i(L;\mathcal{R})$ for
every $n-N_L+1 \leq i \leq n$. (See~\S 4.5 in~\cite{Bi-Co:rigidity}
and Proposition~4.5.1 there.)  As $N_L\geq 2$, we view $H_{n-1}(L;K)$
as a subspace of $QH_{n-1}(L; \mathcal{R})$ and $H_n(L;K)$ as a
subspace of $QH_n(L;\mathcal{R})$. The following proposition gives
information on the quantum product of elements in this special
subspace in terms of the superpotential.

\begin{prop} \label{p:prod-Ci-Cj} Consider $H_{n-1}(L;\mathbb{C})$ as
   a subset of $QH_{n-1}(L;\mathcal{W}_1)$. Then we have:
   \begin{equation} \label{eq:cicj}
      C_i * C_j + C_j * C_i =(-1)^{n} z_i z_j
      \frac{\partial^2 \mathscr{P}}{\partial z_i \partial z_j} \,
      [L]t,
   \end{equation}
   where $[L] \in H_n(L;\mathbb{C}) \subset QH_n(L; \mathcal{W}_1)$ is
   the unity. In other words, for every $\rho = (z_1, \ldots, z_l) \in
   \mathcal{W}_1$ we have the identity~\eqref{eq:cicj}, where $C_i,
   C_j, [L]$ are all viewed as elements of $QH(L; \Lambda^{\rho})$.
\end{prop}

\begin{proof} In case $L$ admits perfect Morse functions this follows
   from the definition of the quantum product together with our
   orientation conventions.  Indeed, in this case, assume that $f$,
   $f'$, $f''$ are three perfect Morse functions and that $a$, $b$ are
   critical points of $f$ and $f'$ of index $n-1$, and let $w$ be the
   maximum of $f''$.  The critical points $a$, $b$ are canonically
   identified with singular homology classes in $H_{n-1}(L;\C)$ and
   obviously $w$ is canonically identified with the fundamental class
   $[L]$.  The product in question (defined over
   $\widetilde{\La}^{+}$) is given by:
   $$a\ast b =a\cdot b + \sum_{\la\in\mathcal{E}_{2}}
   k_{ab}(\la)[L]T^{\la}$$ where $a\cdot b$ is the singular
   intersection product and $k_{ab}(\la)\in \Z$ is the number of
   $J$-holomorphic disks $u$ in the class $\la$ that pass through $w$
   and intersect the unstable manifolds of $a$ and of $b$ is such a
   way that along the boundary of the disk the order of the
   intersection points is $w, W^{u}_{f'}(b)\cap u(\partial D),
   W^{u}_{f}(a)\cap u(\partial D)$.  Obviously, the order requirement
   shows that this intersection condition is not purely homological: a
   different choice of functions $f$ and $f'$ might change the
   coefficient $k_{ab}(\la)$ here. However, the sum $a\ast b + b\ast
   a$ is invariant as the order is now irrelevant. From this
   description the formula claimed is obvious for $i\not=j$ -- it
   simply claims that if for $\lambda \in H_2^D$ with $\mu(\lambda)=2$
   we have algebraically $\nu(\lambda)$ disks in the class $\lambda$
   passing through $w$ then the contribution of these disks to
   $C_{i}\ast C_{j}+C_{j}\ast C_{i}$ is $\tau \nu(\lambda) (\partial
   \lambda)_i (\partial \lambda)_j$, where $(\partial \lambda)_i$ are
   the coefficients of $\partial \lambda$ in the basis $\mathbf{e}$,
   i.e.  $\partial \la = \sum_i (\partial \lambda)_i e_i$, and $\tau =
   \pm 1$ is a sign that depends on the orientations conventions.  In
   other words:
   $$C_i*C_j + C_j*C_i = \tau
   \sum_{\substack{\lambda \in H_2^D \\ \mu(\lambda)=2}} \nu(\lambda)
   (\partial \lambda)_i (\partial \lambda)_j z_1^{(\partial
     \lambda)_1} \cdots z_l^{(\partial \lambda)_j} [L] t.$$

   When $i=j$, we note that contribution of the disks in the class
   $\lambda$ to $2 C_i * C_i$ is $\tau \nu(\lambda)(\partial
   \lambda)_i^2$. We now use that fact that $\rho \in \mathcal{W}_1$
   hence $\partial \mathscr{P}/\partial z_{k}=0$ for every $k$, and so
   $\sum_k \nu(\lambda) (\partial \lambda)_k = 0$ by the
   point~(\ref{i:dC-i-1}) in Proposition~\ref{p:dC-i}.

   This implies the claimed formula up to showing that
   $\tau=(-1)^{n}$. In turn, this is a simple consequence of the
   orientation conventions for the quantum product
   (see~\S\ref{sbsb:or-prod}) and equation~\eqref{eq:ev-cycle} (with
   $\lambda$ instead of $B$).

   In case $L$ does not admit perfect Morse functions the proof uses
   minimal pearl complexes in a rather straightforward way.
\end{proof}

\begin{rem} It might seem slightly surprising that the coefficient
   $k_{ab}(\la)$ above is not necessarily invariant but still the
   quantum product $$QH(L;\widetilde{\La}^{+})\otimes
   QH(L;\widetilde{\La}^{+})\longrightarrow
   QH(L;\widetilde{\La}^{+})$$ is well defined. The explanation is
   that while $H_{n-1}(L;K)$ is canonically embedded in
   $QH(L;\widetilde{\La}^{+})$ this is no longer true for
   $H_{n-2}(L;K)$. Clearly $a\cdot b$ belongs precisely to
   $H_{n-2}(L;K)$ and so, even if $a\ast b$ is well defined and
   independent of choices, the class $a\cdot b$ is not canonically
   identified with a quantum class.  This is why the coefficient
   $k_{ab}(\la)$ is also, in general, not independent of the choices
   of $f,f',f''$.  On the other hand, the independence of
   $k_{ab}(\la)+k_{ba}(\la)$ of all choices can also be seen as an
   immediate consequence of the fact that $a\cdot b+b\cdot a=0$.
\end{rem}

\subsubsection{Relation to previous works} \label{sbsb:prev-works} The
relation of the superpotential to the non-vanishing of Floer
homological was first pointed out in the physics literature
in~\cite{Ho-Va:ms}. Versions of Propositions~\ref{p:dC-i}
and~\ref{p:prod-Ci-Cj} were later proved in~\cite{Cho:products,
  Cho-Oh:Floer-toric} for the case of Lagrangian torus fibres in toric
manifolds and in the setting of Floer homology. The toric case has
been further studied in~\cite{FO3:toric-1, FO3:toric-2}.

\subsubsection{Different versions of the superpotential}
\label{sbsb:different-potentials}

Different authors use different versions of the superpotential
functions, as well as different coordinate systems on
$\textnormal{Hom}_0(H_2^D, \mathbb{C}^*)$. For example, Fukaya, Oh,
Ohta and Ono use in~\cite{FO3:toric-1} coordinates $x_1, \ldots, x_l$
whose relation to ours is that $x_i = \log z_i$ (so that the $x$
coordinates are defined only modulo some periods). The superpotential
is then written as
$$\mathscr{P}'(x_1, \ldots, x_l) =
\sum_{\lambda, \mu(\lambda)=2} e^{x_1 (\partial \lambda)_1 + \cdots +
  x_l(\partial \lambda)_l}.$$ The formula at point~(\ref{i:dC-i-1}) of
Proposition~\ref{p:dC-i} then becomes $d^{\rho}_{min}(C_j) =
\frac{\partial \mathscr{P}'}{\partial x_j}\, [L]t$. Similarly,
formula~\eqref{eq:cicj} becomes now $ C_i * C_j + C_j * C_i =(-1)^{n}
\frac{\partial^2 \mathscr{P}'}{\partial x_i
  \partial x_j} \, [L]t$.

Other authors, e.g.~\cite{Aur:t-duality, Aur:slag-1, Aur:slag-2} work
only with unitary representations, i.e. $\textnormal{Hom}_0(H_1, S^1)$
but allow the Lagrangian $L$ to move in a family of special Lagrangian
submanifolds. The superpotential is in this case a function of two
sets of real variables: the representation and the parameter of the
Lagrangian. However, these two sets of variables can be put together
to form a complex system of coordinates in which the superpotential
becomes holomorphic. The relation between this superpotential and ours
is rather straightforward.

\medskip

There is a more general but also less transparent definition of a
superpotential that also expresses $\mathcal{W}_{2}$ in a way similar
to the one described above.  Moreover, this description also works for
$N_{L}>2$. We indicate it here.

Let $C_1, \ldots, C_k$ be a basis of $H_{n-N_{L}+1}(L;K)$, and $f_1,
\ldots, f_s$ a basis for $(H_2^D)_{\textnormal{free}}$.  Fix a point
$P$ in $L$. Define a function $\mathscr{P}:(\mathbb{C}^*)^{\times k}
\times (\mathbb{C}^*)^{\times s}$ by
\begin{equation}\label{eq:potential1}
   \mathscr{P}(z_{1},\ldots,
   z_{k}, w_1, \ldots, w_s) =
   \sum_{\substack{\alpha\in (H_2^D)_{\textnormal{free}};
       \\ \mu(\alpha)=N_L}} z_{1}^{r_{1}(\alpha)}\cdots
   z_{k}^{r_{k}(\alpha)}w_{1}^{(\alpha)_1} \cdots
   w_{s}^{(\alpha)_s}.
\end{equation}
The exponents $r_{i}, (\alpha)_{j} \in \mathbb{Z}$ are related to
$\alpha$ as follows: $\alpha=\sum_{j} (\alpha_j) f_{j}$ and
$r_{i}(\alpha)$ is the intersection number of the homology class
$C_{i}$ with the class $D_{\alpha} \in H_{N_L-1}(L;K)$ which is
defined as follows. Put $Q = \mathcal{M}_2(\alpha,J)
\lrsub{e_{\scriptscriptstyle +1}}{\times}{i}\{P\}$, where $i: \{P\}
\to L$ is the inclusion (see~\S\ref{sbsb:fiber-product} for the
definitions of the orientation on the fiber product). The closure of
$Q$ is an oriented compact manifold $\overline{Q}$ without boundary.
Moreover, the second evaluation map $e_{\scriptscriptstyle -1}$
extends to $\overline{Q}$. We define $D_{\alpha} =
(e_{\scriptscriptstyle -1})_* [\,\overline{Q}\,]$.

This potential is independent of $P$ as well as of (the generic
choice of) $J$. For convenience we put $\mathbf{z}=(z_{1},\ldots,
z_{k})$.

Similarly to the case $N_{L}=2$ previously discussed we have that if
the real cohomology of $L$ is generated as an algebra by $H^{<
  N_{L}}(L;\mathbb{R})$, then
$$\mathcal{W}_{2}(L)=\{\rho \in
\textnormal{Hom}_0(H_2^D,\mathbb{C}^*) \mid
d_{\mathbf{z}}\mathscr{P}(1,\ldots,1;\rho(f_{1}),
\ldots,\rho(f_{s}))=0\}~.~$$

%% file: qforms.tex
\section{Quadratic forms} \label{s:qforms}

Let $L^n \subset M^{2n}$ be a monotone Lagrangian. We continue to
assume that $L$ is oriented and spin and fix once and for all a spin
structure. We will introduce now a quadratic form associated to $L$
from which we can derive new invariants. The construction works best
when $N_L=2$, so we first describe it in this case and then do the
general case.

\subsection{The case $N_L=2$} \label{sb:qf-NL=2} Let $\mathcal{R}$ be
a $\widetilde{\Lambda}^{+}$--algebra for which $L$ is
$\mathcal{R}$--wide. Assume in addition that $\mathcal{R}_k = 0$ for
every $k>0$ (e.g. $\mathcal{R} = \mathcal{O}(\mathcal{W}) \otimes
\Lambda^+$). We will denote by $\mathcal{W}$ any of the two wide
varieties $\mathcal{W}_{1}$ and $\mathcal{W}_{2}$ as long as the
distinction among them is not important.

We have a canonical isomorphism:
\begin{equation} \label{eq:qh-iso1}
   QH_{n-1}(L;\mathcal{R}) \cong H_{n-1}(L;K) \otimes_K
   \mathcal{R}_{0},
\end{equation}
as well as a canonical exact sequence:
\begin{equation} \label{eq:ex-seq-qh_n-2} 0 \longrightarrow [L]
   \mathcal{R}_{-2} \stackrel{i}{\longrightarrow}
   QH_{n-2}(L;\mathcal{R}) \stackrel{\pi} \longrightarrow H_{n-2}(L;K)
   \otimes_K \mathcal{R}_0 \longrightarrow 0.
\end{equation}
See~\S 4.5 in~\cite{Bi-Co:rigidity} for the details. From now on we
will make the identification~\eqref{eq:qh-iso1} and also view
$[L]\mathcal{R}_{-2}$ as a subspace of $QH_{n-2}(L;\mathcal{R})$ via
$i$. A simple computation shows that:
$$\pi(a*b) = a \cdot b, \quad \forall a,b \in H_{n-1}(L;K),$$ where
$\cdot$ is the classical intersection product
(see~\S\ref{sbsb:intersection-prod}). In particular we have $\pi(a*a)
= 0$ for every $a \in H_{n-1}(L;K)$, and so we can define a map:
\begin{equation} \label{eq:a2-R} \widetilde{\varphi}: H_{n-1}(L;K)
   \longrightarrow \mathcal{R}_{-2}, \quad \textnormal{by the equation
   } \; a*a = \widetilde{\varphi}(a) [L].
\end{equation}
Obviously $\widetilde{\varphi}$ is a quadratic form, i.e. it is
homogeneous of degree $2$ over $K$.

We now restrict to the case $\mathcal{R} =R\otimes \Lambda^{+}$ with
$R$ some $K$-algebra.  In this case $\mathcal{R}_{-2}=tR$ and
$\widetilde{\varphi}$ induces an $R$-valued quadratic form $\varphi$
by putting $\varphi =t^{-1}\widetilde{\varphi}$. A particular case of
interest will be $R= \mathcal{O}(\mathcal{W})$ with $K=\mathbb{C}$ and
$\mathcal{W}$ is any of the wide varieties $\mathcal{W}_1$ or
$\mathcal{W}_2$. In this case we denote the resulting quadratic form
by:
$$\varphi_{_{\mathcal{W}}}:
H_{n-1}(L;\mathbb{C}) \longrightarrow \mathcal{O}(\mathcal{W}).$$ We
can also specialize to a particular $\rho \in \mathcal{W}$, i.e.
compose with the evaluation morphism
$e_{\rho}:\mathcal{O}(\mathcal{W}) \longrightarrow \mathbb{C}$,
$e_{\rho}(f) = f(\rho)$. We write: $\varphi_{\rho} = e_{\rho} \circ
\varphi_{_{\mathcal{W}}}$.

There is an important integral structure in this picture.
Consider the inclusion $H_{n-1}(L;\mathbb{Z}) \subset
H_{n-1}(L;\mathbb{C})$. The restriction of the quadratic form
$\varphi_{_{\mathcal{W}}}$ to $H_{n-1}(L;\mathbb{Z})$, which will
still be denoted by $\varphi_{_{\mathcal{W}}}$ will play an important
role in the sequel.

\begin{rem} \label{r:quad-z} Whenever the trivial representation
   $\rho_0 \equiv 1$ is in $\mathcal{W}$ the quadratic form
   $\varphi_{\rho_0}$ is integral. By this we mean that its
   restriction to $H_{n-1}(L;\mathbb{Z})$ gives values in
   $\mathbb{Z}$.

   It often happens that the variety $\mathcal{W}_1$ is
   $0$-dimensional (e.g. when the superpotential $\mathscr{P}$ has
   isolated critical points. See Proposition~\ref{p:dC-i}.) It follows
   from Remark~\ref{r:wide-var-arith} that in such cases for every
   $\rho \in \mathcal{W}_1$ the image $\rho(H_1(L;\mathbb{Z}))$ lies
   inside a number field $F \subset \mathbb{C}$. It easily follows
   that for every $\rho \in \mathcal{W}_1$ the restriction of the
   quadratic form $\varphi_{\rho}$ to $H_{n-1}(L;\mathbb{Z})$ gives
   values in the same field $F$.
\end{rem}

\subsection{The case of $N_{L}>2$.}\label{subs:highNL-quad}
The definition in this case is based on viewing $\widetilde{\varphi}$
as a secondary operation in the sense that it is defined precisely
when the square of the intersection product vanishes. Now assume that
$N_{L}>2$. We continue to assume that $L$ is $\mathcal{R}$--wide and
that $\mathcal{R}_k=0$ for all $k>0$. Recall that $N_{L}$ is even
because $L$ is orientable and write $N_{L}=2s$. Notice that we still
have a canonical isomorphism $H_{n-s}(L;K)\otimes_{K} \mathcal{R}_0
\cong QH_{n-s}(L;\mathcal{R})$.  Denote by $ H_{n-s}^{\sqrt{0}}(L;K)$
the cone consisting of those elements $x\in H_{n-s}(L;K)$ with $x\cdot
x=0$ where $\cdot$ is the intersection product. We now define:
$$\widetilde{\varphi}^{s}:H_{n-s}^{\sqrt{0}}(L;K)\to 
\mathcal{R}_{-2s}$$ by the relation $$\forall x\in
H_{n-s}^{\sqrt{0}}(L;K),\ x\ast x= \widetilde{\varphi}^{s}(x)[L]t~.~$$
Note that $H_{n-s}^{\sqrt{0}}(L;K)$ is in general only a cone (over
$K$) and might fail to be a $K$-module. Still, in some cases (e.g.
when $s=$ odd), $H_{n-s}^{\sqrt{0}}(L;K)$ is a $K$-module. In the
general case $\widetilde{\varphi}^{s}$ restricts to a quadratic form
on any subset of $H_{n-s}^{\sqrt{0}}(L;K)$ which is a $K$-submodule.

As in the case $N_{L}=2$ for $K = \mathbb{C}$,
$\mathcal{R}=\mathcal{O}(\mathcal{W})\otimes \La^{+}$ we obtain a
quadratic form
$\varphi_{_{\mathcal{W}}}=t^{-1}\widetilde{\varphi}^{s}$ with values
in $\mathcal{O}(\mathcal{W})$. We can also restrict
$\varphi_{_{\mathcal{W}}}$ to $H_{n-s}^{\sqrt{0}}(L;\mathbb{Z})$.

\begin{rem}
   \begin{enumerate}
     \item[a.] The operation defined above seems to be the first step
      in a sequence of higher order operations, each defined whenever
      the previous ones vanish. While these higher order operations
      are of interest we will not further discuss them here.
     \item[b.] The quadratic forms discussed here  have first
      appeared in Cho in \cite{Cho:products}  for $N_{L}=2$, $L$
      a toric fibre and the trivial representation.
   \end{enumerate}
\end{rem}

\subsection{The discriminant} \label{sb:discr} Let $F$ be a free
abelian group and $\mathcal{A}$ a commutative ring. Let $\varphi: F
\longrightarrow \mathcal{A}$ be a quadratic form. Recall that
$\varphi$ has a well defined invariant $\Delta \in \mathcal{A}$ called
the discriminant which is defined as follows. Pick a basis for $F$ and
represent $\varphi$ by a symmetric matrix $A$ in that basis. Then the
discriminant of $\varphi$, $$\Delta_{\varphi}= -\det(A),$$ does not
depend on the choice of the basis because any automorphism of $F$ has
$\det = \pm 1$.
 
Now let $L$ be a Lagrangian with $N_L=2$ as in~\S\ref{sb:qf-NL=2}.  (A
similar computation is possible for $N_{L}>2$). The discriminant
$\Delta$ of the quadratic form $\varphi_{_\mathcal{W}}$ is an element
of $\mathcal{O}(\mathcal{W})$. We denote its value at $\rho$ by
$\Delta(\rho) \in \mathbb{C}$.

To compute $\Delta$ explicitly fix a basis $\mathbf{C} = \{C_1,
\ldots, C_l \}$ for $H_{n-1}(L;\mathbb{Z})$.  Define functions $a_{i
  j} \in \mathcal{O}(\mathcal{W})$ by the relations: $$C_i*C_j +
C_j*C_i = a_{ij}[L]t.$$ Then we clearly have
\begin{equation} \label{eq:discr-1} \varphi_{_\mathcal{W}}(X_1 C_1 +
   \cdots + X_l C_l) = \frac{1}{2}\sum_{i,j} a_{ij} X_i X_j, \quad
   \Delta = -\det (a_{ij}).
\end{equation}
The minus sign in front of the determinant appears here in order to
make our discriminant compatible with conventions common in number
theory. In the same spirit, we take in the determinant the constants
$a_{i j}$ (instead of $\frac{1}{2}a_{ij}$) so that whenever the
trivial representation $\rho_0 \equiv 1$ is in $\mathcal{W}_1$ the
discriminant $\Delta(\rho_0)$ will be an integer.

When $\mathcal{W} = \mathcal{W}_1$ we can express $\Delta$ in terms of
the super potential as follows. We now use the notation
from~\S\ref{sbsb:sp-formulae}. Fix a basis $\mathbf{e} = \{ e_1,
\ldots, e_l \}$ for $H_1(L;\mathbb{Z})_{\textnormal{free}}$ and a
basis $\mathbf{C} = \{C_1, \ldots, C_l \}$ for $H_{n-1}(L;\mathbb{Z})$
which is dual to $\mathbf{e}$ as in~\S\ref{sbsb:sp-formulae}.  Write
$\rho = (z_1, \ldots, z_l)$ with respect to $\mathbf{e}$. Then in view
of formulas~\eqref{eq:cicj} and ~\eqref{eq:discr-1} we have:
\begin{equation} \label{eq:discr-2} \Delta(z_1, \ldots, z_l) =
   (-1)^{ln+1}z_1^2 \cdots z_l^2 \det \left(\frac{\partial^2
        \mathscr{P}}{\partial z_i
     \partial_j}\right).
\end{equation}

%% file: defo.tex
\section{The deformation viewpoint} \label{s:defo} Let $L$ be a
monotone Lagrangian which is $\mathcal{R}$--wide, where $\mathcal{R}$
is of the following kind: $\mathcal{R} = R \otimes_K \Lambda^+ = R[t]$
for some $K$--algebra $R$. We grade $t$ as usual, $|t| = -N_L$, but do
not grade $R$. Of course, $\mathcal{R}$ is also assumed to be endowed
with a $\widetilde{\Lambda}^+$--algebra structure, but we do not make
any special assumptions on it. For example we can take $\mathcal{R} =
K[t]$ with the $\widetilde{\Lambda}^+$--algebra structure given
by~\eqref{eq:algebra-rho} for some $\rho \in \mathcal{W}_2$ (we often
denote this ring also by $(\Lambda^{\rho})^+$ to emphasize its
$\widetilde{\Lambda}^+$--algebra structure coming from $\rho$).
Another example is $\mathcal{R} = \mathcal{O}(\mathcal{W})[t]$ with
the $\widetilde{\Lambda}^+$--algebra structure given
by~\eqref{eq:wide-alg}.

As $L$ is wide there exists an isomorphism $QH(L;\mathcal{R}) \cong
H(L;K) \otimes \mathcal{R}$ and, as mentioned before, usually there is
no canonical one. On the other hand, there is a distinguished class of
isomorphisms $QH(L;\mathcal{R}) \longrightarrow H(L;K) \otimes
\mathcal{R}$ which we now describe.

For simplicity we will assume from now on that $L$ admits a perfect
Morse function. If this is not the case, the use of minimal models
allows essentially the same results to be formulated in full
generality (we remark however that the actual construction of the maps
$\psi$ and $\phi$ from Proposition \ref{prop:min_pearls} is required,
this construction appears in \cite{Bi-Co:rigidity} pages 2929-2933).

\subsection{The quantum product as deformation of the intersection
  product} \label{subs:product-defint}
Let $\mathscr{D} = (f, (\cdot, \cdot), J)$ be a regular triple
consisting of a perfect Morse function $f: L \longrightarrow
\mathbb{R}$, a Riemannian metric $(\cdot, \cdot)$ on $L$ and an almost
complex structure $J \in \mathcal{J}$. Denote by $CM(\mathcal{F})$ the
Morse complex (with coefficients in $K$) associated to the pair
$\mathcal{F} = (f, (\cdot,\cdot))$. Denote by
$\mathcal{C}(\mathscr{D}; \mathcal{R})$ the pearl complex. Note that
the Morse differential on $CM(\mathcal{F})$ vanishes (since $f$ is
perfect). The differential of the pearl complex vanishes too because
$L$ is wide. It follows that the obvious map
$$\widetilde{h}_{\mathscr{D}} : \mathcal{C}(\mathscr{D}; \mathcal{R})
\longrightarrow CM(\mathcal{F}) \otimes_K \mathcal{R}, \;\;
\textnormal{induced by} \; \; \widetilde{h}_{\mathscr{D}}(x) = x, \;
\forall x \in \textnormal{Crit}(f)$$ is a chain map (in fact a chain
isomorphism). Denote by $h_{\mathscr{D}}: QH(L;\mathcal{R})
\longrightarrow H(L;K) \otimes \mathcal{R}$ the induced map in
homology. The isomorphism $h_{\mathscr{D}}$ is of course not
canonical, it depends on $\mathscr{D}$. Denote by $\mathcal{K}$ the
set of all isomorphism $QH(L;\mathcal{R}) \longrightarrow H(L;K)
\otimes \mathcal{R}$ obtained in this way from all possible triples
$\mathscr{D}$.

\begin{prop} \label{p:class-K} Elements of $\mathcal{K}$ have the
   following properties:
   \begin{enumerate}
     \item Every $h_{\mathscr{D}} \in \mathcal{K}$ sends the unity of
      $QH(L;\mathcal{R})$ to the unity $[L]$ of $H(L;K)$.
      \label{i:unity}
     \item For every two elements $h_{\mathscr{D}}, h_{\mathscr{D}'}
      \in \mathcal{K}$ we have:
      $$h_{\mathscr{D}'} \circ {h_{\mathscr{D}}}^{-1} =
      \textnormal{id} + \phi_1 t + \phi_2 t^2 + \cdots$$ where
      $\phi_k: H_*(L;K) \otimes R \longrightarrow H_{*+k N_L}(L;K)
      \otimes R$, $k \geq 1$. In other words, $h_{\mathscr{D}'} \circ
      {h_{\mathscr{D}}}^{-1}$ is a deformation of the identity.
      \label{i:hdhd'}
   \end{enumerate}
\end{prop}

\begin{proof}
   Let $\mathscr{D}' = (f', (\cdot, \cdot)', J')$ be antoher triple
   with $f'$ a perfect Morse function and put $\mathcal{F}' = (f',
   (\cdot,\cdot)')$ . Denote by $F_0: CM(\mathcal{F}) \longrightarrow
   CM(\mathcal{F}')$ the comparison map between the Morse complexes
   and by $F: \mathcal{C}(\mathscr{D}) \longrightarrow
   \mathcal{C}(\mathscr{D}')$ the comparison between the pearl
   complexes. We have: $$F(x) = F_0(x) + F_1(x)t + F_2(x) t^2 +
   \cdots, \; \forall x \in \textnormal{Crit}(f),$$ for some maps $F_k
   : CM_*(\mathcal{F}) \otimes R \longrightarrow CM_{*+k
     N_L}(\mathcal{F}') \otimes R$.  See~\cite{Bi-Co:Yasha-fest,
     Bi-Co:rigidity} for more details.  Notice that the comparison
   chain morphism $F$ is defined by using appropriate homotopies
   relating the data $\mathscr{D}$ and $\mathscr{D}'$ and is unique,
   in general, only up to chain homotopy. In this case however, the
   differentials of the two involved complexes vanish so that $F$
   itself is canonical.
\end{proof}
For further use denote by $$\mathscr{G}_{L}=\{h_{\mathscr{D}'} \circ
{h_{\mathscr{D}}}^{-1} \mid \mathscr{D}, \mathscr{D}' \,
\textnormal{generic triples}\} \subset \textnormal{Aut}(H(L;K) \otimes
R).$$ This is a subgroup of the group of automorphisms of the
$\mathcal{R}$--module $H(L;K)\otimes \mathcal{R}$. It corresponds to
the subgroup generated by all morphisms associated to changes in
choices of data $\mathscr{D}$.

\subsubsection{General deformation theory} \label{sb:gnrl-def} The
previous considerations fit into the general framework of classical
deformation theory of algebras (see for example
Gerstenhaber~\cite{Gerst}). Algebras in this section are assumed to be
associative, unital, but not necessarily commutative.

Let $(A,\cdot)$ be an algebra over the commutative ring $R$ (which is
also a $K$-algebra). We denote by $ - \cdot -$ the product of $A$. A
deformation of $A$ is a structure of an algebra over $R[t]$ on the
module $A \otimes_R R[t]$
$$(A \otimes_R R[t]) \otimes_{R[t]} (A \otimes_R R[t]) \longrightarrow
(A \otimes_R R[t]), \quad x \otimes y \longmapsto x * y,$$ which
satisfies the following conditions:
\begin{enumerate}
  \item $A \otimes_R R[t]$ endowed with $*$ is an (associative unital)
   algebra over $R[t]$.
  \item $1 \in A$ continues to be the unit for $*$.
  \item $*$ reduces to product $\cdot$ for $t=0$.
\end{enumerate}
Sometimes insead of denoting the product on $A$ by $x \cdot y$ and a
deformation of it by $x*y$ we will write $m_0(x,y)$ and $m(x,y)$
respectively.

We will also need a graded version of the story. Our algebra $A =
\bigoplus_{k\geq 0} A^k$ will be cohomologically graded and the ring
$R$ should be regarded as having degree $0$ with respect to $A$, i.e.
$R$ is mapped by a morphism of rings to the center of $A$ in degree
$0$, $R \longrightarrow Z(A^0) \subset A^0$. Let $d \in \mathbb{Z}$.
We will consider deformations $*$ of $A$ where the formal parameter
$t$ has degree $|t|=d$. We denote the set of such deformations by
$\widetilde{\textnormal{Def}}_d(A)$.  Denote by
$\textnormal{Iso}_d(A)$ the group consisting of all $R[t]$-linear,
degree preserving, module isomorphisms $\phi : A \otimes_R R[t]
\longrightarrow A \otimes_R R[t]$ that have the following form:
$$\phi(x) = x + \phi_1(x)t + \phi_2(x)t^2 + \cdots, \;\; \forall x \in
A, \quad \textnormal{where } \phi_k: A^* \longrightarrow A^{*-dk}.$$
Two deformations $m', m'' \in \widetilde{\textnormal{Def}}_d(A)$ are
said to be equivalent if they are related by an element of
$\textnormal{Iso}_d(A)$, i.e. there exists $\phi \in
\textnormal{Iso}_d(A)$ such that $\phi(m''(x,y)) = m'(\phi(x),
\phi(y))$ for every $x, y \in A \otimes_R R[t]$. We denote by
$\textnormal{Def}_d(A) = \widetilde{\textnormal{Def}}_d(A) /
\textnormal{Iso}_d(A)$ the set of equivalence classes of deformations
of $A$. Similarly, when grading is not relevant we have
$\widetilde{\textnormal{Def}}(A)$, $\textnormal{Iso}(A)$ and
$\textnormal{Def}(A) = \widetilde{\textnormal{Def}}(A) /
\textnormal{Iso}(A)$.

We will also use a slight modification of this construction. Assume
$G\subset \textnormal{Iso}_d(A)$ is a subgroup. We then denote by
$$\textnormal{Def}^{G}_d(A)=\widetilde{\textnormal{Def}}_d(A) / G$$
the equivalence classes of deformations of $A$ with respect to
conjugation by elements of $G$.

\subsubsection{The main example} \label{sbsb:main-ex} Let $L$ be a
monotone Lagrangian and $\mathcal{R} = R[t]$ as explained at the
beginning of~\S\ref{s:defo} so that $L$ is $\mathcal{R}$-wide.  Let
$A$ be the singular homology algebra of $L$ (tensored with $R$), $A =
H(L;K) \otimes_K R$, endowed with the intersection product $\cdot$. We
grade $A$ cohomologically, i.e. we put $A^i = H_{n-i}(L;K) \otimes_K
R$ and here the degree of $t$ is $N_{L}$ (note that the unity $1 \in
A^0$ corresponds in the homological notation to $[L]$).

Next consider the quantum homology $QH(L;\mathcal{R})$. For
convenience, we grade it here cohomologically too, namely
$QH^{i}(L;\mathcal{R}) := QH_{n-i}(L; \mathcal{R})$ and whenever
working with $QH^*$ we change the degree of $t$ to be $N_L$ rather
then $-N_L$.

Recall the set of isomorphisms $\mathcal{K}$ introduced at the
beginning of~\S\ref{s:defo}. Pick $h \in \mathcal{K}$. By transferring
the quantum product $*$, via $h$, from $QH(L; \mathcal{R})$ to
$A\otimes_R R[t] = H(L;K)\otimes_K R[t]$ we obtain a deformation
$*_h\in \widetilde{\textnormal{Def}}_{N_{L}}(A)$ of the intersection
product $- \cdot -$. This is so because of point~(\ref{i:unity}) of
Proposition~\ref{p:class-K} and because the quantum product $*$
operation is obviously a deformation of the intersection product $-
\cdot -$ operation on the chain level.

It follows from point~(\ref{i:hdhd'}) of Proposition~\ref{p:class-K}
that $\mathscr{G}_{L}\subset \textnormal{Iso}_{N_{L}}(A)$ and so we
have quotient maps:
$$\widetilde{\textnormal{Def}}_{N_{L}}(A)
\stackrel{\Psi_{1}}{\longrightarrow}
\textnormal{Def}^{\mathscr{G}_{L}}_{N_{L}}(A)
\stackrel{\Psi_{2}}{\longrightarrow} \textnormal{Def}_{N_{L}}(A).$$ We
denote
$$*^{\mathscr{G}}_L=\Psi_{1}(\ast_{h}) \ \mathrm{and}\
\ast_{L}=\Psi_{2}(\ast^{\mathscr{G}}_{L})~.~$$ By the preceding
discussion neither $*^{\mathscr{G}}_L$ nor $\ast_{L}$ depend on the
choice of $h \in \mathcal{K}$. In other words, $(QH(L;\mathcal{R}),*)$
provides us with a well defined {\em class} of deformations of the
classical ring $(H(L)\otimes R, \cdot)$.

Notice that $\ast_{L}$
belongs to a purely algebraic object: indeed
$\textnormal{Def}_{N_{L}}(A)$ only depends on the algebra structure of
$A=H(L;R)$ and not on any properties of the specific Lagrangian
embedding $L\subset M$.  By contrast,
$\textnormal{Def}^{\mathscr{G}_{L}}_{N_{L}}(A)$ depends on this
embedding because $\mathscr{G}_{L}$ is strongly depended on it - for
instance, if $L$ is exact, then $\mathscr{G}_{L}$ reduces to the
identity element.

\subsection{Invariant polynomials in the structural constants of the
  quantum product} \label{subsubsec:struct-const} We pursue the
discussion in~\S\ref{sbsb:main-ex}. In particular, we continue to
write the various structures with cohomological grading. We use the
same assumptions on $K$, $R$ and $R[t]$ as at the beginning
of~\S\ref{s:defo}.  The main examples we have in mind are when $K$ is
a field, or when $K=\mathbb{Z}$. Moreover, for simplicity we will also
assume that $H(L;K)$ is a free $K$-module. (If this is not the case we
can always replace $H(L;K)$ by its free part over $K$,
$H(L;K)_{\textnormal{free}}$, and the discussion below continues to
hold with minor modifications.) As for the $K$--algebra $R$ we will
assume for simplicity that $K$ is {\em embedded} in $R$.

To shorten notation we set $\widetilde{A}=A\otimes_{R}
R[t]=H(L;K)\otimes_{K} R[t]$. Note that $\widetilde{A} = H(L;K) \oplus
(H(L;K) \otimes_K tR[t])$. Denote by $\textnormal{pr}_q: \widetilde{A}
\longrightarrow H(L;K) \otimes_K tR[t]$ the projection on the second
factor.  Put
$$\mathcal{H} = \textnormal{hom}_K^0(H(L;K) \otimes H(L;K),
\widetilde{A}), \quad \mathcal{H}_q = \textnormal{hom}_K^0(H(L;K)
\otimes H(L;K), H(L;K) \otimes_KtR[t]),$$ where $\textnormal{hom}_K^0$
stands for degree preserving $K$--linear homomorphisms. For degree
reasons both $\mathcal{H}$ and $\mathcal{H}_q$ are free $R$--modules
of finite rank. The projection $\textnormal{pr}_q$ induces a map $q:
\mathcal{H} \longrightarrow \mathcal{H}_q$. As explained above an
element $h\in\mathcal{K}$ induces an associative product: $\ast_{h}:
\widetilde{A}\otimes\widetilde{A}\to \widetilde{A}$. In particular we
also get an element which we still denote $*_h \in \mathcal{H}$. We
denote its image in $\mathcal{H}_q$ by $q(*_h)$.

Let $U$ be a finite rank free $K$--module and $V = U \otimes_K R$. By
a polynomial on $V$ with coefficients in $K$ we mean a function $P: V
\longrightarrow R$ for which there is a basis of $U$, $u_1, \ldots,
u_l$ such that $P$ can be written as a polynomial with coefficients in
$K$ in the $R$--basis $u_1 \otimes 1, \ldots, u_l \otimes 1$ of $V$.
Clearly this notion does not depend on the choice of the basis for
$U$. We denote these polynomials by $K[V]$.

Consider now polynomials $P \in K[\mathcal{H}]$ (where $\mathcal{H}$
is written as $U \otimes_K R$ in an obvious way). The purpose of this
section is to discuss polynomials $P$ which have the property that:
$$P(*_h) = P(*_{h'}) \; \textnormal{ for every }
h,h' \in \mathcal{K}.$$ Such polynomials will be called invariant
polynomials. Next let $\sigma \in \textnormal{Aut}^0_K H(L;K)$ be a
degree preserving automorphism. Clearly each such automorphism
$\sigma$ induces an automorphism $\sigma_{\mathcal{H}} \in
\textnormal{Aut}_K(\mathcal{H})$. We say that a polynomial $P$ is a
symmetric polynomial invariant if $P$ is an invariant polynomial and
moreover $P$ remains invariant under composition with
$\sigma_{\mathcal{H}}$ for every $\sigma \in \textnormal{Aut}^0_K
H(L;K)$. We will be particularly interested in invariant polynomials
(symmetric or not) that capture information on the quantum part of the
product, namely polynomials $P$ that factor through $q:\mathcal{H} \to
\mathcal{H}_q$, i.e. there exists a polynomial $Q \in
K[\mathcal{H}_q]$ such that $P(*_h) = Q(q(*_h))$.  We will call them
Lagrangian quantum polynomials.  Finally, we will be interested also
in universal invariant polynomials for $L$, namely those that do not
depend on the particular Lagrangian embedding of $L$.

We will now describe these notions in detail by using coordinates.
While the notation in coordinates might appear heavy, it is more
useful for applications and computations.

Fix a basis $\mathbf{a}=(a_{i})_{i\in I}$ for $H^{\ast}(L;K)$ and put
$\epsilon(i,j,s)=|a_{i}|+|a_{j}| - |a_{s}|$.  We will assume further
that the basis $\mathbf{a}$ is ordered in such a way that the first
element is $a_{0}=1\in H^{0}(L;K)$, the next elements form an ordered
basis of $H{^1}(L;K)$ the ones after that form a basis for
$H^{2}(L;K)$ etc.  Obviously, any graded change of basis leaves the
$\epsilon(i,j,s)$ invariant.

Any associative product $\ast\in
\widetilde{\textnormal{Def}}_{N_{L}}(A)$ is characterized by constants
$k^{i,j}_{s}\in R$ given by:
\begin{equation}\label{eq:struct_cst}
   a_{i}\ast  a_{j}=\sum_{\{\ s\  | \  \ N_{L} \textnormal{ divides }
     \epsilon(i,j,s) ,\ \epsilon(i,j,s)\geq 0\}}
   k^{i,j}_{s}a_{s}t^{\epsilon(i,j,s)/N_{L}}~.~
\end{equation}
The fact that the group $\mathscr{G}_{L}$ is in general non-trivial
implies that for a product $\ast_{h}$ associated to an element
$h\in\mathcal{K}$, the constants $k^{i,j}_{s}$ depend on $h$ (and thus
on $\mathscr{D}=(f,(\cdot,\cdot),J)$) and not only on
$\ast^{\mathscr{G}}_{L}$.  At the same time in the case of quantum
homology of the ambient manifold $M$ the structural constants of the
quantum product are in fact triple Gromov-Witten invariants (see
e.g.~\cite{McD-Sa:Jhol-2}).  This suggests that even if these
structural constants are not themselves invariant in our Lagrangian
setting, it might very well happen that -- as a ``next best case'' --
there exist invariants that are polynomial expressions in these
constants.

Define:
\begin{equation}\label{eqn:pick}
   \mathcal{I}_{L}= \{(i,j,s)\in I\times I\times I \  \ |\  \
   \epsilon(i,j,s)\geq 0\ , \ N_{L} \textnormal{ divides }
   \epsilon(i,j,s)\}~.~
\end{equation}
Notice that the number of elements of $\mathcal{I}_{L}$ only depends
on $H(L;K)$ and $N_{L}$ (and not on the actual basis $\mathbf{a}$).
We let $K[z_{r} ; r\in \mathcal{I}_{L}]$ be the polynomial ring with
coefficients in $K$ and variables $z_{r}, r=(i,j,s)\in
\mathcal{I}_{L}$.  Given any polynomial $P\in K[z_{r} ; r\in
\mathcal{I}_{L}]$ and any product $\ast\in
\widetilde{\textnormal{Def}}_{N_{L}}(A)$ we can evaluate $P$ on the
structural constants associated to this product in the basis
$\mathbf{a}$: we assign to $z_{(i,j,s)}$ the value $k^{i,j}_{s}\in R$.
We denote the value of $P$ computed in this way by $P(\ast
;\mathbf{a})\in R$ and we call it the value of $P$ on the product
$\ast$ in the basis $\mathbf{a}$.

\begin{dfn} \label{def:poly_inv} Fix a smooth closed and oriented
   manifold $L_0$ endowed with a spin structure. Let $N \geq 2$ be an
   integer. Let $i: L_0 \hooklongrightarrow M$ be an $R$--wide
   monotone Lagrangian embedding with minimal Maslov number $N$. Put
   $L=i(L_0)$.
   \begin{enumerate}
     \item[i.] A {\em Lagrangian polynomial invariant} for $L$ is a
      polynomial $P\in K[z_{r} ; r\in \mathcal{I}_{L}]$ so that for
      every $h \in \mathcal{K}$, the value $P(*_h;\mathbf{a})$ is
      independent of $h$ for any basis $\mathbf{a}$ (in other words
      $P(*_h;\mathbf{a})$ only depends on $P$, $*^{\mathcal{G}}_L$ and
      $\mathbf{a}$).
     \item[ii.] A universal Lagrangian polynomial invariant of $L_0$
      is a polynomial $P$ as in point i which has the property that it
      is a polynomial invariant for every wide Lagrangian embedding
      $i: L_0 \hooklongrightarrow M$ (in any $M$) as above.
   \end{enumerate}
   Polynomials as above are called {\em symmetric} if the value
   $P(*_h;\mathbf{a})$ is independent of the basis $\mathbf{a}$. They
   are called {\em quantum} if they depend only on $z_{(i,j,s)}$ with
   $\epsilon(i,j,s)>0$.
\end{dfn}

\begin{ex}\label{eq:trivial} We start with the trivial example.
   Notice that the structural constants $k^{i,j}_{s}$ for
   $\epsilon(i,j,s)=0$ are simply the structural constants of the
   algebra $A$ and thus do not depend on $\mathscr{D}$. Thus, any
   polynomial $P\in K[z_{r}: r\in \mathcal{I}_{L}, \ \epsilon (r)=0]$
   is invariant (and even universal) once it does not depend on the
   chosen basis.

   From now on we will refer to this example as being \emph{trivial}
   and we will eliminate it from any further discussion by focusing on
   quantum polynomial invariants.
\end{ex}

\begin{ex}\label{eq:quad-coeff}
   For this example it is relevant to work with $K=\mathbb{Z}$.
   Furthermore, we assume $N_{L}=2$ and put $l = \textnormal{rank\,}
   H^{1}(L;\mathbb{Z})$. This means that for a basis $\mathbf{a}$ as
   before, the first element is $a_{0}=1$ and the next elements,
   $a_{1},\ldots, a_{l}$, form a basis of $H^{1}(L;\mathbb{Z})$. We
   consider the elements of $(i,j,0)\in \mathcal{I}_{L}$ with $1 \leq
   i, j \leq l$ (hence $\epsilon(i,j,0)=2$) and for each such element
   we define polynomials:
   $$P_{ij}=z_{(i,j,0)}, \quad
   \bar{P}_{ij}=P_{ij}+P_{ji}, \quad P_{\Delta}=-det
   (\bar{P}_{ij})~.~$$ The point of this example is to discuss the
   invariance of these polynomials.

   Let $h\in \mathcal{K}$ with associated product $\ast_{h}$. Then we
   have:
   \begin{itemize}
     \item[i.]  $P_{ij}(\ast_{h};\mathbf{a})=c_{ij}\in R$ where
      $a_{i}\ast_{h} a_{j}=a_{i}\cdot a_{j}+c_{ij}t$ (recall that we
      are using cohomological notation);
     \item[ii.] $\bar{P}_{ij}(\ast_{h};\mathbf{a})=a_{ij}\in R$ where
      $a_{i}\ast_{h} a_{j}+a_{j}\ast_{h} a_{i}=a_{ij}t$ (compare with
      ~\eqref{eq:discr-1} from~\S\ref{s:qforms});
     \item[iii.] $P_{\Delta}(\ast_{h};\mathbf{a})=\Delta$ with
      $\Delta$ the discriminant from~\S\ref{s:qforms}.
   \end{itemize}
   This shows that the polynomials $\bar{P}_{ij}$ are universal
   quantum invariants because by evaluation they provide the
   coefficients of the quadratic form discussed in~\S\ref{s:qforms},
   and this quadratic form is invariant with respect to $\mathscr{D}$.
   However, the $\bar{P}_{ij}$ are not symmetric polynomials since the
   coefficients of the quadratic form depend on the basis in which it
   is written. On the other hand, for obvious reasons, $P_{\Delta}$ is
   a universal, symmetric, Lagrangian quantum invariant.

   Note that in contrast to $\bar{P}_{ij}$, the polynomials $P_{ij}$
   are not quantum invariants, as the example of the $2$-dimensional
   Clifford torus in ${\mathbb{C}}P^2$ shows.
\end{ex}

\begin{rem} \label{rem:uniqueness-discriminant} For $K = \mathbb{Z}$,
   the polynomial $P_{\Delta}$ is (up to composition with a polynomial
   of one variable) the \emph{only} universal quantum invariant that
   depends only on the $z_{(i,j,0)}$'s with $\epsilon(i,j,0)=2$.
   Indeed, any polynomial quantum invariant depending on the variables
   $z_{(i,j,0)}$'s with $\epsilon(i,j,0)=2$ is a polynomial in the
   $\bar{P}_{ij}$'s. In other words, it is a polynomial in the
   coefficients of the quadratic form $\varphi$ defined
   in~\eqref{eq:a2-R} of~\S\ref{s:qforms}. By definition, the values
   of this polynomial in the coefficients of $\varphi$ should be
   independent of the basis in which $\varphi$ is expressed. On the
   other hand it is known since the work of Hilbert~\cite{Hilbert}
   that the ring of polynomial invariants of a quadratic form is
   generated by a single element which can be taken to be the
   discriminant.
\end{rem}

\subsubsection{The variety of algebras} \label{subsubsec:var-alg}
We describe here a more conceptual point of view on the invariant
polynomials introduced in the previous section. We continue to use the
notation from \S\ref{sbsb:main-ex} and \S\ref{subsubsec:struct-const}
and, in particular, continue to use cohomological notation. A survey
of deformation theory from this perspective can be found in
\cite{Makh} for instance.

We begin by noticing that the set
$\widetilde{\textnormal{Def}}_{N_{L}}(A)$ of deformations of the
intersection product on $A=H(L;R)$ has the structure of an algebraic
set.  Indeed, fix a basis $\mathbf{a}=(a_{i})_{i\in I}$ for $H(L;K)$.
The structural constants $k^{i,j}_{s}\in R$ associated to any element
$\nu\in \widetilde{\textnormal{Def}}_{N_{L}}(A)$ by writing the
product structure in the basis $\mathbf{a}$ as
in~\eqref{eq:struct_cst} verify a series of algebraic equations.
First, we have linear equations reflecting the fact that the product
is graded:
\begin{equation}\label{eq:linear1} k^{i,j}_{s}=0 \
   \mathrm{if}\
   \epsilon(i,j,s)\leq 0 \ \ \mathrm{or}\ \ N_{L} \
   \mathrm{does\ not\
     divide}\ \ \epsilon(i,j,s).
\end{equation}
Next, the existence of a unit translates to:
\begin{equation}\label{eq:linear2}
   k^{0,i}_{j}=k^{i,0}_{j}=\delta_{i,j}, \forall i,j\in I.
\end{equation}
The fact that the operation is a deformation of the intersection
product in $A$ gives:
\begin{equation}\label{eq:linear3}
   k^{i,j}_{s}=v^{i,j}_{s} \ \mathrm{if} \ \epsilon(i,j,s)=0,
\end{equation}
where $v^{i,j}_{s}$ are the structural constants of the intersection
product in $A$. Finally we have some quadratic equations that reflect
the associativity of the product:
\begin{equation}\label{eq:quadr}
   \sum_{s}k^{i,j}_{s}k^{s,l}_{m}=\sum_{r}k^{j,l}_{r}k^{i,r}_{m} \quad
   \forall \; i,j,l,m\in I~.~
\end{equation}
Consider variables $z^{i,j}_{s}\in R$ with $i,j,s\in I$ and define the
algebraic set $\mathcal{V}_{N_{L}}(A)$ by demanding that the
$z^{i,j}_{s}$ verify~\eqref{eq:linear1},~\eqref{eq:linear2}
and~\eqref{eq:quadr}. Clearly this set is independent of the basis
$\mathbf{a}$. Denote by $\mathcal{V}_{N_{L}}(A;\mathbf{a})$ the
algebraic set obtained by demanding that the $z^{i,j}_{s}$ verify
additionally~\eqref{eq:linear3}.  We have an identification
$$\Psi_{\mathbf{a}}:\widetilde{\textnormal{Def}}_{N_{L}}(A)
\to\mathcal{V}_{N_{L}}(A;\mathbf{a})\subset
\mathcal{V}_{N_{L}}(A)~.~$$ The group $\mathscr{G}_{L}$ acts on
$\mathcal{V}_{N_{L}}(A)$ and this action restricts to an action on
$\mathcal{V}_{N_{L}}(A;\mathbf{a})$ for each basis $\mathbf{a}$.

Given that $R$ is a $K$-algebra, there is a canonical embedding
$K[z_{r};r \in \mathcal{I}_{L}]\to R[z_{r}; r\in \mathcal{I}_{L}]$ so
that to any polynomial in $K[z_{r};r\in \mathcal{I}_{L}]$ we can
associate one in $R[z_{r};r\in \mathcal{I}_{L}]$.  In this language, a
Lagrangian polynomial invariant is a polynomial in $K[z_{r}; r\in
\mathcal{I}_{L}]$ whose associated regular function on
$\mathcal{V}_{N_{L}}(A)$ is constant on the $\mathscr{G}_{L}$-orbit of
$\ast_{h}\in\mathcal{V}_{N_{L}}(A;\mathbf{a})$ for all $h\in
\mathcal{K}$ and such that this holds for each basis $\mathbf{a}$. It
is symmetric if the value of the respective constant is independent of
the basis $\mathbf{a}$.

\begin{rem} \label{rem:identities} An important point which is an
   immediate consequence of the discussion in this section is that two
   Lagrangian invariant polynomials $P_{1}, P_{2}\in K[z_{r} : r\in
   \mathcal{I}_{L}]$ as defined in Definition \ref{def:poly_inv} can
   be equal, $P_{1}=P_{2}$, as regular functions on
   $\mathcal{V}_{N_{L}}(A)$ without being the same polynomials: the
   polynomial expressions of $P_{1}$ and $P_{2}$ can be different but,
   due to the relations among the variables $k^{i,j}_{s}$, the
   respective regular functions may agree.  Notice also that if we
   have an equality $P_{1}=P_{2}$ over $\mathcal{V}_{N_{L}}(A)$ for
   two polynomials in $K[z_{r}; r]$ and we know that just one of the
   polynomials is invariant, then the second one is necessarily also
   invariant.
\end{rem}

\subsection{Hochschild cohomology} \label{sb:hochschild} The classical
algebraic approach to deformation theory is via Hochschild cohomology.
We recall it here.  We use the standard Hochschild cohomology theory
for associative algebras \cite{Gerst}. We start with a
brief description of this classical construction.

Let $A$ be a graded algebra over a commutative ring $R$. As before we
view $R$ as having degree $0$ with respect to $A$, i.e. $R$ is mapped
by a morphism of rings to the center of $A$ in degree $0$, $R
\longrightarrow Z(A^0) \subset A^0$.

The Hochschild complex of $A$ (with coefficients in $A$) is defined by
$$C^k(A;A) = \textnormal{Hom}_{R} (A^{\otimes k}, A)$$ endowed with the
differential $d: C^k(A;A) \longrightarrow C^{k+1}(A;A)$:
\begin{equation} \label{eq:HH-diff}
   \begin{aligned}
      df(a_1 \otimes \cdots \otimes a_{k+1}) = & a_1 f(a_2 \otimes
      \cdots \otimes a_{k+1}) \\
      & + \sum_{i=1}^k (-1)^i f(a_1 \otimes \cdots \otimes (a_i
      a_{i+1}) \otimes \cdots
      \otimes a_{k+1}) \\
      & + (-1)^{k+1} f(a_1 \otimes \cdots \otimes a_k)a_{k+1}.
   \end{aligned}
\end{equation}
The homology of this cochain complex is called the Hochschild
cohomolgoy of $A$ (with coefficients in $A$) and is denoted by
$HH^*(A;A)$. The second $A$ here should be regarded as the
``coefficients module''. It can be replaced by any $A$-module $M$
yielding $HH^*(A;M)$, but we will not need this in the sequel.

We incorporate the grading into this construction (without  modifying the
formula for the differential). Simply consider for every
$k, l \in \mathbb{Z}$ the following submodule $$C^{k,l}(A;A) =
\textnormal{Hom}_R^l(A^{\otimes k}, A) \subset C^k(A;A),$$ where
$\textnormal{Hom}_R^l$ stands for $R$-linear homomorphisms that shift
degree by $l$. Here, this means that $f \in C^{k,l}(A;A)$ if $f$ is
$R$-linear and for every $k$ homogeneous elements $a_1, \ldots, a_k \in
A$ we have $$|f(a_1 \otimes \cdots \otimes a_k)| = |a_1 \otimes \cdots
\otimes a_k| + l = \sum_{i=1}^k |a_i| + l.$$ Clearly $d(C^{k,l})
\subset C^{k+1,l}$. Put $$HH^{k,l}(A;A) = \frac{\ker (d|_{C^{k,l}})}
{d(C^{\scriptscriptstyle k-1,l})}.$$

Classical deformation theory provides a map:
\begin{equation}
   T^{1}: \textnormal{Def}_d(A) \longrightarrow HH^{2,-d}(A;A) ~.~
\end{equation}
The definition of $T^1$ is straightforward. Given a deformation $* \in
\widetilde{\textnormal{Def}}_d(A)$, we can write $$x*y = x \cdot y +
m_1(x,y)t + \cdots, \quad \forall x, y \in A.$$ A simple computation
shows that $m_1 \in \textnormal{Hom}^{-d}_R(A \otimes A, A)$ is a
Hochschild cycle (this is due to the associativity of $*$), hence we
have an element $[m_1] \in HH^{2,-d}(A;A)$. Moreover, equivalent
deformation $*' \sim *$ give rise to cohomologous cycles: $[m'_1] =
[m_1] \in HH^{2,-d}(A;A)$.  Thus setting $T^1([*]) = [m_1]$ provides a
well defined map.

\subsubsection{Quadratic forms and Hochschild cohomolgoy}
\label{sb:quad-HH}

Let $A$ be an $R$--algebra and $S \subset A$ an $R$-submodule. Denote
by $Q^2(S,R)$ the space of $R$--valued quadratic forms $\varphi:S
\longrightarrow R$. Put $S_{\sqrt{0}} = \{s \in S \mid s \cdot s = 0\}
\subset S$, and consider the restriction map
$\textnormal{rest}:Q^2(S,R) \longrightarrow
\textnormal{Func}(S_{\sqrt{0}}\, , R)$ to the space of $R$--valued
functions on the set $S_{\sqrt{0}}$.  Denote by $Q^2_0(S,R)$ the image
of this map.

Assume from now on that our graded $R$--algebra $A$ is non-trivial
only in degrees between $0$ and $n$. Moreover assume that $A^0=R$.
Then we have a map:
\begin{equation} \label{eq:HH-Q_0} \Theta : HH^{2,-d}(A;A)
   \longrightarrow Q^2_0(A^{d/2}, R),
\end{equation}
defined as follows. Let $\alpha \in HH^{2,-d}(A;A)$. Choose a cocycle
$f_{\alpha} \in C^{2,-d}(A;A)$ so that $[f_{\alpha}] = \alpha$ and
view $f_{\alpha}$ as a map $f_{\alpha}: A\otimes A \longrightarrow A$
of degree $-d$.  Consider the quadratic form $\widehat{f}_{\alpha}:
A^{d/2} \longrightarrow R$, defined by $\widehat{f}_{\alpha}(a) := f(a
\otimes a) \in A^0 = R$.  Finally, define $\Theta(\alpha) =
\textnormal{rest}(\widehat{f}_{\alpha})$, where $\textnormal{rest}$ is
the restriction map $Q^2(A^{d/2},R) \longrightarrow Q^2_0(A^{d/2},R)$.

We claim that the map $\Theta$ is well defined. To see this it is
enough to show that if $f=dg$, where $g \in C^{1,-d}(A;A)$ then $f(a
\otimes a)=0$ for every $a \in A^{d/2}$ with $a \cdot a=0$. Indeed,
let $a \in (A^{d/2})_{\sqrt{0}}$. By the definition of the Hochschild
differential we have $$f(a\otimes a) = dg (a\otimes a) = a\cdot g(a)
-g(a \cdot a) + g(a) \cdot a.$$ But $g:A \longrightarrow A$ has degree
$-d$ hence $g(a) \in A^{-d/2}=0$ and by assumption we also have $a
\cdot a =0$. It follows that $f(a \otimes a)=0$. This proves that
$\Theta$ is well defined.

Consider now the composition
\begin{equation} \label{eq:def-quad} \Gamma : \textnormal{Def}_d(A)
   \longrightarrow Q^2_0(A^{d/2},R), \quad \Gamma = \Theta \circ T^1.
\end{equation}
In case $(A^{d/2})_{\sqrt{0}} = A^{d/2}$, $\Gamma$ assigns to every
graded deformation equivalence class of $A$ a quadratic form on
$A^{d/2}$.

\begin{ex}\label{ex:Hochshild-quadr}
   Let $L^n \subset M^{2n}$ be a monotone Lagrangian with $N_L=2s$ and
   a non-empty wide variety $\mathcal{W}$. Take
   $$R=\mathcal{O}(\mathcal{W}), \quad \textnormal{and } \;
   A^* = H_{n-*}(L;\mathbb{C}) \otimes \mathcal{O}(\mathcal{W}).$$ As
   explained in~\S\ref{sbsb:main-ex}, $Q^+H(L;\mathcal{W}) = QH(L;
   R[t])$ gives rise to a class of deformations $*_L \in
   \textnormal{Def}_{d}(A)$. Applying the map $T^1$ we obtain an
   invariant $T^1(*_L) \in HH^{2,-d}(A;A)$.

   Next, applying the map $\Gamma$ to $*_L$ we obtain a quadratic form
   $\Gamma(*_L) \in Q^{2}_{0}(A^{s}, R)$ on $A^{s}_{\sqrt{0}} \cong
   (H_{n-s}(L;\mathbb{C}) \otimes
   \mathcal{O}(\mathcal{W}))^{\sqrt{0}}$ with values in
   $\mathcal{O}(\mathcal{W})$. A straightforward computation shows
   that this $\Gamma(*_L)$, when restricted to
   $H_{n-s}^{\sqrt{0}}(L;\mathbb{C})$, coincides with the quadratic
   form $\varphi^{s}_{_{\mathcal{W}}}:H_{n-s}^{\sqrt{0}}(L;\mathbb{C})
   \longrightarrow \mathcal{O}(\mathcal{W})$ constructed
   in~\S\ref{subs:highNL-quad}.
\end{ex}

\begin{ex}\label{ex:HH-torus}
   We have a particular interest in the free graded exterior algebra
   $\Lambda_{n}(R)$ which is generated as algebra by $n$ generators
   $a_{1},\ldots, a_{n}\in \Lambda^{1}_{n}(R)$ which we will think of
   as the the singular cohomology of the $n$-torus (with coefficients
   in $R$).  We put $A=\Lambda_{n}(R)$, $d=2$ and consider the
   resulting map:
   $$\Gamma: \textnormal{Def}_{2}(A)\to Q^2(A^{1},R)~.~$$
\end{ex}

\begin{lem}\label{lem:def-quadr-torus} Let $K$ be a
   field of characteristic $0$ or $\mathbb{Z}$. For
   $A=\Lambda_{n}(K)\otimes R$, $n\geq 2$, the map $\Gamma$ is an
   isomorphism.
\end{lem}

\begin{proof}
   Let $* \in \widetilde{\textnormal{Def}_{2}}(A)$. The quadratic form
   $\Gamma(\ast)$ can be described as follows. Pick a basis
   $a_{1},\ldots a_{n}\in A^{1}$ and notice that as $\Lambda_{n}(K)$
   is an exterior algebra we have
   \begin{equation}\label{eq:clifford-alg}a_{i}\ast a_{j}+a_{j}\ast
      a_{i}= \alpha_{ij} t, \textnormal{ for some } \alpha_{ij} \in R.
   \end{equation}
   The quadratic form in question is
   $$\Gamma(*)(X_{1}a_{1}+\ldots X_{n}a_{n})=
   \frac{1}{2}\sum_{i,j} \alpha_{ij}X_{i}X_{j}~.~$$

   Recall now the definition of the Clifford algebra. Let $Q =
   (q_{ij})$ be a symmetric $n \times n$ matrix with coefficients in
   $R$. The Clifford algebra associated to $Q$ is by definition
   $$\textnormal{Cliff}(Q) = \bigl(\Lambda_n(R) \otimes_R R[t]\bigr)
   \big / I,$$ where $|t|=2$ and $I$ is the ideal generated by the
   relations
   $$a_i a_j + a_j a_i = 2q_{ij} t.$$
   For degree reasons this algebra is a deformation of $\Lambda_n(R)$
   endowed with the standard exterior product structure.

   Coming back to our situation we see that $(A[t], *)$ can be
   described as the Clifford algebra associated to the quadratic form
   $\Gamma(*)$. (This has been previously remarked by Cho
   in~\cite{Cho:products}.) More precisely, if we take $Q$ to be the
   matrix corresponding to $\Gamma(*)$ in the basis $a_1, \ldots, a_n$
   (i.e.  take $q_{ij} = \frac{1}{2} \alpha_{ij}$) then the map
   $$c_{{\scriptscriptstyle Q},*}:
   \textnormal{Cliff}(Q) \longrightarrow (A[t], *), \quad
   \textnormal{induced by } c_{{\scriptscriptstyle Q},*}(a_i) = a_i$$
   is an isomorphism of algebras. This follows by a simple dimension
   checking, degree by degree. This shows that the map $\Gamma$ is
   surjective.

   To show that $\Gamma$ is also injective assume that
   $\Gamma(*)=\Gamma(*')$ for some $*, *' \in
   \textnormal{Def}_{2}(A)$. Let $Q$ be the $n \times n$ matrix
   corresponding to $\Gamma(*)=\Gamma(*') \in Q^2(A^1,R)$ in the basis
   $a_1, \ldots, a_n$. We have an isomorphism of algebras
   $$\xi = c_{{\scriptscriptstyle Q},*'}
   \circ c_{{\scriptscriptstyle Q},*}^{-1} : (A[t], *) \longrightarrow
   (A[t], *').$$ This composition is the identity on $A^{\leq 1}$.
   Together with the fact that $\xi$ is an isomorphism of algebras
   this implies immediately that $\xi$ is an equivalence of
   deformations and finishes the proof.
\end{proof}

An interesting consequence of this Lemma is obtained if we assume that
for some wide $n$-dimensional Lagrangian torus $L \subset M$ the group
of geometric equivalences $\mathscr{G}_{L}$ coincides with the group
of algebraic equivalences $\textnormal{Iso}_{2}(A)$, for
$A=H(T^{n};R)=\La_{n}(R)$. In this case we have two identifications
$$\textnormal{Def}^{\mathscr{G}_{L}}_{N_{L}}(A)
\to\textnormal{Def}_{N_{L}}(A)\to Q^{2}(A^{1},R)~.~$$ This implies
that any quantum polynomial invariant of $L$ has to agree as regular
function over $\mathcal{V}_{2}(A)$ with an expression that can be read
off from the coefficients of the quadratic form $\Gamma(\ast_{L})$.
Moreover, by Remark \ref{rem:uniqueness-discriminant} it verifies
$P=\mathcal{F}(P_{\Delta})$ for some polynomial of a singular
variable $\mathcal{F}$, where the equality here is over $\mathcal{V}_{2}(A)$.
Now such Lagrangian tori do exist. For example, it is easy to check
that for the $2$-dimensional Clifford torus $L \subset
{\mathbb{C}}P^2$ we do indeed have
$\mathscr{G}_{L}=\textnormal{Iso}_{2}(A)$. Another way to formalize
this is the following.

\begin{cor} \label{cor:torus}For the torus $\mathbb{T}^{2}$ any
   universal, Lagrangian quantum polynomial invariant $P$ agrees (as
   regular a function, in the sense of Remark~\ref{rem:identities})
   with a polynomial belonging to the subring generated by the
   discriminant.
\end{cor}
We expect this corollary to remain true also for higher dimensional
tori.

\begin{rem} \label{rem:jones-map} The information contained in the
   superpotential from \S\ref{sb:superpotential} can be encoded in a
   representation of the moduli spaces
   $\widetilde{\mathcal{M}}(\la,J)$ with values in the free loop space
   $\La (L)=L^{S^{1}}$. By taking the sum of the cycles represented by
   all these moduli spaces one gets a homology class $\alpha \in
   H_{\ast}(\La (L); K)$. There exists a well known isomorphism $\phi$
   constructed by Jones (see e.g.~\cite{Co-Jo:string}) between
   $H_{\ast}(\La(L), K)$ and the Hochschild cohomology
   $HH^{\ast}(C^{\cdot}(L);C^{\cdot}(L))$ where $C^{\cdot}(-)$ is the
   singular cochain complex. (Note however that one has to adjust the
   grading and the sign conventions here, for example see
   \cite{Fe-Th-Vi}.)  In favorable cases we also have an isomorphism
   $q : HH^{\ast}(C^{\cdot}(L);C^{\cdot}(L))\approx
   HH^{\ast}(H^{\cdot}(L;K), H^{\cdot}(L;K))$.  We point out here
   that, for instance, for Lagrangian tori if we project the class
   $q\circ \phi(\alpha)$ onto $HH^{2,-2}$ we obtain precisely the
   Hochschild cohomology class $T^{1}(\ast_{L})$ that is associated to
   the quantum product when viewed as deformation of the intersection
   product.
\end{rem}

%% file: enum.tex
\section{The discriminant and enumerative geometry}
\label{s:enum}
The purpose of this section is to use the machinery introduced before
to address the problem described at the beginning
of~\S\ref{sec:intro}. Thus we  consider one of the simplest,
non-trivial, enumerative problem in Lagrangian topology: counting
$J$-holomorphic disks $u:(D,\partial D)\to (M,L)$ passing through $3$
distinct points in $L$. As before, we will assume the closed
Lagrangian $L^{n}\subset (M^{2n},\omega)$ to be monotone with
$N_{L}\geq 2$.

Ideally, one would like to be able to estimate the number of disks in
question by separating them according to their homotopy class -- this
is where the wide varieties will be of help.

\subsection{Holomorphic disks through three points}
As in the introduction, let $P,Q,R\in L$ be three distinct points. We
are interested in the number of disks $u$ of Maslov index $2n$
passing, in order, through $P$, $Q$ and $R$. We will count these disks
with coefficients in $\mathcal{O}(\mathcal{W})$ where $\mathcal{W}$ is
one of our two wide varieties for $L$. This will lead to more refined
formulae than working only over $\La^{+}$.

To be more precise, given a class $\la \in H_2^D$ with $\mu(\la)=2n$
consider the map $$\mathbf{e}: \widetilde{\mathcal{M}}(\la,J)
\longrightarrow L \times L \times L, \quad \mathbf{e}(u) = \bigl(u(1),
u(e^{2\pi i/3}), u(e^{4 \pi i/3})\bigr),$$ where
$\widetilde{\mathcal{M}}(\la,J)$ is the moduli space of parametrized
$J$-disks in the homotopy class $\la$.  Note that both the source and
target of this map are $3n$ dimensional. Standard arguments show that
once we fix the points $P, Q, R$ then for generic $J$ the tuple
$(P,Q,R)$ is a regular value of this map and moreover the set
$\mathbf{e}^{-1}(P, Q, R)$ is finite (although the space
$\widetilde{\mathcal{M}}(\la, J)$ is not compact).  We associate to
each $u \in \mathbf{e}^{-1}(P, Q, R)$ a sign $\varepsilon(u;
P,Q,R)=\pm 1$ by comparing orientations via $\mathbf{e}$.  For
$\rho\in \textnormal{Hom}_0(H_2^D, \mathbb{C}^*)$ define now
\begin{equation}\label{eq:npqr}
   n_{PQR} = \sum_{\{\la \mid \mu(\la)=2n\}} \; \sum_{u \in
     {\mathbf{e}}^{-1}(P,Q,R)} \varepsilon(u; P,Q,R) \rho(\la)~.~
\end{equation}
The numbers $n_{PQR}(\rho)$ are neither invariant with respect to
$P,Q,R$ nor to $J$.

\subsubsection{Splitting polynomials}
The approach to estimating $n_{PQR}$ that we will discuss here is
based on the following simple idea: instead of showing that
$n_{-,-,-}$ is a numerical invariant (which it is not) show that there
exists a polynomial $Q\in K[\xi_{1},\ldots, \xi_{q}]$ and a subvariety
$\mathcal{W}\subset \textnormal{Hom}_0(H_2^D, \mathbb{C}^*)$ both
independent of $J, P, Q, R$ so that:
\begin{equation}\label{eq:recursion}n_{PQR}(\rho)=Q(\xi_{1},\ldots
   \xi_{q}) \ , \ \xi_{j}=\#_{\rho}(\mathcal{M}_{j}) \ , \ \forall
   \rho\in \mathcal{W}~.~
\end{equation}
Here $\mathcal{M}_{j}$ is a $0$-dimensional moduli space of pearl-like
trajectories involving only disks of Maslov index {\em at most
  $2n-2$}. Of course, the number $\#_{\rho}(\mathcal{M}_{j})$ depends
on the various data involved (e.g. Morse functions, metric and almost
complex structure), however the equations defining $\mathcal{M}_{j}$
are fixed.  By definition, the counting giving the $\xi_{j}$ is given
by:
$$\#_{\rho}(\mathcal{M}_{j})=\sum_{\la}\#
(\mathcal{M}_{j}(\la))\rho(\la)~.~$$ where $\mathcal{M}_{j}(\la)$ are
the configurations in $\mathcal{M}_{j}$ whose total homology class is
$\la$.

A polynomial $Q$ as above is called a \emph{splitting polynomial over
  $\mathcal{W}$}. Equation (\ref{eq:recursion}) can be interpreted as
an equality in $\mathcal{O}(\mathcal{W})$.

As we will see next such splitting polynomials often exist.  As
in~\S\ref{s:defo} we will assume here that $L$ admits a perfect Morse
functions but if this is not the case the minimal model technique
from~\S\ref{subsubsec:minimal-mod} may be used instead with minor
modifications.

\begin{thm} \label{thm:main-split} Monotone Lagrangians $L$ with
   $N_{L}\geq 2$ that are not rational homology spheres admit
   splitting polynomials $Q$ over their wide varieties
   $\mathcal{W}_{i}(L)$, $i=1,2$.  Moreover, there are such splitting
   polynomials that are universal in the sense that they are
   independent of the particular Lagrangian embedding of $L$.
\end{thm}

As we will see in the proof, this is a rather immediate reflection of
three facts: Poincar\'e duality in singular homology, the fact that
$Q^{+}H(L;\mathcal{W}_{i})$ -- as defined in~\S\ref{sb:reg-functions}
-- is a deformation of the singular homology algebra as discussed
in~\S\ref{s:defo}, and finally the fact that the quantum product is an
associative operation. Splitting polynomials are closely related to
the invariant polynomials in \S\ref{subsubsec:struct-const}. We prefer
to avoid making explicit use of invariant polynomials in the proof of the
theorem but we refer to Remark \ref{eq:assoc-poly} i. for further
discussion of this relationship.

\subsubsection{Proof of Theorem~\ref{thm:main-split}}
For simplicity we assume that $N_L=2$ (the arguments for $N_{L}>2$ are
similar). We will use in this proof homological notation.

Recall that we have assumed that $L$ admits a perfect Morse function,
hence $H_*(L;\mathbb{Z})$ is free. Fix a basis
$\mathbf{a}=(a_{0},a_{1},\ldots, a_{m})$ for $H_{\ast}(L;\mathbb{Z})$,
consisting of elements of pure degree and so that $a_0 =
[\textnormal{pt}]$, $|a_i| \leq |a_j|$ for every $i<j$ and
$a_{m}=[L]$.

Pick two generic perfect Morse functions $f,g$ on $L$ a Riemannian
metric $(\cdot, \cdot)$ and an almost complex structure $J$ on $M$ so
that the pearl complexes associated to $\mathscr{D}_f = (f,(\cdot,
\cdot),J)$ and to $\mathscr{D}_g = (g, (\cdot, \cdot), J)$ are well
defined as well as the chain level quantum product. We also require
that the minimum of $f$ is $x_{0}=R$, the maximum of $f$ is $x_{m}=Q$
and the minimum of $g$ is $y_{0}=P$.

Denote by $\mathcal{W}$ be the wide variety of $L$ (either
$\mathcal{W}_1$ or $\mathcal{W}_2$). The data $\mathscr{D}_f$ and
$\mathscr{D}_g$ give us two identifications
$$h_f, h_g:
Q^+H(L;\mathcal{W}) \longrightarrow H(L;\mathbb{C}) \otimes
\mathcal{O}(\mathcal{W}) \otimes \mathbb{C}[t].$$ For $c \in
H(L;\mathbb{C})$ we write $c^f = h_f^{-1}(c), \, c^g = h_g^{-1}(c) \in
Q^+H(L;\mathcal{W})$. The relation between $a_i^f$ and $a_i^g$ is
given by
\begin{equation} \label{eq:af-ag} a_i^f = a_i^g + \sum_{j>i}
   \sigma^i_j a_j^g t^{r^i_j}, \quad \textnormal{with } \sigma^i_j \in
   \mathcal{O}(\mathcal{W}), r^i_j \geq 1.
\end{equation}
Moreover, the coefficients $\sigma^i_j$ are all determined by counting
pearly moduli spaces involving only configurations of disks with total
Maslov index $\leq n/2$. This follows from the comparison maps
described in~\S\ref{sbsb:invariance} below. Recall also that $a_m =
[L]$ is transformed canonically to the unit of $Q^+H(L;\mathcal{W})$
and we have $a_m^f = a_m^g$. We therefore denote the latter by $a_m$
too.

Next, given $\alpha,\beta \in H(L;\mathbb{Z})$, denote by $x_{\alpha}
\in \mathbb{Z} \langle \textnormal{Crit}(f) \rangle$ the linear
combination of critical points representing in Morse homology the
class $\alpha$. Similarly, denote by $y_{\beta} \in \mathbb{Z} \langle
\textnormal{Crit}(g) \rangle$ the Morse cycle representing $\beta$.

Recall now the chain level product
$\mathcal{C}(\mathscr{D}_f;\mathcal{O}(\mathcal{W})[t]) \otimes
\mathcal{C}(\mathscr{D}_g; \mathcal{O}(\mathcal{W})[t])
\longrightarrow
\mathcal{C}(\mathscr{D}_f;\mathcal{O}(\mathcal{W})[t])$. We will
denote it here by $x \otimes y \longmapsto x \widetilde{*} y$, for $x
\in \textnormal{Crit}(f)$, $y \in \textnormal{Crit}(g)$, in order to
distinguish it from the induced product on homology which is denoted
by $*$. The relation between $*$ and $\widetilde{*}$ is given by:
$\alpha^f * \beta^g = [x_{\alpha} \widetilde{*} y_{\beta}]$. Of
course, in order to calculate $\alpha^f * \beta^f$ (rather than
$\alpha^f * \beta^g$) one needs now to appeal to
formula~\eqref{eq:af-ag}.

The following lemma follows immediately from the discussion above.
\begin{lem} \label{l:low} Let $\alpha, \beta \in \mathbf{a}$. Write
   \begin{equation} \label{eq:alpha*betta} \alpha^f * \beta^f =
      \sum_{i=0}^m s_i a_i^f t^{\nu_i}, \quad s_i \in
      \mathcal{O}(\mathcal{W}), s_i \neq 0.
   \end{equation}
   Then the following holds:
   \begin{enumerate}
     \item $\nu_i \leq n$ for every $i$. Moreover, if $\nu_i=n$ for
      some $i$, then $i=m$ and $\alpha=\beta=a_0$.
     \item The coefficients $s_i$ for $i<m$ are all determined by
      counting pearly moduli spaces that involve configurations of
      disks with total Maslov index strictly smaller than $2n$. This
      continues to hold also for $s_m$ if $\alpha \neq a_0$ or $\beta
      \neq a_0$.
   \end{enumerate}
\end{lem}

For a class $c \in Q^+H(L;\mathcal{W})$ denote by $\langle c, a_m
t^k\rangle_f$ the coefficient of $a_m t^k$ when writing $c$ in the
basis $a_0^f, \ldots, a_m^f$.

Consider now the product $a_0^f * a_0^g$. By the definition of the
product we have $$\langle a_0^f * a_0^g, a_m t^n \rangle_f =
n_{PQR}+\theta_{PQR},$$ where $\theta_{PQR}$ counts pearly
configurations as in~\S\ref{sbsb:or-prod} in which more than a single
$J$-holomorphic disk is present (by contrast, $n_{PQR}$ counts the
configurations given by a single disk through the three points).  The
total Maslov number of the configurations counted by $\theta$ is $2n$
and as there are more than two disks present, each such disk is of
Maslov at most $2n-2$. Thus in order to prove our theorem, it is
enough to show that $\langle a_0^f * a_0^g, a_m t^n \rangle_f$ can be
determined as a polynomial expression in variables that count pearly
configurations with total Maslov index $< 2n$.

For this purpose write $a_0^f = a_0^g + \sum_{j \geq 1}\sigma^0_j
a_j^g t^{r^0_j}$ with $1 \leq r^0_j\leq n/2$, as in~\eqref{eq:af-ag}.
We have:
$$\langle a_0^f*a_0^f, a_m t^n \rangle_f =
\langle a_0^f * a_0^g, a_m t^n \rangle_f + \sum_{j\geq 1} \sigma^0_j
\langle a_0^f * a_j^g, a_m t^{n-r^0_j} \rangle_f.$$ All the elements
in the second summand are determined by configurations with total
Maslov $\leq 2n-2$ (note that $r^0_j \geq 1$). Thus it is enough to
prove the same assertion for the left hand side $\langle a_0^f*a_0^f,
a_m t^n \rangle_f$.

We now use the assumption that $L$ is not a rational homology sphere.
Under this assumption it is possible to choose the basis $\mathbf{a}$
so that there exist $a,b \in \bf{a}$ with $0<|a|,|b|<m$ and $a\cdot b
= a_0$. We now have $$a^f * b^f = a_0^f + E(t)t,$$ where $E(t)$ is a
polynomial in $t$ whose coefficients are linear combinations of
$a^f_1, \ldots, a^f_m$, but $E(t)$ has no term containing $a_0^f$.
Note also that the second summand here is $E(t)t$, hence it has no
free term. It follows now that
\begin{align*}
   \langle (a^f * b^f)*(a^f*b^f), a_mt^n \rangle_f = \langle a_0^f
   * a_0^f, a_m t^n\rangle_f & + \langle a_0^f * E_1(t) +
   E_1(t)*a_0^f, a_m t^{n-1}\rangle_f \\
   & + \langle E(t)*E(t), a_m t^{n-2}\rangle_f.
\end{align*}
The last two summands on the right-hand side are clearly determined by
configurations with total Maslov $\leq 2n-2$ hence we are reduced to
showing that the same holds for the left-hand side.

We now use the fact that the quantum product is associative.  This
implies that
\begin{equation}\label{eq:assoc-poly}
(a^f * b^f)*(a^f*b^f) = ((a^f*b^f)*a^f)*b^f =
(a_0^f*a^f + E(t)t*a^f)*b^f.
\end{equation} By Lemma~\ref{l:low} $a_0^f*a^f$ has no
term with $t^n$ and the same holds also for $E(t)*a^f$. Moreover when
writing $a_0^f*a^f$ and $E(t)*a^f$ in the basis $\mathbf{a}^f$ all the
coefficients are determined by configurations with Maslov $\leq 2n-2$.
By Lemma~\ref{l:low} again, the same holds also for $(a_0^f*a^f +
E(t)t*a^f)*b^f$. It follows that all the coefficients (and in
particular that of $a_m t^n$ in $(a^f * b^f)*(a^f*b^f)$ depend on
configurations of Maslov $\leq 2n-2$. This concludes the proof.  \qed

\begin{rem} \label{rem:simple-split}\begin{itemize}
     \item[i.] Splitting polynomials of the type constructed above
      have a close relationship with the invariant ploynomials
      discussed in \S\ref{subsubsec:struct-const}.  Indeed, in the
      language of that section, suppose that the homology basis
      $(a_{i})$ is so that $a_{0}=[pt]$, $a_{s}=[L]$. Any invariant
      polynomial of the form $P= k^{0,0}_{s}+P'$ with $P'$ depending
      on variables different from $k^{0,0}_{s}$ produces a splitting
      polynomial $Q$.  To see this we first express the coefficients
      $k^{i,j}_{l}$ in terms of the coefficients $w^{i,j}_{l}$ of the
      ``geometric'' product $\widetilde{\ast}$.  We then express $P'$
      in the $w^{i,j}_{l}$'s (there are also other variables appearing
      here as in~\eqref{eq:af-ag}) thus obtaining a new polynomial
      $Q'$.  We then define $Q=-Q'- \theta +P(\ast)$ where
      $\theta=\theta_{PQR}+ \cdots$, and $\cdots$ stands for other
      terms resulting from the expression of $k^{0,0}_{s}$ in terms of
      the $w$'s and $P(\ast)$ is the value of the invariant polynomial
      on the product $\ast$.  The construction in the proof of the
      theorem is precisely of this type with $P$ a particular
      polynomial deduced from the associativity relation as it appears
      in (\ref{eq:assoc-poly}).

     \item[ii.] It would be interesting to know
      what is the ``simplest'' (in some sense yet to be defined)
      splitting polynomial $Q$ that one can produce by these methods.
   \end{itemize}
\end{rem}

\subsection{Lagrangian $2$--tori}\label{subsec:2-torus} In case of the
$2$--torus all the discussion above becomes much simpler and more
elegant. Moreover, we will deduce a splitting formula in terms of some
configurations that have some nice geometric meaning.

Consider three distinct points $P, Q, R \in L$. Choose a smooth
oriented path $\overrightarrow{PQ}$ starting from $P$ and ending at
$Q$. Similarly connect $Q$ to $R$ and $R$ to $P$ by such paths,
denoted $\overrightarrow{QR}$ and $\overrightarrow{RP}$ respectively.
We will refer to this triple of points connected by these curves as a
``triangle'' on the tours.

We will use now the notation from~\S\ref{sb:superpotential}, in
particular the set of classes $\mathcal{E}_2$ and the evaluation map
$ev: (\widetilde{\mathcal{M}}(B,J) \times \partial D)/ G
\longrightarrow L$. By taking $J$ generic we may assume that all three
points $P$, $Q$, $R$ are regular values of $ev$. Given $(u,z) \in
ev^{-1}(P)$ set $\varepsilon(u,z;P)= \pm 1$ according to whether $ev$
preserves or reverses orientations at $(u,z)$. Let $\rho \in
\mathcal{W}$. Define now the following (complex) number:
\begin{equation} \label{eq:np}
   n_P(\rho) = \sum_{B \in \mathcal{E}_2} \;
   \sum_{(u,z) \in ev^{-1}(P)} \rho(B) \varepsilon(u,z;P) \,
   \#\bigl(u(\partial D) \cap \overrightarrow{QR}\bigr),
\end{equation}
where $\#\bigl(u(\partial D) \cap \overrightarrow{QR}\bigr)$ stands
for the intersection number between the oriented curves $u(\partial
D)$ and $\overrightarrow{QR}$. The number $n_P(\rho)$ can be thought
of as the number of $J$-holomorphic disks of Maslov index $2$ whose
boundaries pass through $P$ and the ``edge'' $QR$ of the triangle
$PQR$, only that the count of the disks is weighted by the
representation $\rho$.  We also have the numbers $n_Q(\rho)$ and
$n_R(\rho)$ analogously defined.

\begin{rem} \label{r:n_P} The number $n_P(\rho)$ does not depend on
   the choice of the path $\overrightarrow{QR}$ connecting $Q$ to $R$,
   but only on the points $Q$ and $R$.  The reason for this is that
   the following $1$-dimensional cycle:
   $$\sum_{B \in \mathcal{E}_2} \;
   \sum_{(u,z) \in ev^{-1}(P)} \rho(B) \varepsilon(u,z;P) u(\partial
   D)$$ is null homologous in $H_1(L;\mathbb{C})$.  Indeed, it has
   been shown in \cite{Bi-Co:rigidity} \S 4.2 that if this cycle is
   not null-homologous, then the associated quantum homology vanishes
   (the proof was in fact only done for $\rho$ the identity
   representation but it is immediate to see that the argument also
   applies to any other representation).  In our case we are only
   considering this cycle for $\rho \in \mathcal{W}$ so that this
   forces the respective $1$-cycle to vanish in homology.

   Thus the intersection number of this cycle with the path
   $\overrightarrow{QR}$ depends only on its end points $Q$ and $R$.
   Nevertheless, $n_{P}(\rho)$ is far from being an invariant in any
   sense since it depends on the choice of the almost complex
   structure $J$ as well as on the points $P$, $Q$, $R$.
\end{rem}

As in the previous section we are interested to evaluate the number
$n_{PQR}$ of Maslov $2n=4$ disks through $P,Q,R$. Similarly to $n_P$,
the number $n_{PQR}$ is not an invariant too.

\subsubsection{Triangles on the torus and the discriminant}
To simplify notation we omit the $\rho$'s from the notation, i.e.
abbreviate $n_{PQR} = n_{PQR}(\rho)$, $n_P = n_P(\rho)$, $n_Q =
n_Q(\rho)$, $n_R = n_R(\rho)$.

\begin{thm} \label{t:discr-triangle}Let $PQR$ be a triangle on $L$.
   Then for every $\rho \in \mathcal{W}$ we have:
   \begin{equation}\label{eq:high-Maslov}
      \Delta(\rho) = 4 n_{PQR} +
      n_P^2 + n_Q^2 +
      n_R^2 - 2n_P n_Q - 2n_Q n_R - 2n_R n_P.\end{equation}
\end{thm}

The proof of this result is contained in the next couple of sections.
The first expresses the discriminant as a polynomial in certain
coefficients appearing in the expansion of the Lagrangian quantum
product. The second section continues with the combinatorial work
needed to relate these coefficients to the enumerative
expressions~\eqref{eq:high-Maslov}.

\subsubsection{The discriminant and higher quantum products}
\label{sb:discr-struct}
We continue here with the assumption that $L \subset M$ is a
$2$-dimensional Lagrangian torus with $N_L=2$. We also assume that the
wide variety $\mathcal{W}_2$ is not empty.

Let $\mathcal{W}$ be any of the wide varieties, $\mathcal{W}_1$ or
$\mathcal{W}_2$. Working with the ring $\mathcal{R} =
\mathcal{O}(\mathcal{W}) \otimes \Lambda^+$ we obtain
from~\eqref{eq:ex-seq-qh_n-2}:
\begin{equation} \label{eq:ex-seq-qh_0} 0 \longrightarrow
   \mathcal{O}(\mathcal{W})[L]t \stackrel{i}{\longrightarrow}
   Q^+H_{0}(L;\mathcal{W}) \stackrel{\pi} \longrightarrow
   H_0(L;\mathbb{C}) \otimes \mathcal{O}(\mathcal{W}) \longrightarrow
   0.
\end{equation}
Choose $\widetilde{p} \in Q^+H_0(L;\mathcal{W})$ with
$\pi(\widetilde{p})=[\textnormal{pt}] \in H_0(L;\mathbb{C})$.  Then
$\{\widetilde{p}, [L]t\}$ forms a basis for $Q^+H_0(L;\mathcal{W})$,
so we can write:
\begin{equation} \label{eq:p*p} \widetilde{p} * \widetilde{p} = \sigma
   \widetilde{p} t + \tau [L]t^2,
\end{equation}
where $\sigma, \tau \in \mathcal{O}(\mathcal{W})$. The coefficients
$\sigma, \tau$ depend on $\widetilde{p}$ as follows. If we replace
$\widetilde{p}$ by $\widetilde{p}\,' = \widetilde{p} + r [L]t$, for
some $r \in \mathcal{O}(\mathcal{W})$ then the corresponding
coefficients $\sigma'$ and $\tau'$ change as follows:
\begin{equation} \label{eq:sig-tau} \sigma' = \sigma + 2r, \quad \tau'
   = \tau - \sigma r -r^2.
\end{equation}
This can be verified by a direct computation from~\eqref{eq:p*p}.
Thus neither $\sigma$ nor $\tau$ are invariants. However it is easy to
see that $$\sigma^2 + 4\tau$$ is invariant in the sense that it does
not depend on $\widetilde{p}$ -- this is precisely an example of a
universal, symmetric polynomial Lagrangian invariant. In view of
Corollary \ref{cor:torus} we expect it to be related to the
discriminant. Indeed:
\begin{prop}\label{p:delta-sig-tau}
   We have the following identity in $\mathcal{O}(\mathcal{W})$:
   $\Delta = \sigma^2 + 4\tau$.
\end{prop}

\begin{proof}
   Choose a basis $\{C_1, C_2\}$ for $H_1(L;\mathbb{Z})$ such that
   $C_1 \cdot C_2 = [\textnormal{pt}]$. Write:
   \begin{equation}  \label{eq:c1c2-etc}
      \begin{aligned}
         & C_1*C_1 = \frac{1}{2} a_{11}[L]t, &\quad & C_2*C_2 =
         \frac{1}{2} a_{22}[L]t, \\
         & C_1*C_2 = \widetilde{p} + a'[L]t, & \quad & C_2*C_1 =
         -\widetilde{p} + a''[L]t, \quad a_{12} = a'+a'',
      \end{aligned}
   \end{equation}
   with $a_{11}, a_{22}, a', a'' \in \mathcal{O}(\mathcal{W})$.  Then
   $\widetilde{p} = C_1 * C_2 -a'[L]t = -C_2*C_1 + a''[L]t$, hence:
   \begin{align*}
      \widetilde{p} * \widetilde{p} & = -C_1*C_2*C_2*C_1 + a'' C_1*C_2
      t +
      a' C_2*C_1t - a'a''[L]t^2 \\
      & = (a''-a')\widetilde{p}t + \bigl(-\tfrac{1}{4} a_{11}a_{22} +
      a'a''\bigr)[L]t^2.
   \end{align*}
   Thus $\sigma = a''-a'$ and $\tau = a' a''-
   \frac{1}{4}a_{11}a_{22}$.  It immediately follows that $$\sigma^2 +
   4 \tau = a_{12}^2-a_{11}a_{22} = -\det(a_{ij}) = \Delta.$$
\end{proof}

\subsubsection{Enumerative expressions for $\sigma$ and $\tau$ and
  proof of Theorem~\ref{t:discr-triangle}}
\label{sbsb:enum-sig-tau}
We will relate the two coefficients $\sigma$ and $\tau$ above to the
enumerative expressions $n_{P}, n_{Q}, n_{R}, n_{PQR}$.
Theorem~\ref{t:discr-triangle} will then follow immediately from
Proposition~\ref{p:delta-sig-tau}.

We will use here a method described in~\cite{Bi-Co:qrel-long} and
in~\cite{Bi-Co:Yasha-fest}. This consists in picking two perfect Morse
function $f,g:L\to \R$ with pairwise distinct critical points, a
Riemannian metric $(\cdot, \cdot)$ on $L$ as well as an almost complex
structure $J$ which is sufficiently generic so that all pearl
complexes, products etc are defined. These functions are required to
satisfy a number of additional properties as described below.

Let $x_{0}$ be the minimum of $f$, let $x_{2}$ be the maximum of $f$,
let $y_{0}$ be the minimum of $g$ and similarly let $y_{2}$ be the
maximum of $g$. We may assume that $y_{2}$ is as close as we want to
$x_{2}$ in $L$. We also assume that the choices of $f,g$ as well as
that of the Riemannian metric $(\cdot, \cdot)$ are such that
$y_{0}=P$, $x_{0}=Q$, $x_{2}=R$ and the edge $\overrightarrow{RP}$ is
the the unique flow line of $-\nabla f$ going from $x_{2}$ to $y_{0}$,
and (after slightly rounding the corner at $P$) the edge
$\overrightarrow{PQ}$ is the unique flow line of $-\nabla f$ going
from $y_{0}$ to $x_{0}$. Moreover, the edge $\overrightarrow{QR}$
contains the point $y_{2}$ and it consists of two pieces: one is the
unique flow line of $-\nabla g$ going from $y_{2}$ to $x_{0}$ - the
orientation of this flow line is opposite that of
$\overrightarrow{QR}$; the second consists of a very short flow line,
$\gamma$, of $-\nabla g$ joining $y_{2}$ to $x_{2}$ - the orientation
of this flow line coincides with that of $\overrightarrow{QR}$. The
points $y_{2}$ and $x_{2}$ are taken close enough so that no
$J$-holomorphic disk of Maslov index $2$ passing through $y_{0}$
intersects $\gamma$ (as the number of these disks is finite this is
not restrictive).
\begin{figure}[htbp]
   \begin{center}
      \epsfig{file=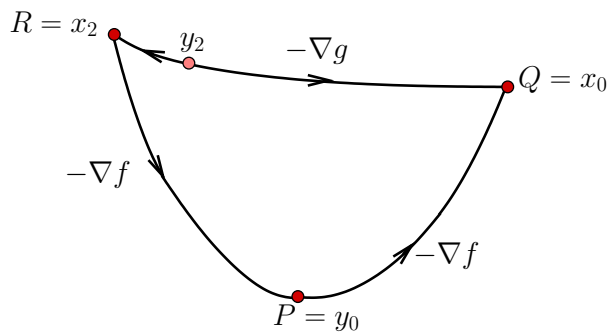, width=0.50\linewidth}
   \end{center}
   \caption{The triangle $PQR$ as drawn by negative flow lines of $f$
     and $g$ .}
   \label{f:triangle}
\end{figure}
Put $\mathscr{D} = (f,(\cdot, \cdot),J)$ and $\mathscr{D}'=(f,(\cdot,
\cdot),J)$ and consider the chain level quantum product
$$\mathcal{C}(\mathscr{D}) \otimes \mathcal{C}(\mathscr{D}')
\longrightarrow \mathcal{C}(\mathscr{D}), \quad x \otimes y
\longmapsto x*y.$$ By abuse of notation we denote by the same $*$ also
the induced product in homology. We work here with coefficients in
$\mathcal{O}(\mathcal{W}) \otimes \mathbb{C}[t]$, where $\mathcal{W}$
is one of the wide varieties $\mathcal{W}_1$ or $\mathcal{W}_2$.
Recall that we also have the comparison map $\Psi_{\mathscr{D}',
  \mathscr{D}}:\mathcal{C}(\mathscr{D}) \longrightarrow
\mathcal{C}(\mathscr{D}')$ whose definition is described
in~\S\ref{sbsb:invariance}.

Put $\widetilde{p} = [x_0] \in Q^+H_0(L;\mathcal{W})$, and write
$\widetilde{p} * \widetilde{p} = \sigma \widetilde{p} t + \tau [L]t^2$
as in~\eqref{eq:p*p}. Consider now the chain level product $x_0 *
y_0$, and write $$x_0*y_0 = \alpha x_0 t + \beta x_2t^2, \quad
\textnormal{for some } \alpha, \beta \in \mathcal{O}(\mathcal{W})$$
The relation between $x_0$ and $y_0$ is given by $\Psi_{\mathscr{D}',
  \mathscr{D}}$, namely $$\Psi_{\mathscr{D}', \mathscr{D}}(x_0) = y_0
+ \kappa y_2 t, \quad \textnormal{for some } \kappa \in
\mathcal{O}(\mathcal{W}).$$ It follows that $$\widetilde{p} *
\widetilde{p} = [x_0]*[x_0] = [x_0]*([y_0]+\kappa [L]t) =
(\alpha+\kappa) \widetilde{p} t + \beta [L]t^2,$$ hence we have
\begin{equation} \label{eq:sigma=alpha+kappa}
   \sigma = \alpha + \kappa, \quad \tau = \beta.
\end{equation}
We will now compute $\alpha, \kappa$ and $\beta$ explicitly.  We begin
with $\beta$. By the definition of the chain level product
(see~\S\ref{sbsb:or-prod}) we have $\beta = \beta_{I} + \beta_{II}$,
where $\beta_I$ counts configurations as in the left part of
figure~\ref{f:beta} with $\mu(\lambda)=4$ and $\beta_{II}$ counts the
configurations drawn in the right-hand side of that figure with
$\mu(\lambda_1) = \mu(\lambda_2) = 2$.

\begin{figure}[htbp]
   \begin{center}
      \epsfig{file=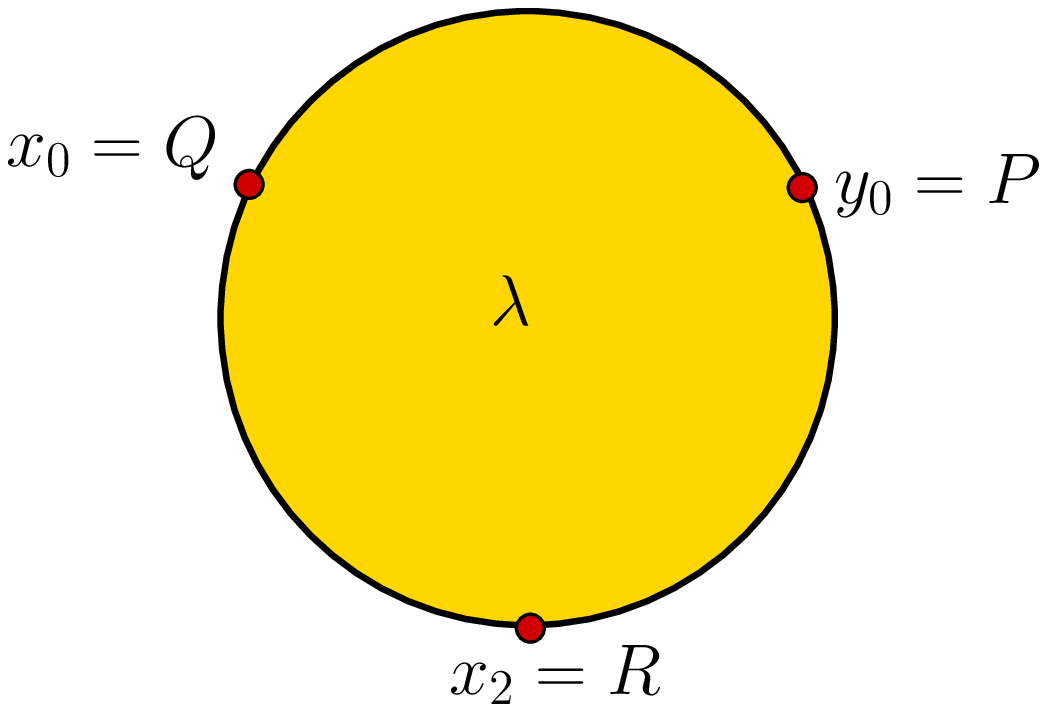, width=0.30\linewidth} \,\,
      \epsfig{file=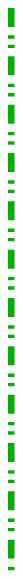, width=0.007\linewidth} \,\,
      \epsfig{file=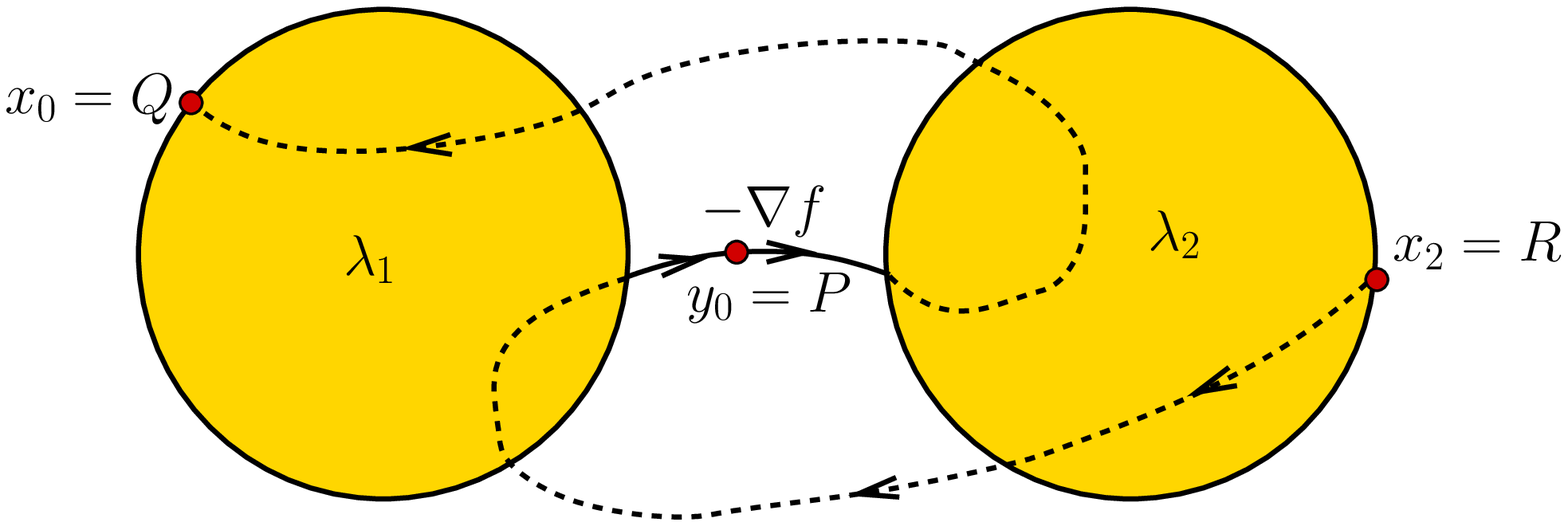, width=0.60\linewidth}
   \end{center}
   \caption{Configurations contributing to $\beta_{I}$ and
     $\beta_{II}$}
   \label{f:beta}
\end{figure}
To compute the precise values of $\beta_I$ and $\beta_{II}$ we use the
definition of the quantum product from~\S\ref{sbsb:or-prod}. We have
\begin{align}
   \beta_{I} = & \sum_{\lambda, \mu(\lambda)=4} \# \Bigl( \{x_0\}
   \times_L \bigl( \{y_0\} \times_L \widetilde{\mathcal{M}}(\lambda)
   \times_L \{x_2\}\bigr) \Bigr) \rho(\lambda)
   = \label{eq:beta_I} \\
   & \sum_{\lambda, \mu(\lambda)=4} \# \Bigl(
   \widetilde{\mathcal{M}}(\lambda) \times_{L \times L \times L}
   \{(R,P,Q)\}\Bigr) \rho(\lambda) = (-1)^n n_{PQR}= n_{PQR}. \notag
\end{align}

The last equality here hols because $n = \dim L = 2$.

We now compute $\beta_{II}$. For $\lambda_1, \lambda_2$ with
$\mu(\lambda_1) = \mu(\lambda_2) = 2$ put
$$\beta'_{II, \lambda_1} = \# \Bigl( \{x_0\} \times_L \bigl(
\mathcal{M}_2(\lambda_1) \times \mathbb{R}_+ \bigr) \times_L \{y_0\}
\Bigr), \quad \beta''_{II, \lambda_2} = \# \Bigl( (\{y_0\} \times
\mathbb{R}_+)\times_L \mathcal{M}_2(\lambda_2) \times_L \{x_2\}
\Bigr).$$ It follows easily from the definition of the quantum product
that
\begin{equation} \label{eq:beta_II} \beta_{II} = \sum_{\lambda_1,
     \lambda_2} \beta'_{II,\lambda_1} \beta''_{II,\lambda_2}
   \rho(\lambda_1)\rho(\lambda_2) = \Bigl(\sum_{\lambda_1}
   \beta'_{II,\lambda_1} \rho(\lambda_1)\Bigr)\Bigl(\sum_{\lambda_2}
   \beta''_{II,\lambda_2} \rho(\lambda_2)\Bigr),
\end{equation}
where the sums are over all $\lambda_1, \lambda_2$ with
$\mu(\lambda_1)=\mu(\lambda_2)=2$.  A straightforward computation
shows that
$$\beta'_{II,\lambda_1}= - \sum_{(u,z) \in ev^{-1}(Q)} \epsilon(u,z;Q)
\# (u(\partial D) \cap \overrightarrow{RP}),$$ where we use here the
notation from the beginning of~\S\ref{subsec:2-torus}.  It follows
from~\eqref{eq:np} that $\sum_{\lambda_1} \beta'_{II,\lambda_1}
\rho(\lambda_1) = -n_Q$. A similar computation gives $\sum_{\lambda_2}
\beta''_{II,\lambda_2} \rho(\lambda_2) = n_R$.  Substituting all this
into~\eqref{eq:beta_II} gives $\beta_{II} = -n_Q n_P$, hence
\begin{equation} \label{eq:beta}
   \beta(\rho) = n_{PQR} - n_Q n_R.
\end{equation}
We now turn to computing $\alpha$. For a class $A$ with $\mu(A)=2$
put:
\begin{align*}
   \alpha_{I,A} & = \{x_0\} \times_L \bigl(\mathcal{M}_2(A) \times
   \mathbb{R}_+ \bigr) \times_L \bigl(
   \{y_0\} \times_L L \times_L W^s_f(x_0) \bigr), \\
   \alpha_{II,A} & = \{x_0\} \times_L \bigl( \{y_0\} \times_L
   (\mathcal{M}_2(A) \times \mathbb{R}_+) \times_L W^s_f(x_0)\bigr).
\end{align*}
Then by the definition of the quantum product we have (see
figure~\ref{f:alpha}):
$$\alpha = \sum_{A} \alpha_{I,A} \rho(A) +
\sum_{A} \alpha_{II,A} \rho(A).$$
\begin{figure}[htbp]
   \begin{center}
      \epsfig{file=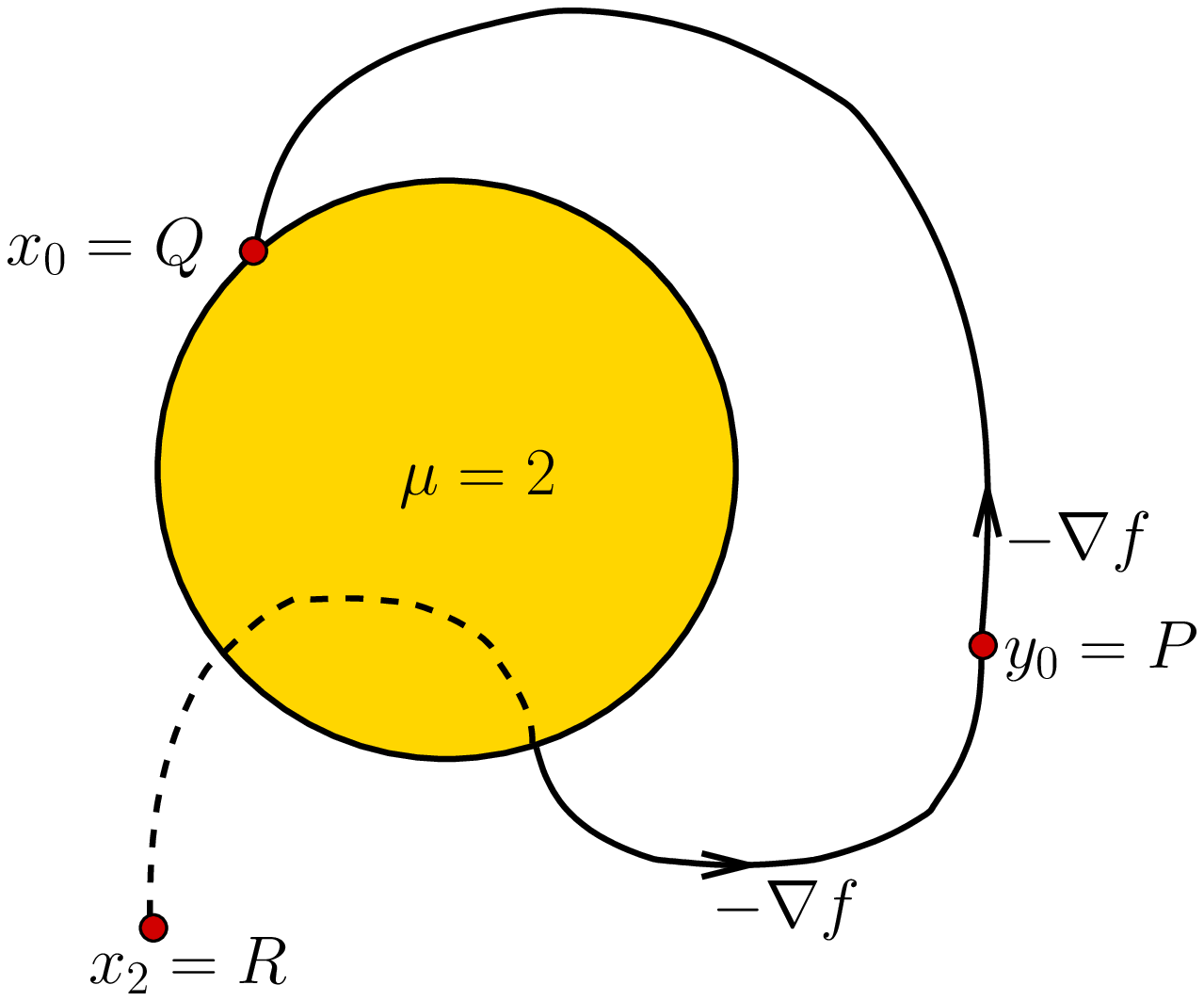, width=0.40\linewidth} \,\,
      \epsfig{file=separator.eps, width=0.007\linewidth} \,\,
      \epsfig{file=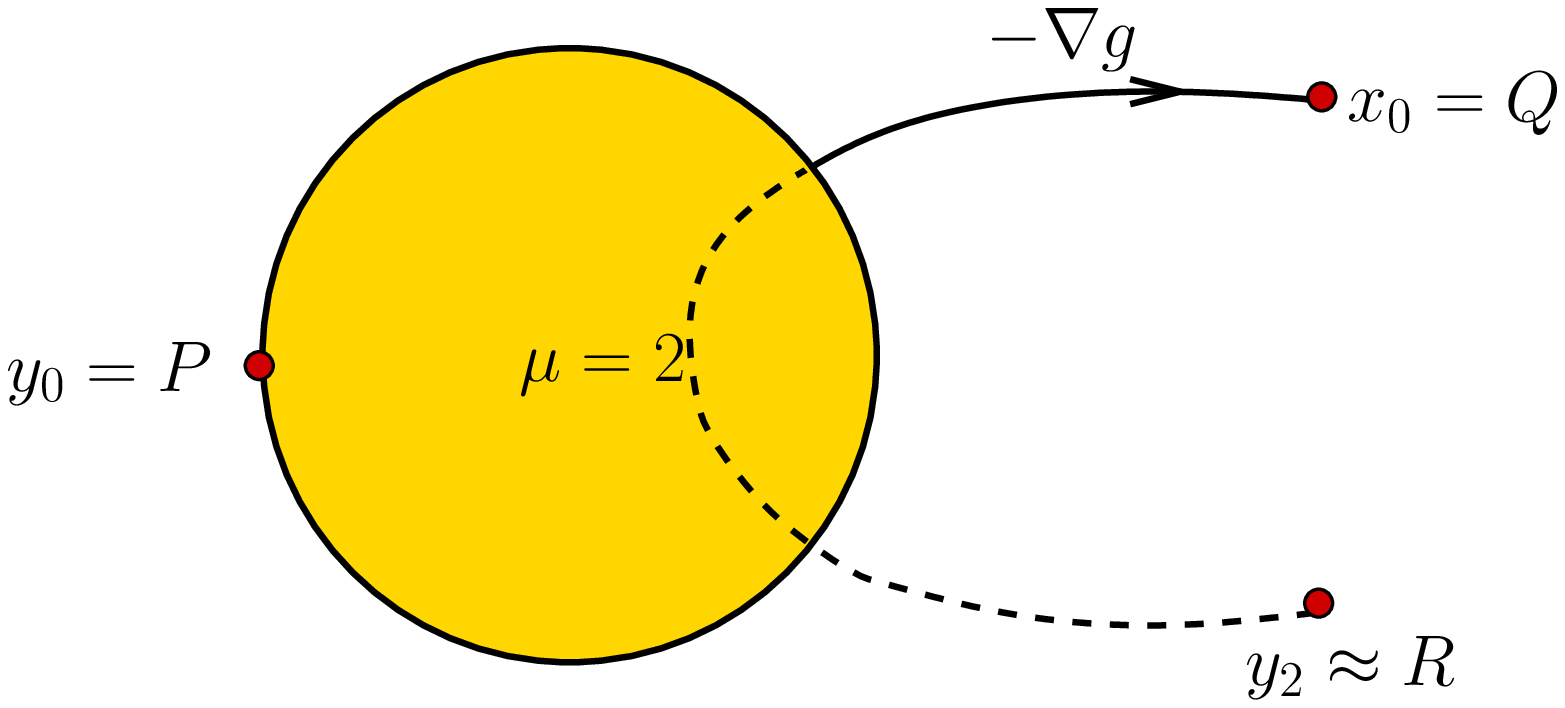, width=0.50\linewidth}
   \end{center}
   \caption{Configurations contributing to $\alpha_{I}$ and
     $\alpha_{II}$}
   \label{f:alpha}
\end{figure}

A straightforward computation shows that
$$\alpha_{I,A} = \{x_0\} \times_L
\bigl(\mathcal{M}_2(A) \times \mathbb{R}_+ \bigr) \times_L \{y_0\} =
-\sum_{(u,z) \in ev^{-1}(Q)} \epsilon(u,z;Q) \# (u(\partial D) \cap
\overrightarrow{RP}).$$ Summing over the $A$'s and weighting by $\rho$
we obtain that $\sum_{A} \alpha_{I,A} \rho(A) = -n_Q$.  A similar
computation gives: $\sum_{A}\alpha_{II,A} \rho(A) = n_P$.
It follows that
\begin{equation} \label{eq:alpha}
   \alpha = n_P - n_Q.
\end{equation}
It remains to compute $\kappa$. According to~\S\ref{sbsb:invariance},
$\kappa$ is computed by (see figure~\ref{f:kappa}):
$$\kappa = \sum_{A} \# \Bigl(\{x_0\}\times_L (L \times \mathbb{R}_+)
\times_L \mathcal{M}_2(A) \times_L\{y_2\}\Bigr) \rho(A).$$
\begin{figure}[htbp]
   \begin{center}
      \epsfig{file=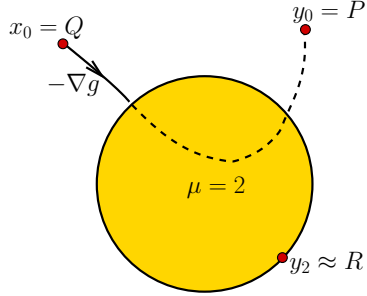, width=0.30\linewidth}
   \end{center}
   \caption{Configurations contributing to $\kappa$}
   \label{f:kappa}
\end{figure}
As $y_2$ was chosen close enough to $R$, a straightforward computation
gives $\kappa = -n_R$.

Substituting this together with~\eqref{eq:alpha} and~\eqref{eq:beta}
into~\eqref{eq:sigma=alpha+kappa} we get
\begin{equation} \label{eq:sig-tau-enum} \sigma = n_P - n_Q - n_R,
   \quad \tau = n_{PQR} - n_Q n_R.
\end{equation}
By Proposition~\ref{p:delta-sig-tau} we obtain:
$$\Delta = \sigma^2 + 4\tau = 4 n_{PQR} +
n_P^2 + n_Q^2 + n_R^2 - 2n_P n_Q - 2n_Q n_R - 2n_R n_P.$$
The proof of Theorem~\ref{t:discr-triangle} is complete. \Qed

\begin{rem} \label{r:eps-symmetry} Knowing the precise signs ($\pm$)
   appearing in the expressions for $\sigma$ and $\tau$ is not really
   necessary in order to prove Theorem~\ref{t:discr-triangle}. Here is
   the shortcut. It is enough to prove that there exist
   $\epsilon_{i}\in \{-1,1\}$, $i\in \{0,1,2,3\}$ so that
   $$\tau = n_{PQR} + \epsilon_0 n_Q n_R, \quad
   \sigma=\epsilon_{1}n_{P}+\epsilon_{2}n_{Q} + \epsilon_{3}n_{R}.$$
   Then by Proposition~\ref{p:delta-sig-tau} we get:
   $$\Delta=4n_{PQR}+4\epsilon_{0}n_{Q}n_{R} +
   (\epsilon_{1}n_{P}+\epsilon_{2}n_{Q}+\epsilon_{3}n_{R})^{2}~.~$$ We
   already know that $\Delta$ is an invariant and, in particular, it
   is left invariant by circular permutations of $P,Q,R$. This
   immediately implies that $\epsilon_{1}$, $\epsilon_{2}$,
   $\epsilon_{3}$ can not all have the same sign and so we may assume
   that just one of them is negative and the other two positive. If
   either one of $\epsilon_{2},\epsilon_{3}$ is negative this circular
   symmetry can not be satisfied. So we conclude that
   $\epsilon_{1}=-1$.  Again for symmetry reasons this implies
   $\epsilon_{0}=-1$ and proves the claim.
\end{rem}

%% file: enum-mod-2.tex
\subsection{Modulo--$2$ invariants} \label{sb:mod-2}

More can be said about the discriminant as well as the enumerative
counts introduced in~\S\ref{subsec:2-torus} after reduction modulo $2$
(and modulo $4$). In the following theorem we focus for simplicity on the trivial representation.
We denote by $\Delta =
\Delta(1) \in \mathbb{Z}$ the discriminant computed at the trivial
representation $\rho \equiv 1$. Similarly, we denote by $n_P, n_Q,
n_R, n_{PQR} \in \mathbb{Z}$ the numbers defined
in~\S\ref{subsec:2-torus}, and by $\sigma, \tau \in \mathbb{Z}$ be the
structural constants defined at~\eqref{eq:p*p}, all computed at
$\rho\equiv 1$.

\begin{thm} \label{t:mod-2} Let $L^2 \subset M^4$ be a wide Lagrangian
   torus with $N_L=2$, where by ``wide'' we mean here that the trivial
   representation $\rho \equiv 1$ belong to $\mathcal{W}_2$. Then:
   \begin{enumerate}
     \item $\Delta \equiv \sigma \equiv n_P + n_Q +n_R \pmod{2}$.
     \item $\Delta \pmod{4}$ admits only the values $0$ or $1$.
   \end{enumerate}
   Moreover, if $\Delta \equiv 1 \pmod{2}$ then:
   \begin{enumerate}[(i)]
     \item The value of $\tau \pmod{2}$ is invariant in the sense that
      it does not depend on the choice of the element $\widetilde{p}$
      in~\eqref{eq:p*p}. \label{i:tau-mod-2}
     \item $n_P n_Q \equiv n_Q n_R \equiv n_R n_P \pmod{2}$.
      \label{i:n_pn_q}
     \item The value of $n_{PQR} + n_P n_Q \pmod{2}$ is invariant, i.e.
      does not depend neither on $P, Q, R$ nor on the almost complex
      structure. This number is congruent to $\tau \pmod{2}$.
      \label{i:n_4+n_pn_q}
   \end{enumerate}
\end{thm}

\begin{proof}
   Recall from proposition~\ref{p:delta-sig-tau} that $\Delta =
   \sigma^2 + 4 \tau$. Hence $\Delta \equiv \sigma \pmod{2}$.  The
   fact that $\sigma \equiv n_P + n_Q + n_R \pmod{2}$ follows
   from~\eqref{eq:sig-tau-enum}. Next note that $\Delta \equiv
   \sigma^2 \pmod{4}$, hence the latter can obtain only the values $0$
   and $1$ $\pmod{4}$. This proves the first two statements in the
   theorem.

   To prove the other statements, assume now that $\Delta \equiv 1
   \pmod{2}$, or equivalently that $\sigma \equiv 1 \pmod{2}$.  The
   fact that $\tau \pmod{2}$ is invariant follows immediately from
   formulae~\eqref{eq:sig-tau}. This proves~(\ref{i:tau-mod-2}).

   To prove the identity~(\ref{i:n_pn_q}) note that if $\sigma \equiv
   1 \pmod{2}$ then either the three numbers $n_P, n_Q, n_R$
   $\pmod{2}$ are all $1$, or exactly two of them are $0$ and one of
   them is $1$.  In both cases the identity in~(\ref{i:n_pn_q}) holds.

   Finally, point~(\ref{i:n_4+n_pn_q}) follows from the arguments
   of~\S\ref{sbsb:enum-sig-tau}. See~\eqref{eq:sig-tau-enum} as well
   as Theorem~7.2.2 in \cite{Bi-Co:Yasha-fest}.
\end{proof}

\begin{rems}
   \begin{enumerate}
     \item Some of the statements in Theorem~\ref{t:mod-2} (e.g.
      point~(\ref{i:n_4+n_pn_q})) do not seem to follow by just
      reducing $\pmod{2}$ the identity~\eqref{eq:high-Maslov}, but
      rather reveal more geometric information on the structure of the
      ``constants'' $n_P, n_Q, n_R$ and $n_{PQR}$.
     \item It seems that one could get more general congruences by
      allowing every representation $\rho \in \mathcal{W}_1$ (not just
      the trivial one). The point is that all the calculations
      involving an element $\rho \in \mathcal{W}_1$ can be done in a
      number field (i.e. a finite extension of $\mathbb{Q}$) and the
      values of $\Delta(\rho)$ and the constants $n_P(\rho),
      n_Q(\rho), n_R(\rho), n_{PQR}(\rho)$ belong to the ring of
      integers of this field. One expects some congruence relations
      (with respect to some ideal in this ring) to hold between these
      numbers.
   \end{enumerate}
\end{rems}

%% file: toric.tex
\section{Toric fibers} \label{s:toric} Here we work
out in detail the theory discussed in the previous sections for the
special case of Lagrangian tori that arise as fibres of the moment map
in a toric manifold. Below we will use in an essential way previous
results of Cho, Oh and of Fukaya-Oh-Ohta-Ono on Floer theory of torus
fibres in toric manifolds (see e.g.~\cite{Cho:Clifford,
  Cho-Oh:Floer-toric, Cho:products, FO3:toric-1, FO3:toric-2}). The
reader is referred to these papers for more details. For the
foundations of symplectic toric manifolds
see~\cite{Audin:torus-actions}, and for an algebro-geometric account
see~\cite{Fulton:toric-book}.

\subsection{Setting} \label{sb:toric-setting} Let $(M^{2n}, \omega)$
be a closed monotone toric manifold. Denote by $$\mathfrak{m}:M
\longrightarrow \textnormal{Lie}(\mathbb{T}^n)^* = \mathbb{R}^n$$ the
moment map and by $P = \textnormal{image\,}(\mathfrak{m})$ the moment
polytope. The symplectic manifold $(M, \omega)$ admits a canonical
$\omega$-compatible (integrable) complex structure $J_0$ which turns
$(M,J_0)$ into a complex algebraic manifold. We will refer to $J_0$ as
the {\em standard complex structure}.

Denote by $F_1, \ldots, F_r$ the codimension-$1$ facets of $P$ and by
$\overrightarrow{v_1}, \ldots, \overrightarrow{v_r} \in \mathbb{Z}^n$
the normal integral primitive vectors to the facets $F_1, \ldots, F_r$
respectively, pointing inwards $P$. Note that the number of
codimension-$1$ facets is $r = n + b_2(M)$. The fibres
$\mathfrak{m}^{-1}(p)$, $p \in P$, are Lagrangian tori. There is a
(unique) special point $p_* \in P$ for which the Lagrangian torus $L =
\mathfrak{m}^{-1}(p_*)$ is monotone (see
e.g.~\cite{Cho-Oh:Floer-toric, Cho:non-unitary, FO3:toric-1}).
Furthermore, we have $N_L = 2$. Note also that $H_2^D \cong
\pi_2(M,L)$, and that $r = \textnormal{rank}\, H_2^D$. We denote by
$\Sigma_i = \mathfrak{m}^{-1}(F_i) \subset M$.  These turn out to be
smooth $J_0$-complex hypersurfaces in $M$ and their sum $[\Sigma_1] +
\cdots + [\Sigma_r]$ represents the Poincar\'{e} dual of the first
Chern class $c_1$ of $M$ (which is by assumption a positive multiple
of $[\omega]$).

Since $L$ is an orbit of the $\mathbb{T}^n$-action we have a canonical
identification $H_1(L;\mathbb{Z}) = H_1(\mathbb{T}^n;\mathbb{Z})$ and
we denote by $\mathbf{e} = \{ e_1, \ldots, e_n \}$ the standard basis
corresponding to this identification.

\subsubsection{Holomorphic disks} \label{sbsb:toric-hol-disks} Due to
the $\mathbb{T}^n$-action the Lagrangian torus $L$ comes with a
preferred orientation as well as a spin structure. Fixing these two,
one can endow the space of holomorphic disks with boundary on $L$ with
a canonical orientation (see~\cite{FO3, Cho:products} for more
details).

We start with a description, due to Cho and
Oh~\cite{Cho-Oh:Floer-toric}, of the subset $\mathcal{E}_2 \subset
H_2^D$ of classes that can be represented by $J$-holomorphic disks
with Maslov index $2$ for generic $J$ (as well as for $J_0$). We use
here the notation from~\S\ref{sb:superpotential}.

\begin{prop}[Cho-Oh~\cite{Cho-Oh:Floer-toric}]
   \label{p:toric-mu=2-disks} The set $\mathcal{E}_2$ consists of
   exactly $r=\textnormal{rank}\, H_2^D$ classes $\mathcal{E}_2 = \{
   B_1, \ldots, B_r\}$ with the following properties:
   \begin{enumerate}
     \item $\# (B_i \cdot \Sigma_j) = \delta_{i,j}$ for every $i,j$.
     \item The set $\mathcal{E}_2$ is a $\mathbb{Z}$--basis for
      $H_2^D$.
     \item Denote by $\partial : H_2^D \longrightarrow
      H_1(L;\mathbb{Z})$ the boundary operator. Then writing $\partial
      B_i$ in the basis $\mathbf{e}$ we have $(\partial B_i) =
      \overrightarrow{v_i}$ for every $1 \leq i \leq r$.
     \item For every $i$, $\nu(B_i)=1$.
   \end{enumerate}
   Furthermore, the standard complex structure $J_0$ is regular for
   all classes $B \in H_2^D$ with $\mu(B)=2$. Moreover given a generic
   point $x \in L$ there exist precisely $r$ $J_0$-holomorphic disks
   $u_i: (D, \partial D) \longrightarrow (M,L)$, $i=1, \ldots, r$,
   upto parametrization, with $\mu([u_i])=2$ and $u(1)=x$. These disks
   satisfy $u_i([D])=B_i$, $i=1, \ldots, r$. The image of $u_i$ under
   the moment map $\mathfrak{m} \circ u_i$ is a straight segment going
   from $p_*$ to a point on the facet $F_i$.
\end{prop}

\subsubsection{The superpotential, the wide variety and the
  discriminant} \label{sb:P-wide-discr-toric} The following is an
immediate corollary of Proposition~\ref{p:toric-mu=2-disks}:
\begin{cor}[Cho-Oh~\cite{Cho-Oh:Floer-toric}]
   \label{c:super-potential}
   The superpotential $\mathscr{P}$ has the following form in the
   coordinates induced by the basis $\mathbf{e}$
   (see~\S\ref{sbsb:sp-formulae}):
   \begin{equation} \label{eq:P-toric}
      \mathscr{P}(z_1, \ldots, z_n) = \sum_{i=1}^r
      z^{\overrightarrow{v_i}},
   \end{equation}
   where for a vector $\overrightarrow{v}=(v^1, \ldots, v^n) \in
   \mathbb{Z}^n$, $z^{\overrightarrow{v}}$ stands for the monomial
   $z^{\overrightarrow{v}} = z_1^{v^{1}} \cdots z_n^{v^{n}}$.
\end{cor}

Since $H^1(L;\mathbb{R})$ generates $H^*(L;\mathbb{R})$ (with respect
to the cup product) we obtain from Proposition~\ref{p:dC-i}:
\begin{cor}[Cho-Oh~\cite{Cho-Oh:Floer-toric}, see also
   Fukaya-Oh-Ohta-Ono~\cite{FO3:toric-1}] \label{c:wide-crit} We have
   $$\mathcal{W}_1 = \textnormal{Crit}(\mathscr{\mathscr{P}}),$$ where
   $\mathcal{W}_1 \subset \textnormal{Hom}(H_1, \mathbb{C}^*)$ is the
   wide variety as defined in~\S\ref{s:wide}.
\end{cor}

Choose a basis $\mathbf{C} = \{C_1, \ldots, C_n \}$ for
$H_{n-1}(L;\mathbb{Z})$ which is dual to $\mathbf{e}$ as
in~\S\ref{sbsb:sp-formulae}. We view $H_{n-1}(L;\mathbb{C})$ as a
subset of $QH_{n-1}(L;\mathcal{W}_1)$ as explained
in~\S\ref{sbsb:sp-formulae} just before the statement of
Proposition~\ref{p:prod-Ci-Cj}. The following corollary immediately
follows from Proposition~\ref{p:prod-Ci-Cj}.
\begin{cor}[Cho~\cite{Cho:products}] \label{c:prod} $$C_i * C_j + C_j
   * C_i = (-1)^n \Bigl(\sum_{k=1}^r v_k^{i} v_k^{j} \,
   z^{\overrightarrow{v_k}} \Bigr)[L] t, \; \forall z \in
   \mathcal{W}_1.$$
\end{cor}

We now turn to the quadratic form defined in~\S\ref{s:qforms} and its
discriminant. Substituting~\eqref{eq:P-toric} in~\eqref{eq:discr-2} we
get:
\begin{equation} \label{eq:discr-toric-1} \Delta(z_1, \ldots, z_n) =
   (-1)^{n+1}\det \Bigl( \sum_{k=1}^r v_k^i v_k^j
   \,z^{\overrightarrow{v_k}} \Bigr)_{i,j}, \; \forall z \in
   \mathcal{W}_1.
\end{equation}
However, there is a nicer formula for the discriminant which we now
present. For a subset of indices $I \subset \{1, \ldots, r\}$ with
$\#I=n$, say $I = \{i_1, \ldots, i_n\}$, define an $n \times
n$\,-matrix $A_I$ whose rows consists of the vectors
$\overrightarrow{v_{i_1}}, \ldots, \overrightarrow{v_{i_n}}$, and a
vector $\overrightarrow{v_I}$ which is the sum of the
$\overrightarrow{v_{i_k}}$'s, i.e.
\begin{equation} \label{eq:A_I}
   A_I =
   \begin{pmatrix}
      \textnormal{-----------} & \overrightarrow{v_{i_1}} &
      \textnormal{-----------} \\
      & \vdots & \\
      \textnormal{-----------} & \overrightarrow{v_{i_k}} &
      \textnormal{-----------} \\
      & \vdots &  \\
      \textnormal{-----------} & \overrightarrow{v_{i_n}} &
      \textnormal{-----------}
   \end{pmatrix}
   , \qquad \overrightarrow{v_I} = \sum_{i \in I} \overrightarrow{v_i}.
\end{equation}
Note that $\det(A_I)^2$ does not depend on the ordering of the indices
$i_j$ in the set $I$.

\begin{prop}
   The discriminant verifies the following formula:
   \begin{equation} \label{eq:discr-toric-2} \Delta(z_1, \ldots, z_n)
      = (-1)^{n+1}\sum_{\substack{I \subset \{1, \ldots, r\} \\
          \#I=n}} z^{\overrightarrow{v_I}} \det(A_I)^2, \; \forall z
      \in \mathcal{W}_1.
   \end{equation}
\end{prop}
The proof follows by direct computation by expanding the determinant
in~\eqref{eq:discr-toric-1}.

\subsection{Formulae for $\mathcal{W}_2$} \label{sb:formulae-W_2}
Recall from Proposition~\ref{p:toric-mu=2-disks} that set
$\mathcal{E}_2 = \{B_1, \ldots, B_r\}$ forms a $\mathbb{Z}$--basis for
$H_2^D$. Using this basis we can identify $\textnormal{Hom}(H_2^D,
\mathbb{C}^*) \cong (\mathbb{C}^*)^{\times r}$. An element of the
latter space $\xi = (\xi_1, \ldots, \xi_r)$ will be identified with
the representation $\rho$ that satisfies $\rho(B_k) = \xi_k$, $k=1,
\ldots, r$.

We continue to work with the basis $\mathbf{e}=\{e_1, \ldots, e_n\}$
for $H_1$ introduced in~\S\ref{sb:toric-setting} and the dual basis
$\mathbf{C} = \{C_1, \ldots, C_{n-1}\}$ for $H_{n-1}(L;\mathbb{Z})$.

With this notation the following is straightforward calculation
resulting from Proposition~\ref{p:toric-mu=2-disks}.

\begin{prop} \label{p:formulae-W_2}
   \begin{enumerate}
     \item The wide variety $\mathcal{W}_2$ is cut by the following
      system of $n$ linear equation (with $r$ unknowns):
      $$\mathcal{W}_2 = \Bigl\{ \sum_{k=1}^r v_k^j \xi_k = 0 \bigm|
      j=1, \ldots, n\Bigr\}.$$ Here, $v_k^j$ is the $j$'th component
      of the vector $\overrightarrow{v_k}$, i.e. $\overrightarrow{v_k}
      = (v_k^1, \ldots, v_k^n)$.
     \item The natural map $\partial_{\mathcal{W}}: \mathcal{W}_1
      \longrightarrow \mathcal{W}_2$ induced by the boundary map
      $\partial: H_2^D \longrightarrow H_1$ is given by:
      $$\partial_{\mathcal{W}}(z_1, \ldots, z_n) =
      (z^{\overrightarrow{v_1}}, \ldots, z^{\overrightarrow{v_r}}).$$
     \item The product of elements of $\mathbf{C}$ satisfies:
      $$C_i*C_j + C_j*C_i = (-1)^n
      \Bigl(\sum_k^r v_k^i v_k^j \xi_k \Bigr)
      [L]t.$$
     \item The discriminant is given by:
      $$\Delta(\xi_1, \ldots, \xi_k) =
      (-1)^{n+1}
      \det \Bigl(\sum_k^r v_k^i v_k^j \xi_k \Bigr)_{i,j}.$$
   \end{enumerate}
\end{prop}

\subsection{Wide varieties and quantum homology of the ambient
  manifold} \label{sb:wide-qh} Here we further study the other quantum
structures, such as the quantum algebra and quantum inclusion, and
their relations to the wide varieties on toric manifolds.

Let $L = \mathfrak{m}^{-1}(p_*) \subset M$ be the monotone torus fibre
in a monotone toric manifold. Assume that the wide variety
$\mathcal{W}_1$ is not empty. By Corollary~\ref{c:wide-crit} the wide
variety $\mathcal{W}_1$ coincides with the variety of critical points
of the superpotential function $\mathscr{P}$, $\mathcal{W}_1 =
\textnormal{Crit}(\mathscr{P})$, hence the ring or global algebraic
functions $\mathcal{O}(\mathcal{W}_1)$ can be written as
\begin{equation} \label{eq:O-W_1} \mathcal{O}(\mathcal{W}_1) =
   \frac{\mathbb{C}[z_1^{\pm 1}, \ldots, z_n^{\pm 1}]}{\langle
     \partial_{z_1} \mathscr{P}, \ldots, \partial_{z_n} \mathscr{P}
     \rangle},
\end{equation}
where the denominator stand for the ideal generated by the partial
derivatives of $\mathscr{P}$. This ring, or rather localizations of
it, plays an important role in singularity theory and is sometime
called the Jacobian ring of $\mathscr{P}$.

Interestingly, this ring appears in the symplectic picture also from a
different angle. Denote by $QH(M;\Lambda)$ the quantum homology of the
{\em ambient} manifold with coefficients in $\Lambda = \mathbb{C}[t,
t^{-1}]$, where for compatibility with the Lagrangian picture we put
$|t| = -N_L=-2$. It is well known that the classes $[\Sigma_i] =
\frak{m}^{-1}(F_i) \in QH_{2n-2}(M;\Lambda)$, $i=1, \ldots, r$,
generate $QH(M; \Lambda)$ with respect to the quantum product $*$,
see~\cite{Batyrev:qh-toric, McD-Sa:Jhol-2}. It turns out that
$QH(M;\Lambda)$ is isomorphic as a ring to $\mathcal{O}(\mathcal{W}_1)
\otimes \Lambda$. More precisely:
\begin{thm}[Batyrev, Givental, Fukaya-Oh-Ohta-Ono] \label{t:Bat-Giv}
   There exists an isomorphism of rings
   \begin{equation} \label{eq:iso-I} I: QH(M; \Lambda) \longrightarrow
      \mathcal{O}(\mathcal{W}_1) \otimes \Lambda,
   \end{equation}
   which satisfies $I([\Sigma_i]) = z^{\overrightarrow{v_i}}t$. This
   isomorphism shifts degrees by $-2n$. Here, the grading on the
   right-hand side comes from the $\Lambda$-factor only (i.e.
   $\mathcal{O}(\mathcal{W}_1)$ is not graded).
\end{thm}
This theorem was first suggested by Givental~\cite{Gi:hom-geom-ICM}
and by Batyrev~\cite{Batyrev:qh-toric} and has been verified since
then at different levels of rigor. A rather rigorous and conceptual
proof has been recently carried out by
Fukaya-Oh-Ohta-Ono~\cite{FO3:toric-1}. See also~\cite{Os-Ty:qh} for a
more algebraically oriented proof. It is important to note that the
isomorphism~\eqref{eq:iso-I} does not send $QH(M;\Lambda^+)$ onto
$\mathcal{O}(\mathcal{W}_1) \otimes \Lambda^+$ but rather into a
subring of the latter.

We now consider the quantum module structure on $QH(L)$.  Recall
from~\ref{sb:qh} that for a $\widetilde{\Lambda}^+$-algebra
$\mathcal{R}$, $QH(L;\mathcal{R})$ is a module (in fact an algebra)
over $QH(M; \mathcal{R})$. We will use here $\mathcal{R} =
\mathcal{O}(\mathcal{W}_1) \otimes \Lambda$, but a similar discussion
holds for $\mathcal{O}(\mathcal{W}_1) \otimes \Lambda^+$ too. For $a
\in QH_j(M; \mathcal{O}(\mathcal{W}_1) \otimes \Lambda)$ and $\alpha
\in QH_k(L; \mathcal{W}_1)$ (see the notation in~\eqref{eq:qh-w}) we
denote by $a*\alpha \in QH_{j+k-2n}(L;\mathcal{W}_1)$ the quantum
module action.  Using the embedding $\Lambda = 1 \otimes \Lambda
\subset \mathcal{O}(\mathcal{W}_1) \otimes \Lambda$, we have a natural
inclusion $QH(M;\Lambda) \subset QH(M; \mathcal{O}(\mathcal{W}_1)
\otimes \Lambda)$. We will now take a closer look at the induced
module operation

\begin{equation} \label{eq:q-mod-W} QH(M; \Lambda) \otimes_{\Lambda}
   QH(L; \mathcal{W}_1) \longrightarrow QH(L;\mathcal{W}_1), \quad a
   \otimes \alpha \longmapsto a*\alpha.
\end{equation}
Note that $QH(L;\mathcal{W}_1)$ is also a module over
$\mathcal{O}(\mathcal{W}_1) \otimes \Lambda$ in an obvious way. For $c
\in \mathcal{O}(\mathcal{W}_1) \otimes \Lambda$, $\alpha \in
QH(L;\mathcal{W}_1)$ we denote this module operation as $c\alpha$. It
turns out that this module structure and the preceding ones are in
fact compatible:
\begin{prop} \label{p:a*alpha}
   For every $a \in QH(M; \Lambda)$ and $\alpha \in QH(L;
   \mathcal{W}_1)$ we have:
   \begin{equation} \label{eq:qmod-identity}
      a*\alpha = I(a)\alpha.
   \end{equation}
   The same continues to hold if we replace $\Lambda$ by $\Lambda^+$
   and $QH(L;\mathcal{W}_1)$ by $Q^+H(L;\mathcal{W}_1)$.
\end{prop}
\begin{proof}
   Since $QH(M;\Lambda)$ is generated by the classes $[\Sigma_i] =
   \frak{m}^{-1}(F_i)$, $i=1, \ldots, r$, it is enough to
   check~\eqref{eq:qmod-identity} for $a = [\Sigma_i]$. Since $*$ is a
   module action, it is also enough to restrict to the case $\alpha =
   [L]$ which is the unity of $QH(L;\mathcal{W}_1)$.

   Next note that $[\Sigma_i]*[L]$ lies in the image of the natural
   map $Q^+H(L;\mathcal{W}_1) \longrightarrow QH(L;\mathcal{W}_1)$
   hence it is enough to show that~\eqref{eq:qmod-identity} holds in
   $Q^+H(L;\mathcal{W}_1)$.

   Recall from~\eqref{eq:ex-seq-qh_n-2} that we have the following
   exact sequence:
   \begin{equation} \label{eq:ex-seq-W_1-qh_n-2} 0 \longrightarrow
      \mathcal{O}(\mathcal{W}_1) [L]t \stackrel{i}{\longrightarrow}
      Q^+H_{n-2}(L;\mathcal{W}_1) \stackrel{\pi} \longrightarrow
      H_{n-2}(L;\mathbb{C}) \otimes \mathcal{O}(\mathcal{W}_1)
      \longrightarrow 0.
   \end{equation}
   Moreover, it follows from the definition of the quantum module
   action that $\pi([\Sigma_i]*[L]) = [\Sigma_i]\cdot [L]$, where
   $[\Sigma_i] \cdot [L]$ stand for the classical intersection product
   in singular homology. But $L$ is disjoint from $\Sigma_i$, hence
   $\pi([\Sigma_i]*[L])= [\Sigma_i] \cdot [L] = 0$. It follows
   from~\eqref{eq:ex-seq-W_1-qh_n-2} that $[\Sigma_i]* [L] =c[L]t$ for
   some function $c \in \mathcal{O}(\mathcal{W}_1)$.

   To determine $c$ note that if we work with coefficients in
   $\widetilde{\Lambda}^+$ we have:
   \begin{equation} \label{eq:qmod-toric-wtld} [\Sigma_i]*[L] =
      \sum_{B \in \mathcal{E}_2} \#(B \cdot \Sigma_i) \nu(B) T^B [L].
   \end{equation}
   Substituting the information from
   Proposition~\ref{p:toric-mu=2-disks}
   into~\eqref{eq:qmod-toric-wtld} we immediately obtain
   $$[\Sigma_i]*[L] = z^{\overrightarrow{v_i}} [L]t.$$ Since
   $z^{\overrightarrow{v_i}} t = I([\Sigma_i])$ this concludes the
   proof.
\end{proof}

Next we consider the quantum inclusion map. Let $f:L \longrightarrow
\mathbb{R}$ be a perfect Morse function having exactly one minimum,
$x_0 \in L$. Let $(\cdot, \cdot)$ be a Riemannian metric on $L$ and
$J$ an $\omega$-compatible almost complex structure on $M$. Put
$\mathscr{D} = (f, (\cdot, \cdot), J)$ and assume the elements of this
triple have been chosen to be generic so that the pearl complex
$\mathcal{C}(\mathscr{D};\mathcal{O}(\mathcal{W}_1)\otimes \Lambda)$
is well defined. Under these assumption $x_0 \in
\mathcal{C}(\mathscr{D};\mathcal{O}(\mathcal{W}_1)\otimes \Lambda)$ is
a cycle and we denote by $[x_0] \in QH(L;\mathcal{W}_1)$ its homology
class. Note that in general $[x_0]$ strongly depends on the choice of
$\mathscr{D}$. (See~\S 4.5 in~\cite{Bi-Co:rigidity}.) Nevertheless, it
turns out that its image under the quantum inclusion~\eqref{eq:qinc}
is well defined.
\begin{prop} \label{p:iL-x_0} Let $a_1, \ldots, a_m \in
   H_*(M;\mathbb{C})$ be elements of pure degree which consist of a
   basis for the total homology $H_*(M;\mathbb{C})$. Denote by
   $a_1^{\#}, \ldots, a_m^{\#}$ the dual basis with respect to
   intersection product.  Then:
   \begin{equation} \label{eq:iL-x_0-1} i_L([x_0]) =
      \sum_{i=1}^{m}I(a_i^{\#})a_i \; \in \; QH(M;
      \mathcal{O}(\mathcal{W}_1) \otimes \Lambda).
   \end{equation}
\end{prop}
Note that we can always take $a_1 = [\textnormal{pt}] \in
H_0(M;\mathbb{C})$ to be the class of a point and $a_m = [M] \in
H_{2n}(M;\mathbb{C})$ to be the fundamental class. We will then have
$a_1^{\#}=[M]$ and $a_m^{\#} = [\textnormal{pt}]$ and
formula~\eqref{eq:iL-x_0-1} becomes:
\begin{equation} \label{eq:iL-x_0-2} i_L([x_0]) = [\textnormal{pt}] +
   \sum_{i=2}^{m-1}I(a_i^{\#})a_i + I([\textnormal{pt}])[M].
\end{equation}

In order to prove Proposition~\ref{p:iL-x_0} we will use the
augmentation map $\epsilon_L : QH(L;\mathcal{R}) \longrightarrow
\mathcal{R}$, defined for every $\widetilde{\Lambda}^+$--algebra
$\mathcal{R}$. The precise definition and properties of this map can
be found in~\cite{Bi-Co:rigidity} (see e.g. Theorem~A in that paper).
The augmentation map $\epsilon_L$ is induced by a map
$\widetilde{\epsilon_L} : \mathcal{C}(f, \rho, J; \mathcal{R})
\longrightarrow \mathcal{R}$ which is defined as follows.  Assume that
$f$ has a unique minimum $x_0$, then $\widetilde{\epsilon_L}(x_0) = 1$
and for every $x \in \textnormal{Crit}(f)$, $x \neq x_0$,
$\widetilde{\epsilon_L}(x)=0$. It satisfies the following identity
\begin{equation} \label{eq:aug-1} \langle \textnormal{PD}(b),
   i_L(\beta) \rangle = \epsilon_L(b * \beta), \quad \forall \; b \in
   H_*(M;\mathbb{C})\subset QH(M;\mathcal{R}), \; \beta \in
   QH(L;\mathcal{R}),
\end{equation}
where $\textnormal{PD}$ stand for Poincar\'{e} duality and $\langle
\cdot, \cdot \rangle$ for the obvious $\mathcal{R}$--linear extension
of the Kronecker pairing.

We are now ready to prove Proposition~\ref{p:iL-x_0}.
\begin{proof}[Proof of Proposition~\ref{p:iL-x_0}]
   Write $i_L([x_0]) = \sum_{i=1}^{m}\varphi_i a_i$, with $\varphi_i
   \in \mathcal{O}(\mathcal{W}_1) \otimes \Lambda$. Apply now
   formula~\eqref{eq:aug-1} with $b=a_j^{\#}$, $\beta=[x_0]$. We
   obtain $$\varphi_j = \langle \textnormal{PD}(a_j^{\#}), i_L([x_0])
   \rangle = \epsilon_L(a_j^{\#} * [x_0]) =
   \epsilon_L(I(a_j^{\#})[x_0]) = I(a_j^{\#}),$$ where the one to last
   equality follows from Proposition~\ref{p:a*alpha}.
\end{proof}

\begin{rem} \label{r:Ia} In a similar way one can prove that
   \begin{equation} \label{eq:Ia} I(a) = \langle \textnormal{PD}(a),
      i_L([x_0])\rangle, \quad \forall \; a \in H_*(M;\mathbb{C}).
   \end{equation}
\end{rem}

\subsection{The Frobenius structure and the quantum Euler class}
\label{sb:frob} The quantum homology $QH(M;\Lambda)$ has the structure
of a Frobenius algebra. In this section we explain how to translate
this structure via the isomorphism $I$ to the Jacobian ring
$\mathcal{O}(\mathcal{W}_1) \otimes \Lambda$. We remark that this
translation has been recently established by Fukaya, Oh, Ohta and
Ono~\cite{FO3:toric-3}. Below we explain our point of view on the
subject and how it is related to our theory.

\subsubsection{Generalities on Frobenius algebras}
\label{sbsb:frob-gnrl} We first recall some basic notions about
Frobenius algebras. The reader is referred
to~\cite{Abrams:semi-simple} and the references therein for the
general theory of Frobenius algebras.

Let $\mathcal{A}$ be an algebra over a ring $\mathcal{R}$ and assume
that $\mathcal{A}$ is a free finite-rank module over $\mathcal{R}$. A
Frobenius structure on $\mathcal{A}$ is an $\mathcal{R}$--linear map
$F:\mathcal{A} \longrightarrow \mathcal{R}$ such that the associated
bilinear pairing
$$\mathcal{A} \otimes_{\mathcal{R}} \mathcal{A} \longrightarrow
\mathcal{R}, \quad a\otimes b \longmapsto F(ab)$$ is
non-singular in the sense that the induced map $\mathcal{A}
\longrightarrow \textnormal{Hom}_{\mathcal{R}}(\mathcal{A},
\mathcal{R})$, $a \longmapsto F(a \cdot -)$, is invertible (or put in
different terms, the associated matrix of the pairing is invertible in
some basis of $\mathcal{A}$ over $\mathcal{R}$). Of course, the
associated bilinear pairing of a Frobenius structures can be viewed as
a generalization of the notion of Poincar\'{e} duality.  Note that
when the ring $\mathcal{R}$ is not a field some authors (e.g.
Abrams~\cite{Abrams:semi-simple}) use the notion of Frobenius extension
rather than Frobenius structure.

To a Frobenius structure one can associate an invariant called the
Euler class, introduced by Abrams~\cite{Abrams:semi-simple}. This is
defined as follows. Pick a basis $a_1, \ldots, a_m$ of $\mathcal{A}$
over $\mathcal{R}$. Let $a_1\spcheck, \ldots, a_m\spcheck$ be the dual
basis with respect to the Frobenius pairing. The Euler class
$\mathscr{E}(\mathcal{A}, F)$ is defined as:
\begin{equation} \label{eq:euler-A} \mathscr{E}(\mathcal{A}, F) =
   \sum_{i=1}^m a_i a_i\spcheck.
\end{equation}
It is straightforward to check that $\mathscr{E}(\mathcal{A}, F)$ does
not depend on the choice of the basis. The importance of the Euler
class comes form the following theorem.
\begin{thm}[Abrams~\cite{Abrams:semi-simple, Abrams:thesis}]
\label{t:Abrams} Let
   $\mathcal{A}$ be a finite dimensional algebra over a field
   $\mathcal{R}$ of characteristic $0$. Then:
   \begin{enumerate}
     \item For every two Frobenius structures $F'$ and $F''$ on
      $\mathcal{A}$ there exists an invertible element $u \in
      \mathcal{A}$ such that $F'' = u F'$. Moreover we have:
      $\mathscr{E}(\mathcal{A}, F'') = u^{-1} \mathscr{E}(\mathcal{A},
      F'')$. Thus the Euler class does not depend on the Frobenius
      structure upto multiplication by an invertible element. In
      particular, whether or not the Euler class is a zero divisor, or
      whether or not it is invertible, does not depend on the
      particular choice of the Frobenius structure.
     \item Suppose that the Euler class of some (hence for every)
      Frobenius structure on $\mathcal{A}$ is not a zero divisor.
      Then the Euler classes determine the Frobenius structures on
      $\mathcal{A}$ in the sense that there exists a unique Frobenius
      structure $F$ on $\mathcal{A}$ with a given Euler class.
     \item The algebra $\mathcal{A}$ is semi-simple iff the Euler
      class $\mathscr{E}(\mathcal{A}, F)$ is invertible for some
      Frobenius structures $F$ on $\mathcal{A}$.
   \end{enumerate}
\end{thm}

We also have the following result that will be relevant for our
purposes.
\begin{thm}[Scheja-Storch~\cite{Sch-St:spur}] \label{t:scheja-storch}
   Let $(\mathcal{A}, F)$ be a Frobenius algebra over a field
   $\mathcal{R}$ of characteristic $0$. Suppose that $\mathcal{A}$ can
   be written as $\mathcal{A} = \mathcal{R}[x_1, \ldots, x_r]/
   \mathcal{I}$ for some ideal $\mathcal{I}$ which is generated by $r$
   elements $f_1, \ldots, f_r \in \mathcal{R}[x_1, \ldots, x_r]$. Put
   $$J = \det \Bigl(\frac{\partial f_i}{\partial x_j}\Bigr)_{i,j} \in
   \mathcal{A}.$$ If $J \neq 0$ then $J = u \mathscr{E}(\mathcal{A},
   F)$ for some invertible element $u \in \mathcal{A}$.
\end{thm}

\subsubsection{The main examples} \label{sbsb:frob-exp} Here are two
examples that are relevant in our context. The first one is classical.
Let $M$ be a closed manifold and $\mathcal{R}$ any ring.  Assume for
simplicity that $H_i(M;\mathcal{R})=0$ for every odd $i$.  Let
$\mathcal{A} = H_*(M;\mathcal{R})$ endowed with the intersection
product $\cdot$. Write $$H_*(M;\mathcal{R}) = \mathcal{R}
[\textnormal{pt}] \; \bigoplus \bigoplus_{j=1}^{\dim M/2}
H_{2j}(M;\mathcal{R}),$$ where $[\textnormal{pt}]$ is the class of a
point. The Frobenius structure $F$ is defined by the projection onto
the $\mathcal{R} [\textnormal{pt}]$ factor. In other words, $F(a)$ is
defined to be the coefficient of $a$ at $[\textnormal{pt}]$. The
associated bilinear pairing is precisely the intersection pairing. A
simple computation shows that the Euler class $\mathscr{E} =
\mathscr{E}(\mathcal{A}, F)$ in this case is exactly
$\chi(M)[\textnormal{pt}]$.

The second example, which is the one we will focus on, is the quantum
cohomology of a symplectic manifold $M$. We assume that $(M, \omega)$
is a closed monotone symplectic manifold and that
$H_i(M;\mathbb{C})=0$ for every odd $i$. Put
$\mathcal{A}=QH(M;\Lambda)$ endowed with the quantum product $*$ and
let $\mathcal{R} = \Lambda$. The Frobenius structure is taken as in
the preceding example, i.e. for $a \in QH(M;\Lambda)$ we set $F(a)\in
\Lambda$ to be the coefficient of $a$ at $[\textnormal{pt}]$. We
denote it from now on by $F_Q$ to emphasize the relation to quantum
homology. The fact that this is indeed a Frobenius structure is not
immediate.  It is proved e.g. in~\cite{Abrams:semi-simple}.

We now turn to the Euler class of the Frobenius structure $F_Q$ on the
quantum homology. We denote it for simplicity by $\mathscr{E}_{Q}$ and
call it the {\em quantum Euler class}. Under the assumptions that $(M,
\omega)$ is monotone and $H_{\textnormal{odd}}(M;\mathbb{C})=0$ we
have the following:
\begin{lem} \label{l:dual-basis} Let $\mathbf{a}= \{a_1, \ldots,
   a_m\}$ be a basis for $H_*(M;\mathbb{C})$ consisting of elements of
   pure degree and $\mathbf{a}^{\#} = \{a_1^{\#}, \ldots, a_m^{\#}\}
   \in H_*(M;\mathbb{C})$ be the dual basis with respect to the {\em
     classical} intersection product. Then $\mathbf{a}^{\#}$ is also a
   dual basis with respect to the quantum product $*$. In particular
   we have
   \begin{equation} \label{eq:EQ} \mathscr{E}_Q = \sum_{i=1}^m
      a_i*a_i^{\#}.
   \end{equation}
   This class belongs to $QH_0(M;\Lambda)$ and is a deformation of the
   classical Euler class, i.e. $\mathscr{E}_Q =
   \chi(M)[\textnormal{pt}] + h.o.(t)$, where $h.o.(t)$ stands for
   higher order terms in $t$.
\end{lem}
The proof can be found in~\cite{Bi-Co:qrel-long} (see
Proposition~6.5.7 and the proof of Proposition~6.5.8 in that paper).

Note that $\Lambda$ is not a field hence Theorems~\ref{t:Abrams}
and~\ref{t:scheja-storch} do not apply for $QH(M;\Lambda)$. To go
around this difficulty we can work with the completion
$\widehat{\Lambda} = \mathbb{C}[t^{-1},t]]$ consisting of formal
Laurent series in $t$ with finitely many terms having negative powers
of $t$. Note that $\widehat{\Lambda}$ is a field. We can define in a
straightforward way $QH(M;\widehat{\Lambda})$, endowed with the
quantum product and we have an inclusion of rings $QH(M;\Lambda)
\subset QH(M; \widehat{\Lambda})$. Obviously the preceding Frobenius
structure $F_Q$ extends to $QH(M;\widehat{\Lambda})$ and the quantum
Euler class remains exactly the same.

\subsubsection{Back to toric manifolds} \label{sbsb:back-to-toric}
We now return to the case of toric manifolds.

Suppose that the superpotential $\mathscr{P}:(\mathbb{C}^*)^{\times n}
\longrightarrow \mathbb{C}$ is a Morse function (in the holomorphic
sense), i.e. it has only isolated critical points and at each such
point the holomorphic Hessian is non-degenerate. In this case
$\mathcal{W}_1$ is a scheme consisting of a finite number of points
each coming with multiplicity $1$. Therefore
$\mathcal{O}(\mathcal{W}_1) = \bigoplus_{z \in \mathcal{W}_1}
{\mathbb{C}}$, hence by the isomorphism~\eqref{eq:iso-I} from
Theorem~\ref{t:Bat-Giv} it follows that the quantum cohomology
$QH(M;\Lambda)$ splits as: $$QH(M;\Lambda) \cong \bigoplus_{z \in
  \mathcal{W}_1} \Lambda,$$ and similarly for
$QH(M;\widehat{\Lambda})$. It follows that $QH(M;\widehat{\Lambda})$
is semi-simple. It turns out that the converse direction is also true,
hence $QH(M;\widehat{\Lambda})$ is semi-simple iff $\mathscr{P}$ is
Morse (see~\cite{Os-Ty:qh} for the proof and for more on
semi-simplicity of $QH$ for toric manifolds).

We now address the question of how does the isomorphism $I$ from
Theorem~\ref{t:Bat-Giv} translate the quantum Frobenius structure from
$QH(M;\widehat{\Lambda})$ to $\mathcal{O}(\mathcal{W}_1) \otimes
\widehat{\Lambda}$.  Theorem~\ref{t:scheja-storch} provides a partial
answer. Write
$$\mathcal{O}(\mathcal{W}_1) \otimes \widehat{\Lambda} =
\widehat{\Lambda}[z_1, u_1, \ldots, z_n, u_n] / \mathcal{I},$$ where
$\mathcal{I}$ is the ideal generated by $$
\partial_{z_1}\mathscr{P}, \ldots, \partial_{z_n} \mathscr{P}, z_1
u_1-1, \ldots, z_n u_n-1.$$ Applying Theorem~\ref{t:scheja-storch} we
obtain that there exists an invertible element $u \in
QH(M;\widehat{\Lambda})$ such that:
$$I(\mathscr{E}_Q) = u z_1 \cdots z_n
\det \Bigl(\frac{\partial^2 \mathscr{P}}{\partial z_i \partial
  z_j}\Bigr)_{i,j}.$$ Since $z_1 \cdots z_n$ is invertible we obtain
from~\eqref{eq:discr-2} that there exists an invertible element $v \in
QH(M;\widehat{\Lambda})$ such that $$I(\mathscr{E}_Q) = v \Delta,$$
where $\Delta$ is the discriminant introduced in~\S\ref{sb:discr}.
Since $I(\mathscr{E}_Q)$ has degree $-2n$ so must have $v$. Since $v$
has pure degree it follows that both $v$ as well as its inverse
$v^{-1}$ in fact lie in $QH(M; \Lambda)$ (i.e. we do not need the
larger field of coefficients $\widehat{\Lambda}$). These
considerations are still far from determining the precise value of
$v$. The following theorem provides this additional information.

\begin{thm} \label{t:qh-frob-struct} Suppose that $\mathscr{P}$ is
   Morse. Then:
   \begin{enumerate}
     \item $I(\mathscr{E}_Q) = (-1)^{n+1}\Delta t^n$, where $\Delta
      \in \mathcal{O}(\mathcal{W}_1)$ is the discriminant introduced
      in~\eqref{eq:discr-2} of~\S\ref{sb:discr}.
     \item Via the isomorphism $I$, the quantum Frobenius structure on
      $\mathcal{O}(\mathcal{W}_1) \otimes \Lambda$ has the form:
      \begin{equation} \label{eq:q-frob} F_Q(I^{-1}(\sigma)) =
         \frac{(-1)^{n+1}}{t^n} \sum_{z \in \mathcal{W}_1}
         \frac{\sigma(z)}{\Delta(z)}, \quad \forall \; \sigma \in
         \mathcal{O}(\mathcal{W}_1) \otimes \Lambda.
      \end{equation}
   \end{enumerate}
\end{thm}
Theorem~\ref{t:qh-frob-struct} (stated in a slightly different form)
has been recently proved by Fukaya, Oh, Ohta and
Ono~\cite{FO3:toric-3} by methods of Floer theory. It seems to be
known for a long time to specialists in quantum homology theory.  In
fact, Givental has pointed out to us~\cite{Gi:private-com} that this
theorem follows from his work~\cite{Gi:elliptic-GW} (see
Proposition~1.1 in that paper). Below in~\S\ref{s:exp} we verify
Theorem~\ref{t:qh-frob-struct} by direct computation on all toric
monotone $4$--manifolds. We sketch in \S\ref{sbsb:frob-more} a more
conceptual proof of this Theorem.

Denote by $F_1, \ldots, F_r$ the codimension-$1$ facets of the moment
polytope $P = \textnormal{image\,}(\mathfrak{m})$ and by
$\overrightarrow{v_1}, \ldots, \overrightarrow{v}_r$ the inwards
pointing normal integral primitive vectors to these facets as
in~\S\ref{sb:toric-setting}. For a subset of indices $I \subset \{1,
\ldots, r\}$ write $\overrightarrow{v_I} = \sum_{i \in I}
\overrightarrow{v_i}$. The following identities, which seem to bear
some arithmetic nature, follow immediately from
Theorem~\ref{t:qh-frob-struct}.
\begin{cor} \label{c:frob-sums} Assume that $\mathscr{P}$ is Morse.
   Let $I \subset \{1, \ldots, r\}$ be a subset of indices.  If \, $\#
   I < n$ then:
   \begin{equation} \label{eq:frob-sums-1} \sum_{z \in \mathcal{W}_1}
      \frac{z^{\overrightarrow{v_I}}}{\Delta(z)} = 0.
   \end{equation}
   If \, $\#I = n$ then:
   \begin{equation} \label{eq:frob-sums-2} \sum_{z \in \mathcal{W}_1}
      \frac{z^{\overrightarrow{v_I}}}{\Delta(z)} =
      \begin{cases}
         0, & \textnormal{if } \; \cap_{i \in I} F_i = \emptyset, \\
         (-1)^{n+1},& \textnormal{if } \; \cap_{i \in I} F_i \neq
         \emptyset.
      \end{cases}
   \end{equation}
\end{cor}

\begin{proof}
   Write $\Sigma_i = \mathfrak{m}^{-1}(F_i)$. The $\Sigma_i$'s are
   codimension--$2$ symplectic submanifolds of $M$.  Recall that by
   the isomorphism $I$ of Theorem~\ref{t:Bat-Giv} we have
   $I([\Sigma_i]) = z^{\overrightarrow{v_i}} t$, hence
   $$I(*_{i \in I} [\Sigma_i]) =
   z^{\overrightarrow{v_I}}t^{\# I}.$$ By formula~\eqref{eq:q-frob},
   the value of the sum
   \begin{equation} \label{eq:sum-W_1}
      \sum_{z \in \mathcal{W}_1}
      \frac{z^{\overrightarrow{v_I}}}{\Delta(z)}
   \end{equation}
   is determined by the value of $F_Q(*_{i \in I} [\Sigma_i])$, i.e.
   by whether or not $*_{i \in I} [\Sigma_i]$ contains
   $[\textnormal{pt}]$. But by degree reasons the coefficient of
   $[\textnormal{pt}]$ in $*_{i \in I} [\Sigma_i]$ is the same as the
   coefficient of $[\textnormal{pt}]$ in $\cdot_{i \in I} [\Sigma_i]$
   where $\cdot$ is the classical intersection product.  The rest of
   the proof now follows from basic intersection properties of the
   $\Sigma_i$'s.
\end{proof}

Finally, putting together Theorem~\ref{t:qh-frob-struct} with
formulae~\eqref{eq:discr-toric-1},~\eqref{eq:discr-toric-2} we obtain
the following:
\begin{cor} \label{c:Euler-class} Suppose that $\mathscr{P}$ is Morse.
   Then the quantum Euler class admits the following expressions:
   \begin{equation} \label{eq:eu-1} \mathscr{E}_Q = \sum_{\substack{I
          \subset \{1, \ldots, r\} \\ \#I=n}} \bigl(*_{i \in I}
      [\Sigma_i]\bigr) \det(A_I)^2,
   \end{equation}
   where $A_I$ is defined in~\eqref{eq:A_I} and $*$ stands for the
   quantum product.
   \begin{equation} \label{eq:eu-2} \mathscr{E}_Q = \det
      \Bigl(\sum_{k=1}^r v_k^i v_k^j [\Sigma_k]\Bigr)_{i,j},
   \end{equation}
   where the determinant here should be evaluated in the quantum
   homology ring.
\end{cor}

\subsubsection{Further remarks on Theorem~\ref{t:qh-frob-struct} and
  its proof} \label{sbsb:frob-more}

Note that by Theorem~\ref{t:Abrams} the Frobenius structure $F_Q$ is
determined by its associated Euler class $\mathcal{E}_Q$. Therefore
point~(2) of Theorem~\ref{t:qh-frob-struct} follows immediately from
point~(1). The next Proposition shows that $\mathcal{E}_Q$ is indeed
very much related to the ``Lagrangian picture''.

Consider the morphism:
$$j_L: QH_*(M;\mathcal{O}(\mathcal{W}_1) \otimes \Lambda)
\longrightarrow QH_{*-n}(L;\mathcal{W}_1), \quad a \longmapsto
a*[L].$$ Consider $[x_0] \in QH_*(L;\mathcal{W}_1)$ as in the
discussion before Proposition~\ref{p:iL-x_0}. We have:
\begin{prop} \label{p:j-circ-i-E_Q} $j_L \circ i_L ([x_0]) =
   I(\mathcal{E}_Q) [L]$.
\end{prop}
\begin{proof}
   This follows at once from Propositions~\ref{p:iL-x_0}
   and~\ref{p:a*alpha}.
\end{proof}

Thus, the proof of Theorem \ref{t:qh-frob-struct} reduces to showing
that
\begin{equation}\label{eq:i-j-Frob}
j_{L}\circ i_{L}([x_{0}])=(-1)^{n+1}\Delta t^{n}[L]~.~
\end{equation}

We sketch here our argument for this identity in dimension $2n=4$.
Recall from \cite{Bi-Co:Yasha-fest} \S 8.3 that given two Lagrangians
$L$ and $L'$ there is a particular formula allowing to express
$j_{L'}\circ i_{L}$.  In our case, we ultimately want to study $L=L'$
so it is sufficient to assume that $L$ is Hamiltonian isotopic to $L'$
(and $L$ is transverse to $L'$) so that the formula has the form:
\begin{equation}
   \label{eq:chain_htpy-j-i}
   j_{L'}\circ i_{L}-\chi_{L,L'}=\Phi_{L,L'}\circ d 
   +d'\circ \Phi_{L,L'}.
\end{equation}

We now explain the formula (\ref{eq:chain_htpy-j-i}). We will then
notice that from this formula we can easily deduce a closely related
one that directly computes $j_{L}\circ i_{L}$ in terms of some pearly
like configurations. Identity (\ref{eq:i-j-Frob}) follows from further
identities involving these configurations.

The notation in (\ref{eq:chain_htpy-j-i}) is as follows:
$(\mathcal{C}(L;f), d)$ is a pearl complex for $L$,
$(\mathcal{C}(L';f'), d')$ is a pearl complex for $L'$ (we assume
appropriate Riemannian metrics fixed on $L$ and $L'$), $\Phi_{L,L'}$
is a certain chain homotopy and $\chi_{L,L'}$ is a chain map that we
will describe in more detail below.  In our case we may assume that
$f$ and $f'$ are perfect Morse functions so that $d=0=d'$ because $L$
and $L'$ are wide tori.  Thus we deduce $j_{L'}\circ
i_{L}([x_{0}])=\chi_{L,L'}([x_{0}])$.  The map $\chi_{L,L'}$ is
described in \S 8.3 of \cite{Bi-Co:Yasha-fest}. Explicitly, it is
defined as follows. For $x\in \Crit{f}$
$$\chi_{L,L'}(x)=\sum_{p,y}\#(\mathcal{N}(p,p;x,y))yt^{k_{y}}$$
where $y\in \Crit(f')$, $p\in L\cap L'$ , $|y|-2k_{y}=|x|-4$ and the
moduli spaces $\mathcal{N}(p,p;x,y)$ are formed by configurations $(u,
v,v')$ where: $u$ is a Floer strip joining the intersection point $p$
to itself and with $u(\R\times \{0\})\subset L$, $u(\R
\times\{1\})\subset L'$; $v$ is a chain of pearls on $L$ joining $x$
to the point $u(0,0)$; $v'$ is a chain of pearls on $L'$ joining $x$
to the point $u(0,1)$.  Because $\chi_{L,L'}$ is a chain map it is
easily seen that we may apply the PSS construction to return from $L'$
back to $L$. This gives rise to another map $\bar{\chi}_{L,L}$ with
two properties:
\begin{itemize}
  \item[-] it verifies a formula similar to (\ref{eq:chain_htpy-j-i})
   except that only involving $L$:
   \begin{equation}
      \label{eq:j-i-L}
      j_{L}\circ i_{L}=\bar{\chi}_{L,L}
   \end{equation}
  \item[-] the definition of $\bar{\chi}_{L,L}$ is similar to that of
   $\chi_{L,L'}$ with the following modifications: $v$ is a string of
   pearls associated to the function $f:L\to \R$, $v'$ is a string of
   pearls associated to the function $f':L\to \R$, $u$ is now also a
   string of pearls associated to a third function $f''$ and joining a
   critical point $p\in\Crit(f'')$ {\em to the same $p$}.  The
   incidence conditions among these three strings of pearls are that
   there is a disk in $u$ (possibly trivial) so that $v$ ends at
   $u(-i)$ and $v'$ ends at $u(i)$.
\end{itemize}
In our case, we are interested in the case when $x$=$x_{0}=\min(f)$.
For degree reasons we see that the only term that matters corresponds
to $y=z_{2}=\max(f')$. Moreover, there is a single disk involved which
is of Maslov class $4$.  In short, $j_{L}\circ i_{L}([x_{0}])$ is
estimated by the number of elements in the moduli space $W(f'',x_{0},
z_{2},J)$ of configurations formed by a single $J$-holomorphic disk
$u$ of Maslov class $4$ and so that $u(-i)=x_{0}$, $u(+i)=z_{2}$ and
there is a critical point $p\in \Crit(f'')$ with the property that a
negative gradient trajectory of $f''$ exiting $p$ reaches $u(-1)$ and
there is a negative gradient trajectory of $f''$ that carries $u(+1)$
to $p$ again (any one of these trajectories can also be degenerate).
The next step is to include $W(f'',x_{0},z_{2},J)$ as boundary in a
$1$-dimensional moduli space whose other end had $-\Delta$ elements.
The first step is rather easy - the $1$-dimensional moduli space in
question, $W'(f'',x_{0}, z_{2},J)$, corresponds to gluing at the point
$p$ - so that the configurations contained in this moduli space are
like the ones in $W(f'',x_{0},z_{2},J)$ except that the two flow lines
there are replaced by a single one that joins $u(+1)$ to $u(-1)$
without breaking at $p$. Finally, it is essentially a delicate
combinatorial verification - that we will not include here - to see
that the number of the other boundary components of
$W'(f'',x_{0},z_{2},J)$ gives precisely $-\Delta$.

%% file: examples.tex
\section{Examples} \label{s:exp}

Here we work out examples of the various objects and invariants
constructed in the previous sections, mainly in the case of toric
manifolds. We use here the notation introduced in~\S\ref{s:toric} and
in particular for $\mathcal{W}_1$ we use the coordinates $(z_1,
\ldots, z_r)$ introduced in~\S\ref{sbsb:sp-formulae} and for
$\mathcal{W}_2$ we use the coordinates $(\xi_1, \ldots, \xi_r)$
introduced in~\S\ref{sb:formulae-W_2}.

\subsection{The complex projective space} \label{sb:cpn} Consider
${\mathbb{C}}P^n$ endowed with its standard Fubini-Study K\"{a}hler
structure $\omega_{\textnormal{FS}}$ normalized so that
$\int_{\mathbb{C}P^1} \omega_{\textnormal{FS}} = 1$. Consider the
Hamiltonian torus action $(\theta_1, \ldots, \theta_n) \cdot [z_0:
\cdots: z_n] = [z_0:e^{-2\pi i \theta_1}: \cdots: e^{-2 \pi i
  \theta_n}z_n]$. The moment polytope is the standard simplex
$$P=\Bigl\{(x_1, \ldots, x_n) \in \mathbb{R}^n \Bigm|
0 \leq x_k \, \forall k, \sum_{i=1}^n x_i \leq 1 \Bigr\}.$$ It has
$n+1$ codimension--$1$ facets with normal vectors
$\overrightarrow{v_i} = (0, \ldots, 1, \ldots, 0)$ (where the $1$ is
in the $i$'th coordinate), $i=1, \ldots, n$ and
$\overrightarrow{v_{n+1}} = (-1, \ldots, -1)$. (See
e.g.~\cite{Audin:torus-actions, McD-Sa:Intro}). The monotone torus $$L
= \mathfrak{m}^{-1}\Bigl(\frac{1}{n+1}, \ldots, \frac{1}{n+1}\Bigr) =
\bigl\{[z_0: \ldots: z_n] \bigm| \, |z_0| = \ldots = |z_n| \bigr\}$$
is the Clifford torus. The wide variety $\mathcal{W}_2$ is given in
this case by
$$\mathcal{W}_2 = \{ (\xi, \ldots, \xi) \mid \xi \in \mathbb{C}^*\}
\cong \mathbb{C}^*.$$
The superpotential is:
$$\mathscr{P}(z_1, \ldots, z_n) =
\sum_{i=1}^n z_i + \frac{1}{z_1 \cdots z_n}.$$ A simple computation
shows that $\mathscr{P}$ is Morse. The wide variety $\mathcal{W}_1$
consists of the following $n+1$ points:
$$\mathcal{W}_1 = \{(z, \ldots, z) \mid z^{n+1}=1\},$$
and each of them comes with multiplicity $1$.  The quadratic form
(see~\eqref{eq:discr-1}) $\varphi_{_{\mathcal{W}}}$ is given in the
basis $\{C_1, \ldots, C_n\}$ by $$\varphi_{_{\mathcal{W}}}(X_1,
\ldots, X_n) = \xi \Bigl(\sum_{i=1}^n X_i^2 + \sum_{i<j} X_i X_j
\Bigr), \; \forall \xi \in \mathcal{W}_2.$$ A simple computation shows
that the discriminant of the quadratic form (on $\mathcal{W}_2$ and
$\mathcal{W}_1$ respectively) is:
$$\Delta(\xi) = (-1)^{n+1}(n+1)\xi^n, \; 
\forall \xi \in \mathcal{W}_2, \quad \Delta(z) = (-1)^{n+1}(n+1) z^n,
\; \forall z \in \mathcal{W}_1.$$ Denote by
$H=[{\mathbb{C}}P^{n-1}]\in QH_{2n-2}({\mathbb{C}}P^n;\Lambda)$ the
class of a linear hyperplane and by $[{\mathbb{C}}P^l] \in
QH_{2l}(M;\Lambda)$ the class of a linear projective $l$-dimensional
plane. The quantum homology of ${\mathbb{C}}P^n$ is given by
\begin{equation*}
   H^{*k} =
   \begin{cases}
      [{\mathbb{C}}P^{n-k}], & \textnormal{if }
      \; 0 \leq k \leq n, \\
      [{\mathbb{C}}P^n]t^{n+1}, & \textnormal{if } \; k=n+1.
   \end{cases}
\end{equation*}
A simple computation shows that the quantum Euler class equals the
topological one: $$\mathscr{E}_Q = \mathscr{E}_{\textnormal{top}} =
(n+1)[\textnormal{pt}].$$

The ring $\mathcal{O}(\mathcal{W}_1)$ is:
$$\mathcal{O}(\mathcal{W}_1) \cong \mathbb{C}[z^{\pm 1}] /
\langle z^{n+1}=1 \rangle,$$ and the isomorphism $I$
from~\eqref{eq:iso-I} satisfies $I([{\mathbb{C}}P^l])=z^{n-l}t^{n-l}$.
One can easily verify that $I(\mathscr{E}_Q) = -t^n \Delta(z)$.

The identities of Corollary~\ref{c:frob-sums} now read:
\begin{equation*}
   \frac{1}{n+1} \sum_{\{z \mid z^{n+1}=1\}} z^k =
   \begin{cases}
      0, & \textnormal{if } \; 1 \leq k \leq n,\\
      1, & \textnormal{if } \;  k=n+1.
   \end{cases}
\end{equation*}
Finally, the quantum inclusion of $[x_0]$ is given by:
$$i_L([x_0])=[\textnormal{pt}] +
\sum_{k=1}^n z^k[{\mathbb{C}}P^{k}]t^k, \quad \forall z \in
\mathcal{W}_1.$$

\medskip Next we will exemplify our theory on all the monotone toric
$4$-manifolds. Recall that apart from ${\mathbb{C}}P^2$ there are
exactly four of them, namely $S^2 \times S^2$ and the blow up of
${\mathbb{C}}P^2$ at $1\leq k \leq 3$ points.

\subsection{$S^2 \times S^2$ } \label{sb:S2xS2} Consider $M = S^2
\times S^2$ with the balanced symplectic form $\omega = \omega_{S^2}
\oplus \omega_{S^2}$ and with the obvious Hamiltonian torus action
coming from circle actions on both factors. The moment polytope is
$$P=\{(x_1, x_2) \in \mathbb{R}^2 \mid 0 \leq x_1 \leq 1, \; \; 0\leq x_2
\leq 1\}.$$ The monotone torus is $L = \mathfrak{m}^{-1}(\tfrac{1}{2},
\tfrac{1}{2})$ which is the product of two equators coming from each
$S^2$--factor. The integral normal vectors to the four facets are
$\overrightarrow{v_1}=(1,0)$, $\overrightarrow{v_2}=(0,1)$,
$\overrightarrow{v_3} = (-1,0)$, $\overrightarrow{v_4}=(0,-1)$.  The
wide variety $\mathcal{W}_2$ is given by: $$\mathcal{W}_2 = \{(\xi_1,
\xi_2, \xi_1, \xi_2) \mid \xi_1, \xi_2 \in \mathbb{C}^*\} \cong
\mathbb{C}^* \times \mathbb{C}^*.$$ The superpotential is:
$$\mathscr{P}(z_1, z_2) = z_1 + z_2 + \frac{1}{z_1} + \frac{1}{z_2}.$$
This function is Morse and its critical points are: $$\mathcal{W}_1 =
\{(1,1), (1,-1), (-1, 1), (-1,-1)\}.$$ The quadratic form is
$$\varphi_{_{\mathcal{W}}}(X_1, X_2) = \xi_1 X_1^2 + \xi_2 X_2^2, \; \forall
(\xi_1, \xi_2) \in \mathbb{C}^* \times \mathbb{C}^*.$$ The
discriminant on $\mathcal{W}_2$ and $\mathcal{W}_1$, respectively, is:
$$\Delta(\xi_1, \xi_2) = -4\xi_1 \xi_2, \quad \Delta(z_1, z_2) = -4z_1
z_2.$$

To describe the quantum homology $QH(M;\Lambda)$ of $M$, put $A = [S^2
\times \textnormal{pt}], B=[\textnormal{pt} \times S^2] \in
QH_2(M;\Lambda)$. Then we have:
$$A*B = \textnormal{pt}, \quad A*A = B*B = [M]t^2.$$
The isomorphism $I$ satisfies:
$$I(A)=z_2t, \quad I(B)=z_1t, \quad I([\textnormal{pt}])=z_1 z_2 t^2.$$

The quantum Euler class equals in this case to the topological one:
$\mathscr{E}_Q = 4 [\textnormal{pt}]$. The quantum inclusion
satisfies:
$$i_L([x_0]) = [\textnormal{pt}] +z_1At + z_2Bt+z_1z_2[M]t^2.$$ The
arithmetic identities of Corollary~\ref{c:frob-sums} can be verified
by a straightforward direct substitution.

\subsection{Blow ups of ${\mathbb{C}}P^2$} \label{sb:bl-cp2} Consider
the standard Hamiltonian torus action on ${\mathbb{C}}P^2$ and let $p$
be a fixed point of the action. This action has exactly three fixed
points $p_1, p_2, p_3$. By blowing up $p_1, \ldots, p_k$, $1 \leq k
\leq 3$, we obtain a manifold $M_k$ which can be endowed with a
monotone symplectic form $\omega$ in such a way that the torus action
on ${\mathbb{C}}P^2$ lifts to a Hamiltonian torus action on $M_k$
(see~\cite{Audin:torus-actions, McD-Sa:Intro} for details). Denote by
$E_i \in H_2(M_k;\mathbb{Z})$ the exceptional divisor over $p_i$ and
by $L \in H_2(M_k;\mathbb{Z})$ the homology class of a projective line
not passing through the exceptional divisors. We denote by $[M_k] \in
H_4(M_k;\mathbb{Z})$ the fundamental class. The Poincar\'{e} dual of
the cohomology class of $\omega$ satisfies: $\textnormal{PD}[\omega] =
L - \tfrac{1}{3}\sum_{i=1}^k E_i$. We will now go over each of the
cases $k=1, 2, 3$.

\subsection{The blow-up of ${\mathbb{C}}P^2$ at one point}
\label{sb:bl-1-cp2}
Denote by $M_1 = \textnormal{Bl}_{p_1}({\mathbb{C}}P^2)$ the blow-up
of ${\mathbb{C}}P^2$ at $p_1$. The moment polytope and the normal
vectors to the facets are depicted in figure~\ref{f:bl-1-cp2}. Note
that: $$[\mathfrak{m}^{-1}(F_1)]=E, \quad
[\mathfrak{m}^{-1}(F_2)]=L-E, \quad [\mathfrak{m}^{-1}(F_3)]=L, \quad
[\mathfrak{m}^{-1}(F_4)]=L-E.$$

\begin{figure}[htbp]
      \psfig{file=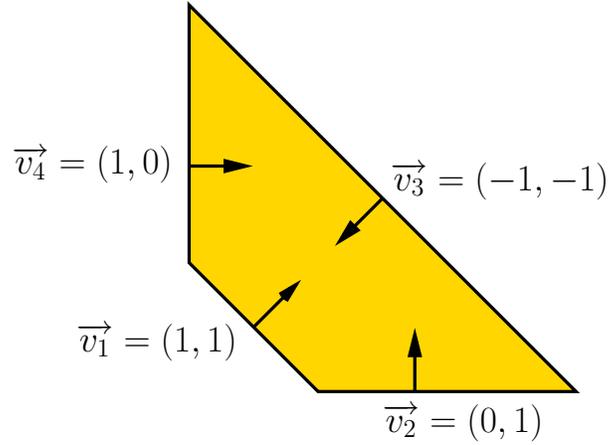, width=0.5 \linewidth}
      \caption{The moment polytope of the blow-up of ${\mathbb{C}}P^2$
        at one point.}
      \label{f:bl-1-cp2}
\end{figure}

The wide variety $\mathcal{W}_2$ is:
$$\mathcal{W}_2 = \{(\xi_1, \xi_2, \xi_1+\xi_2, \xi_2)
\mid \xi_1, \xi_2 \in \mathbb{C}^*, \xi_1 \neq -\xi_2\}.$$ Note that
the trivial representation $(1,1,1,1)$ does not belong to
$\mathcal{W}_2$, so $L$ is narrow with respect to this representation.
The superpotential is:
$$\mathscr{P}(z_1, z_2) = z_1 + z_2 + z_1 z_2 + \frac{1}{z_1 z_2}.$$
The wide variety $\mathcal{W}_1$ consists of $4$ points, all with
multiplicity $1$, and is given by
$$\mathcal{W}_1 = \{(z,z) \mid z^4 +z^3-1=0\}.$$
The ring of functions over $\mathcal{W}_1$ is therefore:
$$\mathcal{O}(\mathcal{W}_1) \cong \mathbb{C}[z,z^{-1}]/
\langle z^4 +z^3-1=0 \rangle.$$

The quadratic form is $$\varphi_{_{\mathcal{W}}}(X_1, X_2) = (\xi_1 +
\xi_2)X_1^2 + (2\xi_1 + \xi_2)X_1 X_2 + (\xi_1 + \xi_2)X_2^2, \;
\forall (\xi_1, \xi_2) \in \mathbb{C}^* \times \mathbb{C}^*.$$ The
discriminant on $\mathcal{W}_2$ and $\mathcal{W}_1$ respectively is:
$$\Delta(\xi_1, \xi_2) = -(4\xi_1 \xi_2 + 3 \xi_2^2), \quad \Delta(z)
= -z^2(4z+3).$$ The quantum product is given by
(see~\cite{Cra-Mir:QH}): $$E*E = -[\textnormal{pt}] + Et + [M_1]t^2,
\quad E*L= [M_1]t^2, \quad L*L = [\textnormal{pt}] + [M_1]t^2.$$ The
quantum Euler class is: $$\mathscr{E}_Q = 4[\textnormal{pt}] - Et.$$

The isomorphism $I$ is given by:
$$I(L) = \frac{1}{z^2}t, \quad I(E)=z^2t, \quad
I([\textnormal{pt}])=(\frac{1}{z^4}-1)t^2.$$

The fact that $I(\mathscr{E}_Q) = -\Delta t^2$ on $\mathcal{W}_1$ can
be verified here by a direct (though long) computation.

The quantum inclusion satisfies:
$$i_L([x_0]) = [\textnormal{pt}] + \frac{1}{z^2}Lt - z^2 Et +
(\frac{1}{z^4}-1)[M_1]t^2.$$

We now turn to the arithmetic identities of
Corollary~\ref{c:frob-sums}. In the following identity $a(z)$ stands
for the function $z^k$, where $-2 \leq k \leq 3$. We have:
\begin{equation} \label{eq:arith-id-M_1} \sum_{\{z:z^4+z^3-1=0\}}
   \frac{a(z)}{4z^3+3z^2} =
   \begin{cases}
      0, & \textnormal{if } a(z) \; \textnormal{is one of } 1, z, z^2,
      \frac{1}{z^2}, \\
      1, & \textnormal{if } a(z) \; \textnormal{is one of } z^3,
      \frac{1}{z}.
   \end{cases}
\end{equation}
These identities seem non-trivial to obtain by a direct computation,
though they can be verified using a numerical mathematical program
such as Matlab, Mathematica or Octave. An alternative elementary
(albeit non-direct) verification of these identities via computations
of residues of rational functions, has been recently pointed out to us
by Andrew Granville~\cite{Granville:private}.

\subsection{The blow-up of ${\mathbb{C}}P^2$ at two points}
\label{sb:bl-2-cp2}
Let $M_2 = \textnormal{Bl}_{p_1, p_2}({\mathbb{C}}P^2)$ be the blow-up
of ${\mathbb{C}}P^2$ at the two points $p_1, p_2$. The moment polytope
and the normal vectors to the facets are depicted in
figure~\ref{f:bl-2-cp2}. Note that:
\begin{align*}
   & [\mathfrak{m}^{-1}(F_1)]=L-E_1, \quad
   [\mathfrak{m}^{-1}(F_2)]=L-E_2,
   \quad [\mathfrak{m}^{-1}(F_3)]=E_2,\\
   & [\mathfrak{m}^{-1}(F_4)]=L-E_1-E_2, \quad
   [\mathfrak{m}^{-1}(F_5)]=E_1.
\end{align*}

\begin{figure}[htbp]
      \psfig{file=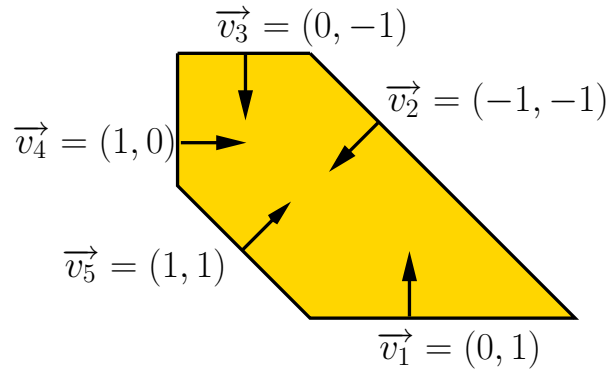, width=0.5 \linewidth}
      \caption{The moment polytope of the blow-up of ${\mathbb{C}}P^2$
        at two points.}
      \label{f:bl-2-cp2}
\end{figure}

The wide variety $\mathcal{W}_2$ is:
$$\mathcal{W}_2 = \{(\xi_1, \xi_2, \xi_3, \xi_1-\xi_3, -\xi_1+\xi_2+\xi_3)
\mid \xi_1, \xi_2, \xi_3 \in \mathbb{C}^*, \xi_1 \neq \xi_3, \xi_1
\neq \xi_2 + \xi_3\}.$$ Note that the point $(1,1,1,1,1)$ does not
belong to $\mathcal{W}_2$, thus $L$ is narrow with respect to the
trivial representation $\rho \equiv 1$.

The superpotential is:
$$\mathscr{P}(z_1, z_2) = z_1 + z_2 + z_1z_2 + \frac{1}{z_2} + \frac{1}{z_1z_2}.$$
The wide variety $\mathcal{W}_1$ consists of $5$ points, all with
multiplicity $1$, and is given by
$$\mathcal{W}_1 = \Bigl\{ \bigl(-1, \frac{-1 \pm \sqrt{5}}{2}\bigr) \Bigr\} \cup
\Bigl\{ \bigl(\frac{1}{z^2},z \bigr) \mid z^3-z-1=0 \Bigr\}.$$ The
quadratic form is
$$\varphi_{_{\mathcal{W}}}(X_1, X_2) = \xi_2 X_1^2 +
(-\xi_1 + 2\xi_2+\xi_3)X_1X_2 + (\xi_2+\xi_3)X_2^2, \; \forall (\xi_1,
\xi_2, \xi_3) \in (\mathbb{C}^*)^{\times 3}.$$ The discriminant on
$\mathcal{W}_2$ and $\mathcal{W}_1$ respectively is:
$$\Delta(\xi_1, \xi_2, \xi_3) = (\xi_1-\xi_3)^2 - 4\xi_1 \xi_2,
\quad \Delta(z_1, z_2) = \bigl( z_2-\frac{1}{z_2} \bigr)^2
-\frac{4}{z_1}.$$

The quantum product is given by (see~\cite{Cra-Mir:QH}):
\begin{align*}
   & L*L = [\textnormal{pt}] + (L-E_1-E_2)t + 2[M_2]t^2, \quad E_1*E_2
   = (L-E_1-E_2)t, \\
   & E_1*E_1 = -[\textnormal{pt}]+(L-E_2)t+[M_2]t^2, \quad
   L*E_1=L*E_2 = (L-E_1-E_2)t + [M_2]t^2, \\
   & E_2*E_2 = -[\textnormal{pt}]+(L-E_1)t+[M_2]t^2.
\end{align*}
The quantum Euler class turns out to be: $$\mathscr{E}_Q =
5[\textnormal{pt}]-Lt.$$

The isomorphism $I$ is given by:
$$I(L)=(z_1z_2+z_2)t, \quad I(E_1)=z_1z_2t, \quad
I(E_2)=\frac{1}{z_2}t, \quad I([\textnormal{pt}]) =
(1+z_2-\frac{1}{z_2^2})t.$$

The quantum inclusion satisfies:
$$i_L([x_0]) = [\textnormal{pt}] -z_1z_2E_1t -
\frac{1}{z_2}E_2t + (z_1z_2+z_2)Lt +
(1+z_2-\frac{1}{z_2^2})[M_2]t^2.$$

The arithmetic identities of Corollary~\ref{c:frob-sums} become:
\begin{equation} \label{eq:arith-id-M_2} \sum_{(z_1, z_2) \in
     \mathcal{W}_1} \frac{a(z_1,
     z_2)}{(z_2-\frac{1}{z_2})^2-\frac{4}{z_1}} =
   \begin{cases}
      0, & \textnormal{if } a(z_1, z_2) \; \textnormal{is one of } 1,
      z_1,
      z_2, z_1z_2, \frac{1}{z_2}, \frac{1}{z_1z_2}, \\
      -1, & \textnormal{if } a(z_1,z_2) \; \textnormal{is one of }
      \frac{1}{z_1}, z_1^2 z_2, z_1 z_2^2, \frac{z_1}{z_2},
      \frac{1}{z_1 z_2^2}.
   \end{cases}
\end{equation}
As with the previous case, $M_1$, it seems non-trivial to verify these
identities by a direct computation.

\subsection{The blow-up of ${\mathbb{C}}P^2$ at three points}
\label{sb:bl-3-cp2}
The moment polytope and the normal vectors to the facets are depicted
in figure~\ref{f:bl-3-cp2}. Note that:
\begin{align*}
   & [\mathfrak{m}^{-1}(F_1)]=L-E_1-E_2, \quad
   [\mathfrak{m}^{-1}(F_2)]=E_2,
   \quad [\mathfrak{m}^{-1}(F_3)]=L-E_2-E_3,\\
   & [\mathfrak{m}^{-1}(F_4)]=E_3, \quad
   [\mathfrak{m}^{-1}(F_5)]=L-E_1-E_3, \quad
   [\mathfrak{m}^{-1}(F_6)]=E_1.
\end{align*}

\begin{figure}[htbp]
      \psfig{file=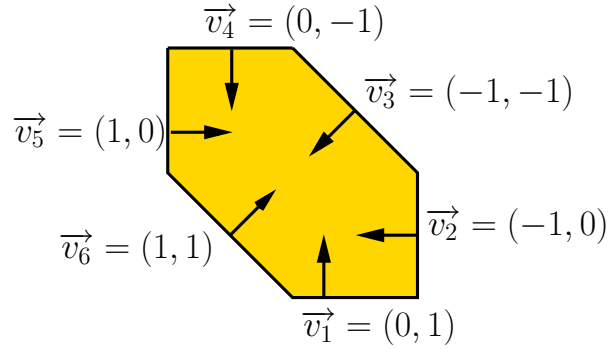, width=0.5 \linewidth}
      \caption{The moment polytope of the blow-up of ${\mathbb{C}}P^2$
        at two points.}
      \label{f:bl-3-cp2}
\end{figure}

The wide variety $\mathcal{W}_2$ is:
$$\mathcal{W}_2 = \{(\xi_1, \xi_2, \xi_3, \xi_4, \xi_1+\xi_2-\xi_4,
-\xi_1+\xi_3+\xi_4) \mid \xi_1, \xi_2, \xi_3, \xi_4 \in \mathbb{C}^*,
\xi_4 \neq \xi_1+\xi_2, \xi_1 \neq \xi_3 + \xi_4\}.$$

The superpotential is:
$$\mathscr{P}(z_1, z_2) = z_1 + z_2 + z_1z_2 + \frac{1}{z_1} +
\frac{1}{z_2} + \frac{1}{z_1z_2}.$$
The wide variety $\mathcal{W}_1$ consists of $6$ points, all with
multiplicity $1$, and is given by
$$\mathcal{W}_1 = \Bigl\{(1,1), (1,-1), (-1,1), (-1,-1),
(e^{2\pi i/3}, e^{2 \pi i/3}), (e^{4 \pi i/3}, e^{4 \pi i/3})
\Bigr\}.$$

The quadratic form is $$\varphi_{_{\mathcal{W}}}(X_1, X_2) =
(\xi_2+\xi_3)X_1^2 + (-\xi_1+2\xi_3 + \xi_4)X_1 X_2 +
(\xi_3+\xi_4)X_2^2, \;\; \forall \,(\xi_1, \xi_2, \xi_3, \xi_4) \in
(\mathbb{C}^*)^{\times 4}.$$ The discriminant on $\mathcal{W}_2$ is:
$$\Delta(\xi_1, \xi_2, \xi_3, \xi_4) = (\xi_1-\xi_4)^2 -
4(\xi_1 \xi_3 + \xi_2 \xi_3 + \xi_2 \xi_4),$$ and on $\mathcal{W}_1$
it is: $$\Delta(z_1, z_2) = \Bigl(z_2-\frac{1}{z_2}\Bigr)^2 -
\frac{4}{z_1}(1+z_2+z_1z_2).$$

The quantum product is given by (see~\cite{Cra-Mir:QH}):
\begin{align*}
   & L*L = [\textnormal{pt}] + (3L-2E_1-2E_2-2E_3)t + 3[M_3]t^2, \quad
   E_i * E_j = (L-E_i-E_j)t \;\; \forall i\neq j, \\
   & E_i*E_i = -[\textnormal{pt}] + (2L-E_1-E_2-E_3)t + [M_3]t^2,
   \quad L*E_i = (2L-E_1-E_2-E_3-E_i)t + [M_3]t^2.
\end{align*}
The quantum Euler class turns out to be: $$\mathscr{E}_Q =
6[\textnormal{pt}]-(3L - E_1 - E_2 - E_3)t.$$

The isomorphism $I$ is given by:
\begin{align*}
   & I(L)=(z_2+z_1 z_2 + \frac{1}{z_1})t, \quad I(E_1)=z_1 z_2 t,
   \quad I(E_2) = \frac{1}{z_1}t, \quad I(E_3)=\frac{1}{z_2}t,\\
   & I([\textnormal{pt}]) = ((1+z_1)(1+z_2)-z_1^2 z_2^2)t^2.
\end{align*}
The quantum inclusion satisfies:
$$i_L([x_0]) = [\textnormal{pt}] + (z_2+z_1 z_2 + \tfrac{1}{z_1})L t
-z_1 z_2 E_1t - \tfrac{1}{z_1}E_2 t - \tfrac{1}{z_2}E_3 t +
((1+z_1)(1+z_2) - z_1^2 z_2^2)[M_3]t^2.$$

The arithmetic identities of Corollary~\ref{c:frob-sums} become:
\begin{equation} \label{eq:arith-id-M_3} \sum_{(z_1, z_2) \in
     \mathcal{W}_1} \frac{a(z_1,
     z_2)}{(z_2-\frac{1}{z_2})^2-\frac{4}{z_1}(1+z_2+z_1z_2)} =
   \begin{cases}
      0, & \textnormal{if } a(z_1, z_2) \; \textnormal{is one of } 1,
      z_1,
      z_2, z_1z_2, \frac{1}{z_1}, \frac{1}{z_2}, \frac{1}{z_1z_2}, \\
      -1, & \textnormal{if } a(z_1,z_2) \; \textnormal{is one of }
      \frac{z_2}{z_1}, \frac{z_1}{z_2}, \frac{1}{z_1^2, z_2},
      \frac{1}{z_1 z_2^2}, z_1^2 z_2, z_1 z_2^2.
   \end{cases}
\end{equation}

\subsection{The Chekanov torus in ${\mathbb{C}}P^2$}
\label{sb:chekanov-torus}

This is a non-toric example. The Lagrangian torus which we will
describe below was discovered by
Chekanov~\cite{Chekanov:exotic-torus} (where he proved that a version
of this torus, lying in $\mathbb{R}^4$ is not Hamiltonianly isotopic
to the standard split torus). Further studies of holomorphic disks
with boundary on this torus were later carried out by
Eliashberg-Polterovich~\cite{El-Po:lag-knots} and more recently by
Chekanov and Schlenk~\cite{Che-Schl:tori}. A very nice exposition of
the subject and calculations of the related superpotential have been
carried out by Auroux~\cite{Aur:t-duality}. Below we partially follow
the notation from the latter paper.

Let $\gamma \subset \mathbb{C}$ be a closed embedded curve which does
not enclose $0 \in \mathbb{C}$. The Chekanov torus
$\mathbb{T}_{\gamma} \subset {\mathbb{C}}P^2$ is:
$$\mathbb{T}_{\gamma} = \bigl\{[x:y:1]
\in {\mathbb{C}}P^2 \mid xy \in \gamma \; \& \; |x|=|y|\bigr\}.$$ To
simplify the story, in what follows we will use a specific choice of
$\gamma$, namely $$\gamma = \bigl\{p_0 + re^{i\theta} \mid \theta \in
[0,2 \pi]\bigr\},$$ where $p_0 \in (0,\infty)$ and $0<|r|<p_0$. See
figure~\ref{f:gamma}.

\begin{figure}[htbp]
   \psfig{file=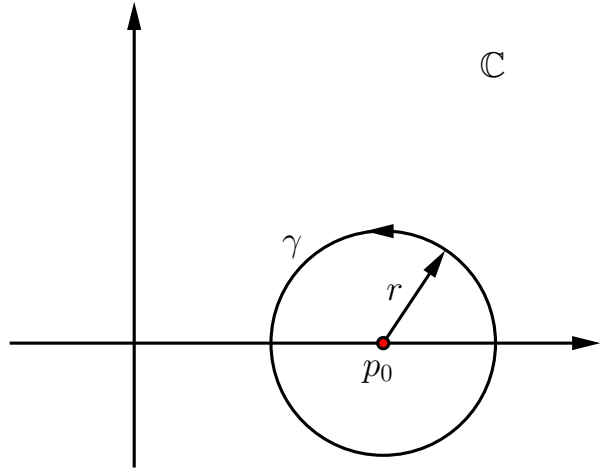, width=0.5 \linewidth}
      \caption{Construction of the Chekanov torus.}
      \label{f:gamma}
\end{figure}

We now fix a basis for $H_2^D = H_2^D({\mathbb{C}}P^2,
\mathbb{T}_{\gamma})$. For this end, let $\mathcal{U} \subset
\mathbb{C}$ be the bounded domain with $\partial
\overline{\mathcal{U}} = \gamma$ and choose a domain $\mathcal{V}
\subset \mathbb{C}$ such that $\{x^2 \mid x \in
\overline{\mathcal{V}}\} = \overline{\mathcal{U}}$. The map
$$\overline{\mathcal{V}} \ni x \longmapsto [x:x:1] \in
{\mathbb{C}}P^2$$ parametrizes a disk with boundary in
$\mathbb{T}_{\gamma}$. We denote the homology class of this disk by
$\beta \in H_2^D$. Next, consider the disk $\{[x:\bar{x}:1] \mid
|x|\leq p_0 - r\}$. This is a Lagrangian disk with boundary on
$\mathbb{T}_{\gamma}$. Denote its homology class by $\alpha \in
H_2^D$. Finally, denote by $L=[\mathbb{C}P^1]$ the homology class of a
line in ${\mathbb{C}}P^2$, viewed as an element of $H_2^D$.  The
Maslov indices of these three elements are: $\mu(\alpha)=0$,
$\mu(\beta) = 2$, $\mu(L)=6$.

According to Chekanov-Schlenk~\cite{Che-Schl:tori}
(see~\cite{Aur:t-duality} for a detailed proof) the set
$\mathcal{E}_2$ of homology classes represented by holomorphic disks
of Maslov index $2$ consists of $4$ elements and is as follows:
$\mathcal{E}_2 = \{A,B,C,D\}$, where
\begin{align*}
   & A = L-2\beta, \;\; \nu(A)=2, \quad B=L-2\beta+\alpha, \;\;
   \nu(B)=1, \\
   & C = \beta, \;\; \nu(C)=1, \quad \quad \quad \;
   D = L-2\beta-\alpha,
   \;\; \nu(D)=1.
\end{align*}
Here $\nu$ stands for the degree of the corresponding evaluation map,
as described in~\S\ref{sb:superpotential}.

Note that $\{A, B, C\}$ form a $\mathbb{Z}$--basis for $H_2^D$ and we
will use this basis to identify $\textnormal{Hom}(H_2^D, \mathbb{C}^*)
\cong (\mathbb{C}^*)^{\times 3}$. We will denote elements in this
space by $(\xi_A, \xi_B, \xi_C)$.

A straightforward calculation now show that
$$\mathcal{W}_2 = \{(\xi_A, \xi_A, 8\xi_A) \mid
\xi_A \in \mathbb{C}^*\}.$$ Note that the trivial representation
$(1,1,1)$ does not belong to $\mathcal{W}_2$. To write down the
quadratic form we will use the basis $\{a, b\}$ for
$H_1(\mathbb{T}_{\gamma};\mathbb{Z})$, where $a=\partial \alpha$,
$b=\partial \beta$.  In this basis, the quadratic form (on
$\mathcal{W}_2$) is:
$$\varphi_{_{\mathcal{W}_2}}(X_1, X_2) = 12 \xi_A X_1^2 + \xi_A X_2^2,
\quad \forall \; (\xi_A, \xi_A, 8\xi_A) \in \mathcal{W}_2,$$ and the
discriminant is $\Delta(\xi_A) = -48\xi_A^2$.

Next we describe $\mathcal{W}_1$. For this we choose as a basis for
$H_1=H_1(\mathbb{T}_{\gamma};\mathbb{Z})$ the elements $a=\partial
\alpha$, $b = \partial \beta$ and via this basis we write elements of
$\textnormal{Hom}(H_1, \mathbb{C}^*)$ as $(z_a, z_b) \in \mathbb{C}^*
\times \mathbb{C}^*$.
With this notation the superpotential is given by:
$$\mathscr{P}(z_a, z_b) = \frac{2}{z_b^2} + \frac{z_a}{z_b^2} +
\frac{1}{z_a z_b^2} + z_b.$$ A straightforward computation shows that
$\mathcal{W}_1 = \{(1, 2), (1, 2e^{2 \pi i/3}), (1, 2e^{4 \pi
  i/3})\}$. The quadratic form, in the basis $\{a,b\}$ (not in the
dual basis $\{b,-a\}$ !) is:
$$\varphi_{_{\mathcal{W}_1}}(X_1, X_2) = 12z_b^{-2}X_1^2 +
z_b^{-2}X_2^2, \quad \forall \; z_b \in \{1, e^{2\pi i/3}, e^{4 \pi
  i/3}\},$$ and the discriminant is $\Delta(z_b) = -\frac{6}{z_b}$.

%% file: orientations.tex
\section{Orientations} \label{a:orientations}
\subsection{Orientations -- general conventions} \label{sb:orient} In
order to define the pearl complex over a general ground ring we now
describe how to orient the moduli space of pearl trajectories.  

Below we denote orientations on vector spaces or manifolds $V$ by
$o_V$. We often denote dimensions of manifolds $V$ by $|V|$.

\subsubsection{Exact sequences}
Let $0 \longrightarrow F \stackrel{i}{\longrightarrow}
E\stackrel{p}{\longrightarrow} B \longrightarrow 0$ be a short exact
sequence of finite dimensional vector spaces. Orientations on any two
of these spaces induces an orientation on the third as follows.  Pick
a right inverse $s: B \longrightarrow E$ of $p$, so that $E = s(B) +
i(F)$. We require that $o_E = s(o_B) + i(o_F)$. Clearly the definition
is independent of the choice of $s$. Thus we orient exact sequence by
reading them from ``right to left'' rather than vice-versa.  We remark
that this is consistent with the standard orientation on products,
i.e. $o_{(B \times F)} = o_B + o_F$.

\subsubsection{Fibrations}
Orienting exact sequence implies a convention for the orientation of fibrations.
Namely, let $\pi: E \longrightarrow B$ be a (locally trivial smooth)
fibration with fiber $F$.  Given orientations on two of $F, E, B$ we
orient the third according to the exact sequence $0 \longrightarrow
TF \stackrel{Di}{\longrightarrow} TE\stackrel{Dp}{\longrightarrow} TB
\longrightarrow 0$, where $i$ is the inclusion of the fiber in $E$.

\subsubsection{Group actions and quotients} \label{sbsb:group-actions}
A special important case of orientations on fibrations is the
following. Let X be an oriented manifold and $K$ an oriented Lie group
acting freely on X. We orient the quotient space $X/K$ by viewing $X
\longrightarrow X/K$ as a fibration. Equivalently, we use the
exact sequence: $0 \longrightarrow T_x(K \cdot x) \longrightarrow
T_x X \longrightarrow T_{[x]}(X/K) \longrightarrow 0$.

\subsubsection{Orienting boundaries of manifolds}
Let $W$ be an oriented manifold with boundary, then the orientation of
$\partial W$ is such that $\overrightarrow{n}+o_{\partial W}  =
o_{W}$, where $\overrightarrow{n}$ is an exterior pointing vector to
$\partial W$.

\subsubsection{Normal bundles}
Let $W$ be an oriented manifold and $V \subset W$ an oriented
submanifold. We orient the normal bundle $\leftexp{\nu}{V} = TW/TV$ of
$V$ by the exact sequence $0 \longrightarrow TV \longrightarrow TW
\longrightarrow \leftexp{\nu}{V} \longrightarrow 0$, or by abuse of
notation $o_{(\leftexp{\nu}{V})} + o_V = o_W$.

\subsubsection{Preimages}
Let $U$ and $W$ be oriented manifolds and $V \subset W$ an oriented
submanifold.  Let $f: U \longrightarrow W$ be a map transverse to $V$.
We orient $f^{-1}(V)$ as follows. We first orient the normal bundle of
$f^{-1}(V)$ in $U$, by pulling back the orientation of
$\leftexp{\nu}{V}$ via the isomorphism $Df: \leftexp{\nu}{f^{-1}(V)}
\longrightarrow \leftexp{\nu}{V}$. The orientation on
$\leftexp{\nu}{f^{-1}(V)}$ induces an orientation on $f^{-1}(V)$.

\subsubsection{Intersections} \label{sbsb:intersections}
If $U,V$ are two transverse oriented submanifolds of an oriented
manifold $W$.  We orient $U \cap V$ via the exact sequence $0
\longrightarrow T(U \cap V) \longrightarrow TW \longrightarrow
\leftexp{\nu}{U} \oplus \leftexp{\nu}{V} \longrightarrow 0$. In other
words we have $\leftexp{\nu}U \oplus \leftexp{\nu}V \oplus T(U\cap
V)=TW$ as oriented vector spaces.

\subsubsection{Fiber products} \label{sbsb:fiber-product} Here we use
a convention taken from~\cite{FO3}, though our presentation is
somewhat different. Let $e_{i}: V_{i}\to X$, $i=1,2$, be two
transverse smooth maps, where $V_1, V_2, X$ are oriented manifolds.
Denote by $\Delta\subset X\times X$ the diagonal.  We denote by
$V_{1}\times_{X} V_{2}$ the submanifold
$(e_{1},e_{2})^{-1}(\Delta)\subset V_{1}\times V_{2}$ endowed with the
following orientation -- which is, in general, \emph{different} from
the standard preimage orientation. At the level of tangent spaces
there exists an exact sequence
$$0 \longrightarrow K\longrightarrow TV_{1}\oplus TX\oplus TV_{2}
\stackrel{h}{\longrightarrow} TX\oplus TX \longrightarrow 0$$ where
$h(v_1,x,v_2)=(De_{1}(v_1)-x,x-De_{2}(v_2))$, and $K$ is the kernel of
$h$. Note that $K$ is canonically identified with the tangent space of
$(e_{1},e_{2})^{-1}(\Delta)$ under the map $(v_1,x,v_2)\to (v_1,v_2)$.
Following our conventions above, the kernel $K$ above inherits an
orientation from those of $V_{1}, V_{2}, X$.  The \emph{fiber product}
orientation of $V_{1}\times_{X} V_{2}$ is induced by that of $K$.  We
will sometimes denote this fiber product also by $V_{1} \;
\lrsub{e_1}{\times}{e_2} V_{2}$ in case we need to make explicit the
maps $e_{1}, e_{2}$.

It is easy to see that our fiber product convention coincides with
that in~\cite{FO3}. In case $V_{1}$ and $V_{2}$ are oriented
submanifolds of $X$ and the two evaluations are just the respective
inclusions one can check that, as oriented submanifolds,
$V_{1}\times_{X} V_{2}=V_{2}\cap V_{1}$.

The motivation for introducing the fiber product orientation is that
it verifies an important associativity property. If $e_{1}:U\to X$,
$e_{2}:V\to X$, $f_1 :V\to Y$, $f_2:W\to Y$ are smooth maps with the
appropriate transversality conditions, then we have an oriented
equality $$(U \times_{X} V) \times_{Y} W = U \times_{X} (V \times_Y
W).$$ This is easily seen by noticing that both orientations can be
viewed as induced by the kernel orientation in the short exact
sequence:
$$0 \longrightarrow K \longrightarrow  TU \oplus TX \oplus TV \oplus 
TY\oplus TW\stackrel{h'}{\longrightarrow} TX \oplus TX \oplus TY
\oplus TY \longrightarrow 0$$ with
$h'(u,x,v,y,w)=(De_{1}(u)-x,x-De_{2}(v), Df_{1}(v)-y,y-Df_{2}(w))$.
Obviously, a similar formula remains valid for longer iterated fiber
products.

Another important feature of the fiber product is its behavior with
respect to taking boundaries (see~\cite{FO3}). Ler $U$, $V$ be
oriented manifolds possibly with boundary and $X$ an oriented manifold
without boundary.  Let $e:U \longrightarrow X$, $f:V \longrightarrow
X$ be two transverse maps. Then we have the following ``Leibniz''
formula for fiber products:
\begin{equation} \label{eq:fiber-Leibniz}
   \partial (U \times_X V) = (\partial U) \times_X V \; \coprod \;
   (-1)^{\scriptscriptstyle |X|-|U|}\, U \times_X \partial V.
\end{equation} 

\subsubsection{Lagrangian submanifolds}
Throughout the paper, by a Lagrangian $L\subset (M,\omega)$ we mean an
oriented Lagrangian submanifold together with a fixed spin structure.

\subsubsection{The group of biholomorphisms of the disk $Aut(D)$} 
\label{sbsb:aut-D}
Denote by $D \subset \mathbb{C}$ the closed unit disk. We orient its
boundary $\partial D$ by the counterclockwise orientation.

Denote by $G = Aut(D)$ the group of biholomorphisms of the disk. We
orient $G$ as follows.  Every element in $G$ can be written uniquely
as $$\sigma_{\theta, \alpha} (z) = e^{i \theta} \frac{z+\alpha}{1+
  \bar{\alpha}z}, \quad \textnormal{with } \theta \in [0,2 \pi),
\alpha \in \textnormal{Int\,}D.$$ This gives an identification between
$G$ and $[0, 2\pi) \times \textnormal{Int\,}D$ according to which we
orient $G$.

Denote by $H \subset G$ the subgroup of elements that preserve the two
points $-1, +1 \in \partial D$. This $1$-dimensional subgroup consists
of the elements $\sigma_{0,\alpha}$ with $\alpha \in (-1,1)$. We
orient $H$ by the orientation of the interval $(-1,1)$. Note that here
our conventions are different from those of~\cite{FO3}. In our case
$\sigma_{0,\alpha}(0) \longrightarrow +1$ (respectively $-1$) when
$\alpha \longrightarrow +1$ (respectively $-1$), whereas in~\cite{FO3}
it is vice-versa, thus our orientation of $H$ is the opposite of the
one used in~\cite{FO3}.

\subsubsection{Moduli spaces of holomorphic disks}
\label{sbsb:moduli-of-disks} Fix a generic almost complex structure $J
\in \mathcal{J}$. Let $B \in H_2^D$.  Denote by
$\widetilde{\mathcal{M}}(B,J)$ the space of (parametrized)
$J$-holomorphic disks $u:(D, \partial D) \longrightarrow (M,L)$ with
$u_*([D]) = B$. It is well-known by the work~\cite{FO3} that a spin
structure on L induces orientations on the moduli spaces
$\widetilde{\mathcal{M}}(B,J)$. Given $\zeta \in D$ (resp. $\partial
D$) we denote by $e_{\zeta}:\widetilde{\mathcal{M}}(B,J)
\longrightarrow M$ (resp. $L$) the evaluation map given by
$e_{\zeta}(u) = u(\zeta)$.

Let $p, q \geq 0$ and consider the space of (parametrized)
$J$-holomorphic disks with $p$-marked points on the boundary and $q$
marked points in the interior: $\widetilde{\mathcal{M}}_{p,q}(B,J) =
\widetilde{\mathcal{M}}_{p,q}(B,J) \times T_{p,q}$, where $T_{p,q}
\subset (\partial D)^{\times p} \times (\textnormal{Int\,}(D))^{\times
  q}$ is the open set consisting of all tuples of points
$(\underline{z}, \underline{\xi}) = (z_1, \ldots, z_p, \xi_1, \ldots,
\xi_q)$ with the properties that the $z_i$'s are all distinct, the
$\xi_j$ are all distinct and in addition if $p \geq 3$ the points
$z_1, \ldots, z_q$ are required to be in cyclic order along $\partial
D$ with respect to the standard (counterclockwise) orientation. As
$T_{p,q}$ is an open subset of $(\partial D)^{\times p} \times
(\textnormal{Int\,}(D))^{\times q}$ it inherits an orientation from
the latter. Apart from that we will require that $B \neq 0$ when $p
\leq 2$ and $q=0$ or when $p=0$ and $q=1$.

We let $G = Aut(D)$ (as wells as subgroups of it) act on
$\widetilde{\mathcal{M}}_{p,q}(B,J)$ as follows. If $\sigma \in G$ and
$(u, z_1, \ldots, z_p, \xi_1, \ldots, \xi_q) \in
\widetilde{\mathcal{M}}_{p,q}(B,J)$ define $$\sigma \cdot (u, z_1,
\ldots, z_p, \xi_1, \ldots, \xi_q) = (u \circ \sigma^{-1},
\sigma(z_1), \ldots, \sigma(z_p), \sigma(\xi_1), \ldots,
\sigma(\xi_q)).$$ We denote the space of disks with marked points by
$\mathcal{M}_{p,q}(B,J) = \widetilde{\mathcal{M}}_{p,q}(B,J)/G$, with
the orientation induced from the preceding conventions.  This space
comes with evaluation maps $E_{i,-}: \mathcal{M}_{p,q}(B,J)
\longrightarrow L$ and $E_{-,j}: \mathcal{M}_{p,q}(B,J)
\longrightarrow M$ defined by $E_{i,-}[u,\underline{z},
\underline{\xi}] = u(z_i)$ and $E_{-,j}[u,\underline{z},
\underline{\xi}] = u(\xi_j)$.

In what follows it will be often useful to deal with quotients by the
group $H \subset G$ of those elements that fix the points $-1,1 \in
D$, namely with $\widetilde{\mathcal{M}}(B,J)/H$. Recall that we have
oriented $H$ in~\S\ref{sbsb:aut-D} above. The space
$\widetilde{\mathcal{M}}(B,J)/H$ comes with two evaluation maps
$e_{\scriptscriptstyle -1}, e_{\scriptscriptstyle +1} :
\widetilde{\mathcal{M}}(B,J)/H \longrightarrow L$, defined by
$e_{\scriptscriptstyle -1}[u] = u(-1)$ and $e_{\scriptscriptstyle
  +1}[u] = u(+1)$.

With these conventions it is not hard to verify that the following
maps are orientation preserving diffeomorphisms:
\begin{equation} \label{eq:marked-pts}
   \begin{aligned}
      & \widetilde{\mathcal{M}}(B,J)/H \longrightarrow
      \mathcal{M}_{2,0}(B,J),
      \quad [u] \longmapsto [u,1,-1], \\
      & \widetilde{M}(B,J) \longrightarrow \mathcal{M}_{1,1}(B,J),
      \quad u \longmapsto [u,1,0], \\
      & \widetilde{M}(B,J) \longrightarrow \mathcal{M}_{3,0}(B,J),
      \quad u \longmapsto [u,1,e^{2\pi i/3}, e^{4 \pi i/3}].\\
   \end{aligned}
\end{equation}
In view of the first map above we will identify
$\mathcal{M}_{2,0}(B,J)$ with $\widetilde{\mathcal{M}}(B,J)/H$ and
view $e_{\scriptscriptstyle -1}, e_{\scriptscriptstyle +1}$ as maps
defined on $\mathcal{M}_{2,0}(B,J)$.

To simplify the notation, when $q=0$, we will sometimes write
$\mathcal{M}_p(B,J)$ instead of $\mathcal{M}_{p,0}(B,J)$. We will
especially use $\mathcal{M}_2(B,J)$.

\subsubsection{Bubbling and gluing} \label{sbsb:blgl} Let $B, B', B''
\in H_2^D$ with $B = B' + B''$. Consider the fiber product
$$\mathcal{M}_{2}(B',J) \lrsub{e_{\scriptscriptstyle +1}}{\times}{
  e_{\scriptscriptstyle -1}} \mathcal{M}_{2}(B'',J),$$ where
$e_{\scriptscriptstyle \pm 1}$ are the evaluation maps at $\pm 1 \in
\partial D$. By compactness, gluing, as well as further regularity
assumptions, this spaces can be embedded into the main stratum of the
boundary of the compactification of the space
$\mathcal{M}_{2}(B,J)$:
\begin{equation} \label{eq:bndry-disks-1} \mathcal{M}_{2}(B',J)
   \lrsub{e_{\scriptscriptstyle +1}}{\times}{e_{\scriptscriptstyle
       -1}} \mathcal{M}_{2}(B'',J) \hooklongrightarrow \partial
   \overline{\mathcal{M}_{2}(B,J)}.
\end{equation}
This embedding is so that the pair of marked points $-1\in
\textnormal{dom}(u')$ and $+1\in \textnormal{dom}(u'')$ with
$(u',u'')\in \mathcal{M}_{2}(B',J)\times\mathcal{M}_{2}(B'',J)$
corresponds after gluing to the pair of marked points $-1, +1\in
\partial D$ in the domain of the glued disk $u'\#_{\tau}
u''\in\mathcal{M}_{2}(B,J)$ for all gluing parameters $\tau$.

The embedding~\eqref{eq:bndry-disks-1} is in general not orientation
preserving. In fact the orientations on the left and right hand sides
differ by $(-1)^{n-1}$. This can be proved by a direct computation
based on~\cite{FO3}. We write this fact as:
\begin{equation} \label{eq:bndry-disks-2}
   \partial_{\scriptscriptstyle \textnormal{bubble}} \,
   \overline{\mathcal{M}_{2}(B,J)} \; = \coprod_{B'+B''=B}
   (-1)^{n-1} \mathcal{M}_{2}(B',J)
   \lrsub{e_{\scriptscriptstyle +1}}{\times}
   {e_{\scriptscriptstyle -1}} \mathcal{M}_{2}(B'',J).
\end{equation}
There is a slight abuse of notation here, since the right hand side is
just part of the boundary of $\mathcal{M}_{2}(B,J)$. However for the
purpose of the pearl complex the other boundary components are not
relevant. We will also write $\partial^{\scriptscriptstyle
  (B',B'')}_{\scriptscriptstyle \textnormal{bubble}} \,
\overline{\mathcal{M}_{2}(B,J)}$ for the boundary component
in~\eqref{eq:bndry-disks-2} that corresponds to bubbling of the type
$(B', B'')$.

\begin{rem}\label{rem:comp-orient-FOOO}
   There is a subtle difference between our conventions for gluing and
   those in~\cite{FO3}. In our case for the first moduli space in the
   fiber product we evaluate at the point $+1$ and for the second at
   the point $-1$ while~\cite{FO3} use the opposite convention.
   Furthermore, our conventions for the orientation on $H$ are
   opposite to theirs. These different sign conventions turn out to
   cancel each other in this case, hence our sign $(-1)^{n-1}$
   coincides with the one that appears in~\cite{FO3}.
\end{rem}

\subsubsection{Orientations in Morse theory} \label{sbsb:orient-morse}
Let $V$ be an oriented manifold, $f:V \longrightarrow \mathbb{R}$ a
Morse function and $(\cdot, \cdot)$ a Riemannian metric. Stable and
unstable submanifolds are always taken with respect to the {\em
  negative} gradient flow of $f$ which we denote by $\Phi_t: V
\longrightarrow V$.

For every $x \in \textnormal{Crit}(f)$ fix an orientation on the
unstable submanifold $W^u(x)$. This induces an orientation on the
stable submanifolds $W^s(x)$ by requiring that $o_{W^s(x)} +
o_{W^u(x)} = o_V$.

Assume now that the pair $(f, (\cdot, \cdot))$ is Morse-Smale. Given
$x, y \in \textnormal{Crit}(f)$ we have the following spaces of
gradient trajectories connecting $x$ to $y$:
$$\widetilde{m}(x,y) = W^s(y) \cap W^u(x), \quad m(x,y) =
\widetilde{m}(x,y) / \mathbb{R},$$ where $\mathbb{R}$ acts on
$\widetilde{m}(x,y)$ by $t \cdot p = \Phi_t(p)$. All spaces here are
oriented by the conventions we have described so far. The Morse
complex (with coefficients in $\mathbb{Z}$) is now defined by $CM =
\mathbb{Z} \langle \textnormal{Crit}(f) \rangle$, $\partial: CM_*
\longrightarrow CM_{*-1}$, where
$$\partial(x) = \sum_{|y| = |x|-1}  \# m(x,y)y, \quad 
\forall x \in \textnormal{Crit}(f).$$

\subsubsection{Some useful identities for boundaries}
\label{sbsb:bnd-ident} 
We start with two useful formulae for the boundary of the stable and
unstable submanifolds of critical points. Recall that these manifold
admit a natural compactification in terms of stable and unstable
submanifolds of lower indices. Here are the signs that appear in these
boundaries. Let $(f, (\cdot, \cdot))$ be a Morse-Smale pair as
in~\S\ref{sbsb:orient-morse}. Let $x \in \textnormal{Crit}(f)$ and $x'
\in \textnormal{Crit}(f)$ with $|x'|=|x|-1$. Then the part of the
boundary of $\overline{W^u(x)}$ that involves the critical point $x'$
satisfies:
\begin{equation} \label{eq:bndry-Wu}
   \partial \overline{W^u(x)} = m(x,x') \times W^u(x').
\end{equation}
Similarly, if $y,y' \in \textnormal{Crit}(f)$ with $|y'|=|y|+1$ then
the part of the boundary of $\overline{W^s(y)}$ that involves $y'$
satisfies:
\begin{equation} \label{eq:bndry-Ws}
   \partial \overline{W^s(y)} = (-1)^{|V|-|y|} \, m(y',y) \times W^s(y').
\end{equation}

\

Next we derive some general formulas for boundaries of moduli spaces
of gradient trajectories ``connecting'' two manifolds. Consider two
oriented manifolds $X$ and $Y$ with maps $e_{\scriptscriptstyle X}:X
\longrightarrow L$ and $e_{\scriptscriptstyle Y}:Y \longrightarrow L$.
Let $\Phi_t$ be the negative gradient flow of $f$ and consider the map
$e'_{\scriptscriptstyle X}: X \times \mathbb{R}_+ \longrightarrow L$,
given by $(x,t) \longmapsto \Phi_t \circ e_{\scriptscriptstyle X}(x)$.
Finally, consider the fiber product $Z = (X \times \mathbb{R}_+)
\times_L Y$, where the first factor is mapped to $L$ by
$e'_{\scriptscriptstyle X}$ and the second one by
$e_{\scriptscriptstyle Y}$. One might think of $Z$ as the space of
gradient trajectories connecting $X$ to $Y$. Ignoring orientations for
a moment, we note that part of the boundary of $Z$ is formed by broken
trajectories, i.e. by elements of the space $(X \times_L W^s(z))
\times (W^u(z) \times_L Y)$, where $z \in \textnormal{Crit}(f)$. Here
the (un)stable submanifolds are mapped to $L$ by inclusion and $X$,
$Y$, by the maps $e_X$, $e_Y$ respectively. We denote this component
of the boundary by $\partial_z ((X \times \mathbb{R}_+) \times_L Y)$.
Taking now orientations into account one obtains by straightforward
computation the following identity:
\begin{equation} \label{eq:del-broken}
   \partial_z \Bigl((X \times \mathbb{R}_+) \times_L Y \Bigr) 
   = (-1)^{|X|} \bigl(X \times_L W^s(z) \bigr)
   \times \bigl( W^u(z) \times_L Y \bigr).
\end{equation}

Another boundary component of $(X \times \mathbb{R}_+) \times_L Y$
arises when the gradient trajectory between $X$ and $Y$ shrinks to
zero length. Ignoring orientations, the corresponding part of the
boundary can be written as $X \times_L Y$, where $X$, $Y$ are mapped
to $L$ by $e_{\scriptscriptstyle X}$, $e_{\scriptscriptstyle Y}$
respectively. We denote it by $\partial_{\scriptscriptstyle
  \textnormal{shrink}} \bigl( (X \times \mathbb{R}_+) \times_L Y
\bigr)$. Taking orientations into account, one obtains the following
identity:
\begin{equation} \label{eq:shrink}
   \partial_{\scriptscriptstyle \textnormal{shrink}} 
   \bigl( (X \times \mathbb{R}_+) \times_L Y \bigr) = (-1)^{|X|+1}(X \times_L Y).
\end{equation}

\subsection{Orientation conventions for the pearl complex}
\label{sb:or-pcomplex} Our purpose now is to describe the orientation
conventions for the various pearly moduli spaces needed. With the
conventions that we will describe, the various algebraic structures
described in~\S\ref{sb:qh} verify the usual identities in
non-commutative differential graded homological algebra. We will only
justify here some of these facts, they are relatively straightforward
but tedious exercises. We remark that for the constructions below to
work with our orientation conventions it is important that the
algebraic structures discussed here are only defined by counting
elements of $0$-dimensional moduli spaces. In our case, the main
equations of interest concern the product from (\ref{eq:qprod}) that
verifies {\em at the chain level} the equation:
\begin{equation}\label{eq:leibniz-prod} d(x\ast y)= d(x)\ast y +
   (-1)^{n-|x|} x \ast d(y)
\end{equation}
and the module action from~\S\ref{subsubsec:module} that verifies a
similar identity. Besides this we claim that the other identities:
$d^{2}=0$, associativity of the product etc are all verified with
signs as well.

\subsubsection{Orienting the space of pearly trajectories}
\label{sbsb:orient-pearl} We first recall that a string of pearls
associated to the data $\mathscr{D}=(f,(\cdot, \cdot), J)$ and joining
two points $x,y\in \Crit(f)$ can be viewed as a sequence
$(a,u_{1},t_{1},u_{2},t_{2},\ldots, u_{k}, b)$ where $a\in W^{u}(x)$,
$b\in W^{s}(y)$, $u_{i}\in \mathcal{M}_{2}(B_{i},J)$, $B_i \neq 0$,
$t_{i}\in \mathbb{R}_+$, subject to the following incidence conditions
$\Phi_{t_{i}}(u_{i}(+1))=u_{i+1}(-1)$ for $1\leq i<k$, $u_{1}(-1)=a$,
$u_{k}(+1)=b$. Here $\Phi_t$ is the negative gradient flow of $f$.
Appropriate genericity conditions are required to insure the
transversality of the relevant evaluation maps.  The resulting pearl
moduli space is denoted $\mathcal{P}(x,y;\mathscr{D};(B_{1},\ldots,
B_{k}))$. When $k=1$ we also allow $B_1=0$ and put
$\mathcal{P}(x,y;\mathscr{D}, 0) = m(x,y)$ i.e. the space of gradient
trajectories going from $x$ to $y$ as in~\S\ref{sbsb:orient-morse}
above.

All orientation conventions described below are established by
assuming that we restrict attention only to the moduli spaces
involving absolutely distinct sequences of simple disks in the sense
of~\cite{Bi-Co:rigidity, Bi-Co:qrel-long}.

The moduli space $\mathcal{P}(x,y;\mathscr{D};(B_{1},\ldots, B_{k}))$
is thus a subset of $W^{u}(x)\times (\mathcal{M}_{2}(B_{1},J)\times
\mathbb{R}_+) \times \ldots \times \mathcal{M}_{2}(B_{k},J)\times
W^{s}(y)$ obtained from a multi-diagonal in $L \times L^{\times 2k}
\times L$ by taking the preimage by a suitable evaluation map.
However, this procedure will not be used in order to orient these
spaces. For the purpose of orientations we describe $\mathcal{P}$ as
an iterated fiber product.

Let $B_1, \ldots, B_k$, $k \geq 1$, be a sequence of classes in
$H_2^D$ with $B_j \neq 0$ for all $j$. Consider the fiber product
\begin{equation} \label{eq:P-xy-fiber-prod}
   \begin{aligned}
      \mathcal{P}(x,y; \mathscr{D}; (B_1, \ldots, B_k)) = W^{u}(x)&
      \times_L (\mathcal{M}_{2}(B_{1},J)\times \mathbb{R}_+)
      \times_L
      \ldots \\
      \ldots & \times_{L} (\mathcal{M}_{2}(B_{i},J)\times
      \mathbb{R}_+) \times_L \ldots
      \\
      \ldots & \times_{L}\mathcal{M}_{2}(B_{k},J) \times_{L}
      W^{s}(y) \; ,
   \end{aligned}
\end{equation}
where the first and last factor here are mapped into $L$ by inclusion.
The $i$'th moduli space ($i<k$) is mapped to the term $L$ on its left
by $(u_i, t) \mapsto e_{\scriptscriptstyle -1}(u_i) = u_i(-1)$, and to
the term $L$ on its right by $(u_i, t) \mapsto \Phi_t \circ
e_{\scriptscriptstyle +1}(u_i) = \Phi_t(u_i(+1))$. The pre-last factor
$\mathcal{M}_{2}(B_k,J)$ is mapped to the $L$ on its left by
$e_{\scriptscriptstyle -1}$ and to the $L$ on its right by
$e_{\scriptscriptstyle +1}$. When $B=0$ we simply put
$\mathcal{P}(x,y; \mathscr{D}; 0) = m(x,y)$ without any orientation
adjustment. Next, for a fixed $0 \neq B\in H^{D}_{2}$, the disjoint
union of all the moduli spaces
$\mathcal{P}(x,y;\mathscr{D};(B_{1},\ldots, B_{k}))$ such that $B=\sum
B_{i}$ is denoted by $\mathcal{P}(x,y;\mathscr{D}, B)$. Sometimes we
will omit $\mathscr{D}$ from the notation. We also put $\delta(x,y;B)
= |x|-|y|-1+\mu(B)$ which is the virtual dimension of
$\mathcal{P}(x,y;\mathscr{D};B)$.

Fix a $\widetilde{\Lambda}^{+}$--algebra $\mathcal{R}$ with its
structural morphism $q:\widetilde{\Lambda}^{+} \to \mathcal{R}$.  The
differential on the pearl complex $\mathcal{C}(\mathscr{D})$
(mentioned at the beginning of~\S\ref{sb:qh}) is defined as follows.
For $x \in \textnormal{Crit}(f)$:
\begin{equation} \label{eq:diff-pearl}
   dx=\sum_{\scriptscriptstyle y; \, |y|=|x|-1} 
   \# \mathcal{P}(x,y; \mathscr{D}; 0)\,y \;\; +
   \sum_{\substack{\scriptscriptstyle y, B \neq 0 ; \\ \scriptscriptstyle
       \delta(x,y;B)=0}} (-1)^{|y|}\,\# \mathcal{P}(x,y;\mathscr{D}; B)\,
   y\,q(T^B). 
\end{equation}
Notice that the first summand coincides with the Morse differential.
Note also the $(-1)^{|y|}$ sign standing in front of the elements in
the second summand. This sign is needed in order to make $d$ be a
differential (i.e. $d^2=0$) and is implied by our sign conventions for
the moduli spaces. See Remark~\ref{r:stable-spheres} for more on that.

Showing that $d^2=0$ reduces to the verifications in the
$\mathbb{Z}_2$ case as described in \cite{Bi-Co:rigidity} together
with two points having to do with the orientation conventions. The
first concerns the coherence of the orientation conventions with
respect to bubbling and, respectively, with respect to the contraction
of a flow line joining two consecutive disks. The claim in this case
is that a configuration that appears with a certain sign by bubbling,
also appears by the contraction of a flow line but with a reversed
sign. The second has to do with the signs that appear at the breaking
of a $1$-dimensional pearl moduli space at a critical point of $f$: we
need to make sure that these signs are the correct ones so that
$d^{2}=0$.  We now intend to explain why our conventions take care of
these two points.

For the first point, let us analyze the boundary points of a
$1$-dimensional moduli space of pearly trajectories
$\mathcal{P}(x,y;\mathscr{D};(B_1, \ldots, B_k))$ that appear when a
gradient trajectory between the $i$'th disk and the $(i+1)$'th disk
($1 \leq i \leq k-1$) shrinks to zero length. The relevant part of the
fiber product in~\eqref{eq:P-xy-fiber-prod} is the space
$$\mathcal{P}_i=(\mathcal{M}_{2}(B_i,J)\times \mathbb{R}_+)
\times_{L}\mathcal{M}_{2}(B_{i+1},J).$$ Applying
formula~\eqref{eq:shrink} we get
$$\partial_{\scriptscriptstyle{shrink}}\, \overline{\mathcal{P}_i} 
= (-1)^n \, \mathcal{M}_{2}(B_i,J)\times_{L}
\mathcal{M}_{2}(B_{i+1},J).$$ Note that $\dim
\mathcal{M}_{2}(B_i,J) + 1 = n + \mu(B_i) \equiv n \pmod{2}$, since
$\mu(B_i)$ is even because $L$ is orientable. Next, by
formula~\eqref{eq:bndry-disks-2} we have that the component of the
boundary of $\overline{\mathcal{M}_{2}(B_i+B_{i+1},J)}$ that
corresponds to bubbling into two disks of classes $B_i$, $B_{i+1}$ is
$$\partial^{\scriptscriptstyle (B_i,B_{i+1})}_{\scriptscriptstyle 
  \textnormal{bubble}} \, \overline{\mathcal{M}_{2}(B_i+B_{i+1},J)}
= (-1)^{n-1} \mathcal{M}_{2}(B_i,J) \times_L
\mathcal{M}_{2}(B_{i+1},J).$$ Applying the Leibniz formula for fibre
products~\eqref{eq:fiber-Leibniz} it follows that bubbling and
shrinking of a gradient trajectory between two disks come with
opposite signs in boundaries of $1$-dimensional spaces of pearly
trajectories. Now fix $B \neq 0$. Summing this up over all $k \geq 1$
and $(B_1, \ldots, B_k)$ with $\sum B_i = B$ we obtain that
\begin{equation} \label{eq:bubble-shrink}
   \#\partial_{\scriptscriptstyle bubble}
   \mathcal{P}(x,y;\mathscr{D};B) \; + \;
   \#\partial_{\scriptscriptstyle shrink}
   \mathcal{P}(x,y;\mathscr{D};B)=0.
\end{equation}
Of course other bubbles might a priori occur (such as side bubbling,
or sphere bubbles) but they actually do not appear when $L$ is
monotone (see~\cite{Bi-Co:rigidity, Bi-Co:Yasha-fest}). This concludes
the first point in the proof that $d^2=0$.

\ 

We now come to the second point in the proof. By the results
of~\cite{Bi-Co:rigidity, Bi-Co:qrel-long} when the virtual dimension
is $\delta(x,y;B)=1$, the spaces $\mathcal{P}(x,y;\mathscr{D};B)$
admit a compactification into a $1$-dimensional manifold with
boundary. Moreover, the boundary of this compactification consists of
precisely the following three types of spaces:
\begin{equation} \label{eq:full-bndry}
   \partial \overline{\mathcal{P}(x,y;\mathscr{D};B)} = 
   \partial_{\scriptscriptstyle bubble} 
   \mathcal{P}(x,y;\mathscr{D};B) \coprod 
   \partial_{\scriptscriptstyle shrink} 
   \mathcal{P}(x,y;\mathscr{D};B) \coprod
   \partial_{\scriptscriptstyle break} 
   \mathcal{P}(x,y;\mathscr{D};B),
\end{equation}
where $\partial_{\scriptscriptstyle break}$ stands for breaking of a
pearly trajectory at a critical point which we now elaborate more
about. Let $\mathbf{B} = (B _1, \ldots, B_k)$ be such that $\sum B_j =
B$, and consider the space $\mathcal{P} =
\mathcal{P}(x,y;\mathscr{D};\mathbf{B})$. We assume that its dimension
is $1$, namely $\delta(x,y;B)=1$. There are three types of places
where the gradient trajectory might break at.  The first is at a
critical point $x'$ between $x$ and the first disks $B_1$. The second
possibility is at a critical point $z$ between two consecutive disks
$B_i$ and $B_{i+1}$. The last possibility is that this occurs at a
critical point $y'$ between the last disk $B_k$ and the point $y$.
Applying the Leibniz formula~\eqref{eq:fiber-Leibniz} together with
formulae~\eqref{eq:bndry-Wu},
~\eqref{eq:del-broken},~\eqref{eq:bndry-Ws} we obtain:
\begin{equation} \label{eq:breaking}
   \begin{aligned}
      & \partial_{x'} \mathcal{P} =
      m(x,x') \times \mathcal{P}(x',y; \mathscr{D};\mathbf{B}), \\
      & \partial_{z} \mathcal{P} = (-1)^{|x|+1}
      \mathcal{P}(x,z;\mathscr{D};(B_1, \ldots, B_i)) \times
      \mathcal{P}(z,y;\mathscr{D};(B_{i+1}, \ldots, B_k)), \\
      & \partial_{y'} \mathcal{P} =
      -\mathcal{P}(x,y';\mathscr{D};\mathbf{B}) \times m(y',y).
   \end{aligned}
\end{equation}
Recall also that by our conventions $m(x,x') =
\mathcal{P}(x,x';\mathscr{D};0)$ and similarly for $m(y',y)$.  The
union of the spaces in~\eqref{eq:breaking} over all relevant $x'$,
$z$, $y'$, $i$, $k$ and $(B_1, \ldots, B_k)$ with $\sum B_j=B$ form
the space $\partial_{\scriptscriptstyle break}
\mathcal{P}(x,y;\mathscr{D};B)$.

We are now ready to show that $d^2(x) = 0$ for every $x \in
\textnormal{Crit}(f)$. We will work here with the ring
$\widetilde{\Lambda}^+$, which implies that the same statement holds
for every $\widetilde{\Lambda}^+$--algebra. Fix $y \in
\textnormal{Crit}(f)$ and $B \in H_2^D$ so that $\delta(x,y;B)=1$.  We
have to show that the coefficient of $y T^B$ in $d \circ d(x)$, which
we denote by $\langle d^2(x), y T^B \rangle$ is $0$.  Clearly, if
$B=0$ this amounts to showing that the Morse differential squares to
$0$ which is well known, thus we assume that $B \neq 0$. A simple
computation now shows that:
\begin{equation} \label{eq:d^2=0-a}
   \begin{aligned}
      \langle d^2(x), y T^B \rangle = & \sum_{|x'|=|x|-1} (-1)^{|y|}
      \#\mathcal{P}(x,x';0) \#\mathcal{P}(x',y;B) \; + \\
      & \sum_{\substack{z, A; \\ \delta(x,z;A)=0 \\ A \neq
          B}}(-1)^{|z|+|y|} \#\mathcal{P}(x,z;A)\#\mathcal{P}(z,y;B-A)
      \;
      + \\
      & \sum_{y'; \delta(x,y';B)=0} (-1)^{|y'|}\#\mathcal{P}(x,y';B)
      \#\mathcal{P}(y',y;0).
   \end{aligned}
\end{equation}
Applying~\eqref{eq:breaking} we now arrive to:
\begin{equation} \label{eq:d^2=0-b}
   \begin{aligned}
      \langle d^2(x), y T^B \rangle = & \sum_{|x'|=|x|-1} (-1)^{|y|}
      \#
      \partial_{x'}\mathcal{P}(x,y;B) \; + \\
      & \sum_{\substack{z, A; \\
          \delta(x,z;A)=0 \\ A \neq B}}(-1)^{|z|+|y|+|x|+1} \#
      \partial_{z} \mathcal{P}(x,y;B) \; + \\
      & \sum_{y'; \delta(x,y';B)=0} (-1)^{|y'|+1} \#\partial_{y'}
      \mathcal{P}(x,y;B).
   \end{aligned}
\end{equation}
Note that for the $z$'s that appear in the second summand we have
$|z|+|x|+1 \equiv 0 \pmod{2}$, hence $(-1)^{|z|+|y|+|x|+1} =
(-1)^{|y|}$. Similarly, for the third summand we have $(-1)^{|y'|+1} =
(-1)^{|y|}$. Thus we obtain
$$\langle d^2(x), y T^B \rangle = (-1)^{|y|} \,
\#\partial_{\scriptscriptstyle break} \mathcal{P}(x,y;B) = (-1)^{|y|}
\, \#\partial \overline{\mathcal{P}(x,y;B)} = 0,$$ where the pre-last
equality follows from~\eqref{eq:bubble-shrink}
and~\eqref{eq:full-bndry}. This concludes the verification that 
$d^2 = 0$.

\begin{rem} \label{r:stable-spheres} Here we explain in a more
   conceptual way the role of the sign $(-1)^{|y|}$
   in~\eqref{eq:diff-pearl}. This sign naturally appears from slightly
   different moduli spaces than $\mathcal{P}(x,y;\mathscr{D};B)$.  For
   every $x \in \textnormal{Crit}(f)$ denote by $S^u(x)$ the unstable
   sphere corresponding to $x$. This can be thought of as small radius
   (or infinitesimal) sphere inside $W^u(x)$ oriented as the boundary
   of small disk around the critical point which lies inside $W^u(x)$
   (recall that $W^u(x)$ is oriented). Similarly we have the stable
   sphere $S^s(y)$ for every $y \in \textnormal{Crit}(f)$. Consider
   now the following moduli space:
   \begin{equation}\label{eq:P-sph-xy-fiber-prod}
      \begin{aligned}
         \mathcal{P}_{\scriptscriptstyle \textnormal{sph}}(x,y;
         \mathscr{D}; (B_1, \ldots, B_k)) = (S^{u}(x) \times
         \mathbb{R}_+) & \times_L (\mathcal{M}_{2}(B_{1},J)\times
         \mathbb{R}_+) \times_L
         \ldots \\
         \ldots & \times_{L} (\mathcal{M}_{2}(B_{i},J)\times
         \mathbb{R}_+) \times_L \ldots
         \\
         \ldots & \times_{L} (\mathcal{M}_{2}(B_{k},J) \times
         \mathbb{R}_+ )\times_{L} S^{s}(y) \;,
      \end{aligned}
   \end{equation}
   where the first factor is mapped to $L$ by $(p,t) \mapsto
   \Phi_t(p)$, and the last one by inclusion. The only difference
   between $\mathcal{P}$ and $\mathcal{P}_{\scriptscriptstyle
     \textnormal{sph}}$ is the first factor in the fiber product as
   well as the last two ones. For $B \neq 0$, we define
   $\mathcal{P}_{\scriptscriptstyle \textnormal{sph}}(x,y;
   \mathscr{D}; B)$ to be the union of all
   $\mathcal{P}_{\scriptscriptstyle \textnormal{sph}}(x,y;
   \mathscr{D}; (B_1, \ldots, B_k))$ over all $(B_1, \ldots, B_k)$
   with $\sum B_j= B$. When $B=0$ we put
   $\mathcal{P}_{\scriptscriptstyle
     \textnormal{sph}}(x,y;\mathscr{D},0) = (S^u(x) \times
   \mathbb{R}_+) \times_L S^s(y)$.  The relation between these spaces
   and the one we have used so far is give by:
   \begin{equation} \label{eq:P-P-sph}
      \begin{aligned}
         & \mathcal{P}_{\scriptscriptstyle
           \textnormal{sph}}(x,y;\mathscr{D};0) = (-1)^{n+|x|-|y|-1} \,
         \mathcal{P}(x,y;\mathscr{D};0) = (-1)^{n+|x|-|y|-1} \, m(x,y), \\
         & \mathcal{P}_{\scriptscriptstyle
           \textnormal{sph}}(x,y;\mathscr{D};B) =
         (-1)^{n+1+|x|} \, \mathcal{P}(x,y;\mathscr{D};B), \quad
         \textnormal{when } B \neq 0.
      \end{aligned}
   \end{equation}
   In particular, when $\delta(x,y;B)=0$ we have:
   $$\mathcal{P}_{\scriptscriptstyle
     \textnormal{sph}}(x,y;\mathscr{D};0) = (-1)^n m(x,y), \quad
   \mathcal{P}_{\scriptscriptstyle
     \textnormal{sph}}(x,y;\mathscr{D};B) = (-1)^{n+|y|} \,
   \mathcal{P}(x,y;\mathscr{D};B), \; \textnormal{when } B \neq 0.$$
   Thus our differential~\eqref{eq:diff-pearl} can be written also as:
   $$d(x) = (-1)^n \sum_{\substack{y, B; \\ \delta(x,y;B)=0}}
   \# \mathcal{P}_{\scriptscriptstyle
     \textnormal{sph}}(x,y;\mathscr{D};B) y \, T^B.$$ Moreover, the
   spaces $\mathcal{P}_{\scriptscriptstyle \textnormal{sph}}$ behave
   better with respect to breaking at critical points, at least as far
   as orientations go. In fact, if $\delta(x,y;B)=1$ we have:
   $$\partial_{\scriptscriptstyle \textnormal{break}}
   \bigl( \mathcal{P}_{\scriptscriptstyle \textnormal{sph}}(x,y;
   \mathscr{D};
   B) \bigr) \coprod_{\substack{v \in \textnormal{Crit}(f),\, B'+B''=B;\\
       \delta(x,v;B')=0}} (-1)^{n+1} \,
   \mathcal{P}_{\scriptscriptstyle \textnormal{sph}}(x,v; \mathscr{D};
   B') \times \mathcal{P}_{\scriptscriptstyle \textnormal{sph}}(v,y;
   \mathscr{D}; B'').$$ This together with~\eqref{eq:bubble-shrink}
   immediately implies that $d^2=0$.

   Although the spaces $\mathcal{P}_{\scriptscriptstyle
     \textnormal{sph}}$ seem more natural from the point of view of
   orientations we have chosen not to explicitly work with them. One
   reason is that they seem less convenient for the purpose of the
   other quantum operations (e.g. the quantum product). Another
   drawback is that one has to redefine these spaces in some
   situations, e.g. when $x$ is a minimum the unstable sphere $S^u(x)$
   is, naively speaking, void.  Another case is when the holomorphic
   disks in $\mathcal{M}_2(B_1)$ come closer to the point $x$ than
   $S^u(x)$ (or even touch that point).
  \end{rem}

\subsubsection{Orientations for the quantum product}
\label{sbsb:or-prod}
The various operations described earlier in this section are modeled
on trees with nodes of valence at most four. In other, words they
correspond to strings of pearls that possibly meet a disk with at most
three entries and one exit.

As an example we now focus on the quantum product
(see~\cite{Bi-Co:rigidity, Bi-Co:qrel-long} for a complete definition
of the product). Fix three Morse functions $f,f',f''$ and the pearl
data $\mathscr{D}=(f, (\cdot, \cdot),J)$, $\mathscr{D}'=(f', (\cdot,
\cdot)', J)$, $\mathscr{D}''=(f'', (\cdot, \cdot)'', J)$. Let $v\in
\Crit(f)$, $w\in\Crit(f')$, $y\in \Crit(f'')$.  The coefficient of $y$
in the product $v\ast w$ is the sum over all classes $B,B',B'',\la\in
H^{D}_{2}$ of the number of configurations in the moduli space
$\mathcal{P}(v,w,y;B,B',B'',\la)$ given as an iterated fiber product
that we now make explicit.

Given data $\mathscr{D}=(f,(\cdot, \cdot),J)$, $x\in \Crit(f)$ and
$(B_1, \ldots, B_k)$ with $B_i \neq 0$ we first define the {\em
  unstable} pearl moduli space $\mathcal{P}^{u}(x;\mathscr{D};(B_1,
\ldots, B_k))$ to be the following iterated fiber product (together
with its orientation):
$$W^{u}(x)\times _{L} 
(\mathcal{M}_{2}(B_{1},J)\times \mathbb{R}_+)\times_L \ldots
\times_{L} (\mathcal{M}_{2}(B_{i},J)\times \mathbb{R}_+) \times_L
\ldots \times_{L}(\mathcal{M}_{2}(B_{k},J)\times \mathbb{R}_+)~.~$$
Given $B \neq 0$ we denote by $\mathcal{P}^{u}(x;\mathscr{D};B)$ the
union of all $\mathcal{P}^{u}(x;\mathscr{D}; (B_{1}, B_{2},\ldots,
B_{k}))$ with $\sum B_i = B$. In case $B=0$ we just put
$\mathcal{P}^{u}(x;\mathscr{D};0)=W^{u}(x)$ (again, as oriented
manifolds). This is similar to~\eqref{eq:P-xy-fiber-prod} with the
exception that the last fiber product is missing here and is replaced
by the term $\mathbb{R}_+$. The space
$\mathcal{P}^{u}(x;\mathscr{D};B)$ comes with an evaluation map
$$e^{u}_{B}:\mathcal{P}^{u}(x;\mathscr{D};B) \to L$$ whose restriction 
to $\mathcal{P}^{u}(x;\mathscr{D}; (B_{1}, B_{2},\ldots, B_{k}))$ is
induced from the evaluation on the last factor
$$\mathcal{M}_{2}(B_k,J) \times \mathbb{R}_+ \longrightarrow L, \quad (u,t)
\longmapsto \Phi_t(u(+1)).$$ For $B=0$ we take this evaluation map to
be the inclusion $W^u(x) \longrightarrow L$.

Similarly, we define the moduli space
$\mathcal{P}^{s}(x;\mathscr{D};B)$ whose components are defined when
$B \neq 0$ by the fiber product:
$$(L\times
\mathbb{R}_+)\times_{L}(\mathcal{M}_{2}(B_{1},J)\times \mathbb{R}_+)
\times_L \ldots \times_{L} (\mathcal{M}_{2}(B_{i},J)\times
\mathbb{R}_+)\times_L\ldots
\times_{L}(\mathcal{M}_{2}(B_{k},J)\times_{L} W^{s}(x),$$ and
$\mathcal{P}^{s}(x;\mathscr{D},0)=W^{s}(x)$ when $B=0$. There is also
an evaluation map $e^{s}_{B}:\mathcal{P}^{s}(x;\mathscr{D};B)\to L$
whose restriction to the component written above is induced from the
identity $L \to L$ defined on the leftmost term in the product.

Next, consider the parametrized moduli space
$\widetilde{\mathcal{M}}(\la,J)$ together with the following three
evaluation maps $e_{\scriptscriptstyle
  \zeta_j}:\widetilde{\mathcal{M}}(\la,J) \longrightarrow L$ where
$\zeta_j = e^{-2j \pi i/3}$, $j=1,2,3$. Finally, we define the space
$\mathcal{P}(v,w,y;B,B',B'',\la)$ by the fiber product:
\begin{equation}\label{eq:orient-prod}
   \mathcal{P}^{u}(v;\mathscr{D};B) 
   \lrsub{e^u_B}{\times}{e'_{\scriptscriptstyle \zeta_1}}
   \Bigl( \mathcal{P}^{u}(w;\mathscr{D}';B')
   \lrsub{e^{u}_{B'}}{\times}{e_{\scriptscriptstyle \zeta_2}}
   \widetilde{\mathcal{M}}(\la,J)
   \lrsub{e_{\scriptscriptstyle \zeta_3}}{\times}{e^{s}_{B''}}
   \mathcal{P}^{s}(y;\mathscr{D''};B'') \Bigr),
\end{equation}
where in this formula the map $e'_{\scriptscriptstyle \zeta_1}$ is
induced on the fiber product in the brackets by the evaluation
$e_{\scriptscriptstyle \zeta_1}$ originally defined on
$\widetilde{\mathcal{M}}(\lambda,J)$. Note that the dimension of
$\mathcal{P}(v,w,y;B,B',B'',\la)$ is $|v|+|w|-|y|-n +
\mu(B)+\mu(B')+\mu(B'')+ \mu(\lambda)$.

For $y \in \textnormal{Crit}(f)$, $C \in H_2^D$ such that
$|y|-\mu(C)=|v|+|w|-n$. The coefficient of $y T^C$ in the product
$v*w$ is given by 
$$\sum_{B+B'+B''+\lambda = C} \#\mathcal{P}(v,w,y;B,B',B'',\la).$$
By using similar arguments as those used above in the verification
showing $d^{2}=0$ (see~\S\ref{sbsb:orient-pearl}), it is easy to see
that the product defined by these moduli spaces verifies the Leibniz
formula~\eqref{eq:leibniz-prod} and, moreover, that the classical term
in this definition coincides with the Morse intersection product
(on the chain level). Furthermore similar arguments show that the
induced product on homology makes $QH(L; \mathcal{R})$ a unital
associative ring.

\subsubsection{Orientations for the quantum module structure}
\label{sbsb:or-mod} Similar conventions are used to define the
orientations required for the module structure
from~\S\ref{subsubsec:module}. Explicitly, let $h:M\to \R$ be a Morse
function and fix a metric $(\cdot, \cdot)_M$ on $M$ so that the pair
$(h,(\cdot, \cdot)_M)$ is Morse-Smale. Fix a pearl data on $L$,
$\mathscr{D}=(f, (\cdot, \cdot), J)$. Let $a\in \Crit(h)$ and $x\in
\Crit(f)$. Let $y \in \textnormal{Crit}(f)$ and $C \in H_2^D$ with
$|y|-\mu(C) = |a|+|x|-2n$. The coefficient of $y T^C$ in the product
$a\ast x$ is given by counting elements in moduli spaces of the form:
\begin{equation}\label{eq:orient-module}
   W^{u}(a) \lrsub{i}{\times}{e'_0} 
   \Bigl( \mathcal{P}^{u}(x;\mathscr{D}';B')
   \lrsub{e^{u}_{B'}}{\times}{e_{-1}}
   \widetilde{\mathcal{M}}(\la,J)
   \lrsub{e_{+1}}{\times}{e^{s}_{B''}}
   \mathcal{P}^{s}(y;\mathscr{D''};B'') \Bigr),
\end{equation}
for all $B',B'', \lambda$ with $B'+B''+\lambda=C$. Here $i:W^{u}(a)
\longrightarrow M$ is the inclusion, $e'_0$ is the map induced from
$e_0:\widetilde{M}(\lambda,J) \longrightarrow M$, $e_0(u) = u(0)$, and
$e_{\pm 1}: \widetilde{M}(\lambda,J) \longrightarrow L$ are the
evaluation maps $e_{\pm 1}(u) = u(\pm 1)$.

Proving that this operation induces on $QH(L;\mathcal{R})$ a structure
of a module over $QH(M;\mathcal{R})$ is based on arguments similar to the
ones in~\S\ref{sbsb:orient-pearl}.

\subsubsection{Orientations for the quantum inclusion}
\label{sbsb:or-inc} Here we fix our conventions for the quantum
inclusion $i_{L}:QH(L;\mathcal{R})\to QH(M;\mathcal{R})$ which has
been recalled in~\S\ref{subsubsec:inclu}. The basic data is similar
here as in the case of the module multiplication: besides
$\mathscr{D}$ we also fix the Morse-Smale pair $(h,(\cdot, \cdot)_M)$
on $M$. We fix $x\in \Crit(f)$. Let $a \in \textnormal{Crit}(h)$ and
$B \in H_2^D$ with $|a|-\mu(B) = |x|$. The coefficient of $a T^B$ in
the expression of $i_{L}(x)$ is given by counting elements in moduli
spaces of the form:
\begin{equation}\label{eq:orient-inclusion}
   \mathcal{P}^{u}(x;\mathscr{D}; B') \lrsub{e^{u}_{B'}}{\times}{e_{-1}}
   \widetilde{\mathcal{M}}(\la,J) \lrsub{e_0}{\times}{i} W^{s}(a),
\end{equation}
where $i:W^{s}(a)\to M$ is the inclusion. It is not difficult to see
that with our conventions this defines a chain morphism whose
classical part coincides with the usual Morse (or singular homology)
inclusion.

\subsubsection{Invariance of the structures} \label{sbsb:invariance}
We now shortly discus the proof of the invariance of all these
structures with respect to changes in the data $\mathscr{D}$. This is
based on constructing comparison chain maps associated to any two
pairs of data. In turn, to construct such comparison maps there are
two distinct methods each perfectly similar to those described over
$\Z_{2}$ as in~\S3.2-e of~\cite{Bi-Co:rigidity} and, respectively, in
the proof of~Proposition~4.4.1 in the same paper. The first method is
based on a cone construction naturally appearing in a pearly version
of Morse cobordisms. It provides a quasi-isomorphism (canonical upto
chain homotopy) $\Psi_{\mathscr{D}',
  \mathscr{D}}:\mathcal{C}(L;\mathscr{D};\mathcal{R}) \longrightarrow
\mathcal{C}(L;\mathscr{D}';\mathcal{R})$ for {\em any two} tuples of
data $\mathscr{D}$, $\mathscr{D}'$. In view of the previous
subsections, the right convention to orient the relevant moduli spaces
in this case is rather straightforward and we omit the details.

The second methods is less general in the sense that it allows to
compare the pearl complexes associated only to two tuples
$\mathscr{D}$, $\mathscr{D}'$ having the same almost complex structure
$J$ and moreover the two Morse functions should be mutually in general
position. The resulting chain map $\phi_{\mathscr{D}', \mathscr{D}}$
coincides in homology with the one provided by the general method
$\Psi_{\mathscr{D}',\mathscr{D}}$.  As we use explicitely in the paper only 
the second construction we indicate briefly
the orientation conventions in that case.  Let $\mathscr{D}=(f,(\cdot,
\cdot), J)$ and $\mathscr{D}'=(f',(\cdot, \cdot), J)$ with $f$ and
$f'$ in general position. The map
$\phi_{\mathscr{D}',\mathscr{D}}:\mathcal{C}(L;\mathscr{D})\to
\mathcal{C}(L;\mathscr{D}')$ is defined by counting elements in the
moduli spaces of the form:
$$\Phi(x,y;\mathscr{D},B;\mathscr{D}',B') = 
\mathcal{P}^{u}(x;\mathscr{D};B)
\times_{L}\mathcal{P}^{s}(y;\mathscr{D}';B')~.~$$ The evaluation maps
here are the obvious ones.  The chain map
$\phi_{\mathscr{D}',\mathscr{D}}$ is now defined by
$$\phi_{\mathscr{D}',\mathscr{D}} (x) = 
\sum_{\substack{y,B,B'; \\ |y|=|x|+\mu(B+B')}}
\#\Phi(x,y;\mathscr{D},B;\mathscr{D}',B') y T^{B+B'}.$$ By the same
type of arguments as above, it is easy to see that this definition
provides a chain map that induces an isomorphism in homology and that
this definition provides the usual Morse comparison map in the
classical case.

\subsubsection{Orientation conventions for duality}
\label{sbsb:or-duality}
This is a topic that has been discussed in~\S4.4
of~\cite{Bi-Co:rigidity} but only over $\mathbb{Z}_2$ hence in the
absence of orientations. We fix a ground ring $K$ (it will be here a
field or $\mathbb{Z}$). We now recall some notation from
\cite{Bi-Co:rigidity} and adapt it to the present setting.

Assume that $\mathcal{R}$ is a commutative
$\widetilde{\La}^{+}$--algebra and suppose that
$(\mathcal{C},\partial)$ is a free $\mathcal{R}$-chain complex
(see~\cite{Bi-Co:rigidity}~\S2.2.1 for the precise definition). Thus
$\mathcal{C}=\mathcal{R}\otimes G$ for some graded free $K$-module
$G$. To the chain complex $(\mathcal{C},\partial)$ we associate the
following two closely related complexes:
\begin{itemize}
  \item[a.] $(\mathcal{C}^{\odot},\partial^{\ast})$;
   $\mathcal{C}^{\odot}=\hom_{\mathcal{R}}(\mathcal{C},\mathcal{R})$
   with the grading given by $g\in \mathcal{C}^{\odot}$, $|g|=k$ if
   $g(\mathcal{C}_{i})\subset \mathcal{R}_{i+k}$. The differential
   $\partial^{\ast}$ is given by
   \begin{equation}\label{eq:signs-adjoint-diff}
      \langle\partial^{\ast} g, x\rangle  = -(-1)^{|g|}\langle g,\partial x\rangle
      ~.~
   \end{equation}
   Clearly, $\mathcal{C}^{\odot}$ is a chain complex and 
   we have an isomorphism of graded
   modules $\mathcal{C}^{\odot}\cong \mathcal{R}\otimes \hom_{K}(G,K)$.
  \item[b.] $(\mathcal{C}^{\ast},\partial^{\ast})$;
   $\mathcal{C}^{q}=\mathcal{C}^{\odot}_{-q}$ and the differential
   $\partial^{\ast}$ coincides with the differential of
   $\mathcal{C}^{\odot}$ but it has now degree $+1$ so that
   $\mathcal{C}^{\ast}$ is a cochain complex. The cohomology of
   $\mathcal{C}$ is by definition
   $H^{k}(\mathcal{C})=H^{k}(\mathcal{C}^{\ast})$.  Notice that
   $\mathcal{C}^{\ast}=\mathcal{R}^{inv}\otimes \hom_{K}(G,K)^{inv}$
   where for a graded vector space $A$, $A^{inv}$ is the graded vector
   space so that for $a\in A^{inv}$, $|a|=-deg_{A}(a)$.
\end{itemize}

\begin{rem}
   \begin{itemize}
     \item[a.]  The identification between chain complexes $(C_{k},
      d_{k})$ and cochain complexes $(C^{k}, d^{k})$, $C^{k}=C_{-k}$,
      $d^{k}=d_{-k}$ that appears at point b. is standard in
      homological algebra but we have preferred to make it explicit
      here by means of the functor $(-)^{inv}$.
     \item[b.] The sign that appears in the definition of
      $\partial^{\ast}$ in formula~\eqref{eq:signs-adjoint-diff} is
      the only addition to the notation in~\S4.4 from
      \cite{Bi-Co:rigidity} (where we worked over $\Z_{2}$). This sign
      appears in other situations in algebraic topology as well. For
      instance, let $(S_{\bullet}X,\delta)$ be the standard singular
      chain complex of a space $X$. In the definition of the singular
      cohomology of $X$ the literature contains essentially two
      variants for the differential: one is the adjoint of $\delta$,
      without the signs in~\eqref{eq:signs-adjoint-diff}, and the
      other is given by formula~\eqref{eq:signs-adjoint-diff}. Many
      authors use the signed formula at least as soon as they deal
      with products and duality (see for
      instance~\cite{Dold:alg-top}). The advantage of this formula is
      that the Kronecker pairing $S^{\bullet}(N)\otimes
      S_{\bullet}(N)\to \Z$ is a chain map. If $X$ is an oriented
      manifold and once Poincar\'e duality is defined by $(-)\cap
      [X]$, the intersection product is the dual of the cup-product,
      and both the intersection product and the cup-product verify the
      respective Leibniz formulas with the usual signs. These are the
      conventions concerning classical algebraic topology that we also
      use in this paper.
     \item[c.]  Clearly, in our situation equation
      (\ref{eq:signs-adjoint-diff}) insures that the pairing
      $\mathcal{C}^{\odot}\otimes_{\mathcal{R}} \mathcal{C}\to
      \mathcal{R}$ is a chain map (where the differential on
      $\mathcal{R}$ is trivial).
   \end{itemize}
\end{rem}

For a complex $\mathcal{C}$ we denote by $s^{n}\mathcal{C}$ its
$n$-fold suspension: this coincides with $\mathcal{C}$ but is graded
so that the degree of $x$ in $s^{n}\mathcal{C}$ is
$n+deg_{\mathcal{C}}(x)$, in other words $(s^n \mathcal{C})_k =
\mathcal{C}_{k-n}$. The differential on this complex remains the same.
In particular, $H_{k}(s^{n}\mathcal{C}^{\odot})\cong
H_{k-n}(\mathcal{C}^{\odot})=H^{n-k}(\mathcal{C}^{\ast})$.

With the conventions above the results stated in~Proposition 4.4.1
in~\cite{Bi-Co:rigidity} remain true. Namely, there exists a degree
preserving morphism of chain complexes:
\begin{equation}\label{eq:duality}
   \eta:\mathcal{C}(L;\mathscr{D})\longrightarrow s^{n}
   (\mathcal{C}(L;\mathscr{D})^{\odot})
\end{equation}
that induces an isomorphism in homology. Thus there is an induced
isomorphism $$QH_{k}(L;\mathcal{R})\cong QH^{n-k}(L;\mathcal{R})~.~$$

The proof of this fact is basically the same as that of
Proposition~4.4.1 in~\cite{Bi-Co:rigidity}: $\eta$ is written as the
composition of two morphism, each of them inducing an isomorphism in
homology.  The first is the comparison chain morphism between
$\mathcal{C}(L; f, (\cdot, \cdot),J)$ and $\mathcal{C}(L; - f, (\cdot,
\cdot) ,J)$. The second is the following identification:
$$\Theta:\mathcal{C}(L;-f,(\cdot, \cdot),J)\cong 
s^{n}(\mathcal{C}(L;f,(\cdot, \cdot), J)^{\odot})$$ where $\Theta$ is
induced by $$\textnormal{Crit}(-f) \ni x \longmapsto (x')^{\ast}\in
\hom_{K}(K\langle\Crit(f)\rangle,K).$$ Here we have denoted by $x'$
the point $x \in \textnormal{Crit}(-f)$ viewed as critical point of
$f$ and by $(x')^{\ast}$ the dual of $x'$ with respect to the basis
$\{y'\}_{y' \in \textnormal{Crit}(f)}$ of $K \langle
\textnormal{Crit}(f) \rangle$.  The orientations of the stable and
unstable manifolds of $-f$ are related to those for $f$ as follows.
First we orient the stable submanifolds of $-f$ by requiring that
$W^{s}_{-f}(x)=W^{u}_{f}(x)$.  Next, in order to orient the unstable
submanifolds of $-f$ we apply to $-f$ the standard orientation
conventions. Namely we require that $T_{x}W^{s}_{-f}(x)\oplus
T_{x}W^{u}_{-f}(x)=T_{x}(L)$ at each $x\in \Crit(-f)$.

With these conventions one obtains the following identities.  Let $x,y
\in \textnormal{Crit}(-f)$, $B \in H_2^D$ with $\delta(x,y;B)=0$.  Put
$\mathscr{D} = (-f,(\cdot,\cdot), J)$ and $\mathscr{D}' =
(f,(\cdot,\cdot), J)$. Then:
\begin{equation} \label{eq:pearly-space-duality}
   \begin{aligned}
      & \mathcal{P}(x,y;\mathscr{D};0) = -(-1)^{|x'|} \mathcal{P}(y',x';
      \mathscr{D}';0), \\
      & \mathcal{P}(x,y;\mathscr{D};B) = (-1)^{|y|+1}
      \mathcal{P}(y',x';\mathscr{D}';B).
   \end{aligned}
\end{equation}
From this it follows that the coefficient of $y$ in $d(x)$ satisfies:
$d'(x)|_y = -(-1)^{|x'|} d(y')|_{x'}$. This immediately implies that
$\Theta$ is a chain map.

Finally, with these conventions it is not difficult to verify in the
present context the following formula:
\begin{equation}\label{eq:duality-augmentation}
   \langle PD(h), x\rangle = \epsilon(h\ast x), 
   \ \forall h\in H_{\ast}(M,K),\  x\in QH(L) ~.~
\end{equation} 
This formula appeared in~\cite{Bi-Co:rigidity}, as ``formula~(6)'' in
the point~iii of Theorem~A in that paper (which was proved there only
over $\Z_{2}$). More specifically: Here $\langle\, , \rangle $ is the
Kronecker product, $- \ast -$ is the module operation discussed
in~\S\ref{subsubsec:module} and $\epsilon$ is the augmentation defined
in~\cite{Bi-Co:rigidity}. Recall that for a pearl complex
$\mathcal{C}(L;f,(\cdot, \cdot),J)$ where $f$ has a unique minimum,
the augmentation $\epsilon$ is induced by the map that sends the
minimum to $1\in\mathcal{R}$ (and sends all the other critical points
to $0$).
